\documentclass[11pt]{article}
\usepackage{amsmath,amsthm,amssymb}
\usepackage[usenames,dvipsnames]{xcolor}
\usepackage{enumerate}
\usepackage{graphicx}
\usepackage{cite}
\usepackage{comment}
\usepackage{oands}
\usepackage{tikz}
\usepackage{changepage}
\usepackage{bbm}
\usepackage{mathtools}
\usepackage[margin=1in]{geometry}

\usepackage[pagewise,mathlines]{lineno}
\usepackage{appendix}
\usepackage{stmaryrd} 
\usepackage{multicol}
\usepackage{microtype}
\usepackage[colorinlistoftodos,textsize=footnotesize]{todonotes}

\definecolor{dgreen}{RGB}{0,180,0}

\usepackage[pdftitle={The Tutte embedding of the Poisson-Voronoi tessellation of the Brownian disk converges to Liouville quantum gravity},
  pdfauthor={Ewain Gwynne, Jason Miller, and Scott Sheffield},
colorlinks=true,linkcolor=NavyBlue,urlcolor=RoyalBlue,citecolor=PineGreen,bookmarks=true,bookmarksopen=true,bookmarksopenlevel=2,unicode=true,linktocpage]{hyperref}

\theoremstyle{plain}
\newtheorem{thm}{Theorem}[section]

\newtheorem{lem}[thm]{Lemma}
\newtheorem{prop}[thm]{Proposition}

\def\@rst #1 #2other{#1}
\newcommand\MR[1]{\relax\ifhmode\unskip\spacefactor3000 \space\fi
  \MRhref{\expandafter\@rst #1 other}{#1}}
\newcommand{\MRhref}[2]{\href{http://www.ams.org/mathscinet-getitem?mr=#1}{MR#2}}

\theoremstyle{definition}
\newtheorem{defn}[thm]{Definition}
\newtheorem{remark}[thm]{Remark}
\newtheorem{prob}[thm]{Problem}

\numberwithin{equation}{section}

\newcommand{\dsb}{\begin{adjustwidth}{2.5em}{0pt}
\begin{footnotesize}}
\newcommand{\dse}{\end{footnotesize}
\end{adjustwidth}}

\newcommand{\ssb}{\begin{adjustwidth}{2.5em}{0pt}}
\newcommand{\sse}{\end{adjustwidth}}

\newcommand{\aryb}{\begin{eqnarray*}}
\newcommand{\arye}{\end{eqnarray*}}
\def\alb#1\ale{\begin{align*}#1\end{align*}}
\def\allb#1\alle{\begin{align}#1\end{align}}
\newcommand{\eqb}{\begin{equation}}
\newcommand{\eqe}{\end{equation}}
\newcommand{\eqbn}{\begin{equation*}}
\newcommand{\eqen}{\end{equation*}}

\newcommand{\BB}{\mathbbm}
\newcommand{\ol}{\overline}
\newcommand{\ul}{\underline}
\newcommand{\op}{\operatorname}

\newcommand{\frk}{\mathfrak}
\newcommand{\eqD}{\overset{d}{=}}
\newcommand{\ep}{\epsilon}
\newcommand{\rta}{\rightarrow}

\newcommand{\wt}{\widetilde}
\newcommand{\wh}{\widehat} 
\newcommand{\mcl}{\mathcal}

\newcommand{\bdy}{\partial}

\newcommand{\tr}{{\mathrm{tr}}}

\newcommand{\SLE}{\mathrm{SLE}}

\let\originalleft\left
\let\originalright\right
\renewcommand{\left}{\mathopen{}\mathclose\bgroup\originalleft}
\renewcommand{\right}{\aftergroup\egroup\originalright}

\title{The Tutte embedding of the Poisson-Voronoi tessellation of the Brownian disk converges to $\sqrt{8/3}$-Liouville quantum gravity}

\date{  }
\author{
\begin{tabular}{c} Ewain Gwynne\\[-5pt]\small Cambridge \end{tabular}
\begin{tabular}{c} Jason Miller\\[-5pt]\small Cambridge \end{tabular}
\begin{tabular}{c} Scott Sheffield\\[-5pt]\small MIT \end{tabular}
}

\begin{document}

\maketitle

\begin{abstract}
Recent works have shown that an instance of a Brownian surface (such as the Brownian map or Brownian disk) a.s.\ has a canonical conformal structure under which it is equivalent to a $\sqrt{8/3}$-Liouville quantum gravity (LQG) surface.  In particular, Brownian motion on a Brownian surface is well-defined.  The construction in these works is indirect, however, and leaves open a basic question: is Brownian motion on a Brownian surface the limit of simple random walk on increasingly fine discretizations of that surface, the way Brownian motion on $\mathbb R^2$ is the $\epsilon \to 0$ limit of simple random walk on $\epsilon \mathbb Z^2$?

We answer this question affirmatively by showing that Brownian motion on a Brownian surface
is (up to time change) the $\lambda \to \infty$ limit of simple random walk
on the Voronoi tessellation induced by 
a Poisson point process whose intensity is $\lambda$ times the associated area measure.  
Among other things, this implies that as $\lambda \to \infty$ the
{\em Tutte embedding} (a.k.a.\ {\em harmonic embedding}) of the discretized Brownian disk converges to the canonical conformal embedding of the continuum Brownian disk, which in turn corresponds to $\sqrt{8/3}$-LQG. 
 
Along the way, we obtain other independently interesting facts about conformal embeddings of Brownian surfaces, including information about the Euclidean shapes of embedded metric balls and Voronoi cells. For example, we derive moment estimates that imply, in a certain precise sense, that these shapes are unlikely to be very long and thin.
\end{abstract}

\tableofcontents

\section{Main results}
\label{sec-intro}

\subsection{Overview}
\label{sec-overview}

This paper concerns relationships between several different topics, including Liouville quantum gravity and the Brownian map.  Let us begin by briefly reviewing the objects under consideration.

A planar map is a graph together with an embedding into the plane so that no two edges cross.  Two planar maps are considered to be equivalent if they differ by an orientation preserving homeomorphism of the plane. The study of planar maps goes back to work of Tutte \cite{tutte} and Mullin \cite{mullin-maps} from the 1960s.  A planar map can be viewed as a metric measure space by equipping it with the graph distance and assigning each vertex one unit of mass.  In recent years, there has been considerable progress in studying the large scale metric behavior of planar maps chosen uniformly at random.  Of particular relevance to the present article are the scaling limit results which give the convergence of uniformly random planar maps towards a continuous object in the Gromov-Hausdorff-Prokhorov topology.  The first results of this type were due to Le Gall \cite{legall-uniqueness} and Miermont \cite{miermont-brownian-map} which are both focused on random planar maps with the sphere topology.  In this case, the limiting object is a random metric measure space with the topology of the sphere \cite{legall-paulin-tbm,miermont-sphere} called the \emph{Brownian map}, which was first defined (in different forms) in~\cite{marc-mokk-tbm,legall-topological}, building on~\cite{cv-bijection,schaeffer-bijection,cs-superbrownian-excursion}.  The works \cite{legall-uniqueness,miermont-brownian-map} have since been extended to uniformly random planar maps with several other topologies, including the disk \cite{bet-mier-disk,gwynne-miller-simple-quad}, the plane \cite{curien-legall-plane}, and the half-plane \cite{bmr-uihpq,gwynne-miller-uihpq}.  The limiting objects that one obtains are collectively known as \emph{Brownian surfaces}.

The aforementioned limit theorems are concerned with the metric measure space structure of large uniformly random planar maps, but not how they are embedded into the plane.  However, it is an important problem to understand scaling limits for canonical embeddings of random planar maps.  This is motivated in part by a desire to better understand the relationship between statistical mechanics models (e.g., percolation, self-avoiding walks, the Ising model) on random planar maps and their counterparts on planar lattices.  It has long been believed that if the embedding is of a conformal type (e.g., Riemann uniformization, circle packing, or the Tutte embedding we consider here), then the large scale geometry of the statistical mechanics model should be the same as if it were considered on a planar lattice.  This idea underlies the famous KPZ relation~\cite{kpz-scaling,shef-kpz}, which serves to convert critical exponents computed on a random geometry to the corresponding exponents on a deterministic geometry, and has been used numerous times to give predictions for exponents for critical lattice models which were later verified using $\SLE$ techniques (e.g., \cite{duplantier-bm-exponents, lsw-bm-exponents1, lsw-bm-exponents2, lsw-bm-exponents3}).

Liouville quantum gravity (LQG) is another theory of random surfaces which was introduced by Polyakov \cite{polyakov-qg1,polyakov-qg2} in the 1980s in the context of string theory.  To define LQG, one starts with a (form of) the Gaussian free field (GFF) $h$ on a domain $\mcl D$ and then considers the random two-dimensional Riemannian manifold with metric tensor
\begin{equation}
\label{eqn:lqg_def}
e^{\gamma h(z)} (dx^2 + dy^2)
\end{equation}
where $\gamma \in (0,2]$ is a parameter.  This definition does not make rigorous mathematical sense since~$h$ is a distribution and not a function.  Making rigorous sense of various aspects of LQG has been a major topic of research in recent years. Some of this work builds on \cite{shef-kpz}, which constructs the volume form associated with~\eqref{eqn:lqg_def}, which is a random measure $\mu_h$ on $\mcl D$ (see~\cite{kahane,rhodes-vargas-review} for a more general construction of random measures of this type). One can construct various kinds of {\em LQG surfaces} by varying the precise definition of $h$ (what domain it lives on, how boundary conditions are chosen, whether one ``weights'' the law of $h$ in some locally absolutely continuous way). Generally, one writes {\em $\gamma$-LQG} to refer to LQG surfaces with parameter $\gamma$. The special case $\gamma=\sqrt{8/3}$ has long been known to be special: $\sqrt{8/3}$-LQG surfaces, like Brownian surfaces, are related to ``undecorated'' planar maps, and are also called {\em pure} LQG surfaces.

In fact, a recent series of works by the second two authors has shown that the Brownian map and the so-called $\sqrt{8/3}$-LQG sphere are in some sense equivalent (and similar statements can be made for disk, whole-plane, or half-plane variants) \cite{lqg-tbm1,lqg-tbm2,lqg-tbm3}. Although both $\sqrt{8/3}$-LQG spheres and Brownian maps come with a natural measure, they also have additional structure: a Brownian surface has a {\em metric}, and a $\sqrt{8/3}$-LQG surface has a {\em conformal structure} (i.e., an embedding into a flat domain, defined up to conformal automorphism of that domain).  The works \cite{lqg-tbm1,lqg-tbm2,lqg-tbm3} show that each of these objects can be canonically endowed with the other one's structure, and that once this is done the objects agree in law.

We have made the present work as self-contained as possible, which in particular means that the reader is not required to have read the series \cite{lqg-tbm1,lqg-tbm2,lqg-tbm3} relating Brownian and $\sqrt{8/3}$-LQG surfaces, or any papers about quantum Loewner evolution. The results we do need will be recalled and restated.

Very roughly, the argument in \cite{lqg-tbm1,lqg-tbm2,lqg-tbm3} proceeds as follows. First, \cite{lqg-tbm1} uses $h$ to construct a metric $D_h$ on $\mcl D$ using a growth process called \emph{quantum Loewner evolution (QLE)} \cite{qle}.  We will not need to recall the precise construction of $D_h$ or the definition of QLE here. Section~\ref{sec-lqg-metric} contains all of the background on $D_h$ that is necessary to understand this paper. Second, \cite{lqg-tbm2} shows that in the case of the $\sqrt{8/3}$-LQG sphere the corresponding metric measure space agrees in law with the Brownian map.  Thus, by sampling $h$ (which determines the $\sqrt{8/3}$-LQG sphere), and then generating the corresponding Brownian map, one obtains a {\em coupling} of the $\sqrt{8/3}$-LQG sphere and the Brownian map.  Finally, \cite{lqg-tbm3} uses certain ``welding and resampling'' arguments to show that in this coupling, the latter object almost surely determines the former.  In other words, one can recover the $\sqrt{8/3}$-LQG instance (i.e., the distribution $h$) as a measurable function of the corresponding Brownian surface instance.

This measurable function is obtained in a remarkably non-explicit way.  Although the argument tells us that an instance of a Brownian map (or disk, plane, or half-plane) has a canonical {\em conformal structure}, i.e., a canonical embedding into $\BB C$ (or a domain in $\BB C$) defined up to M\"obius transformation, it gives us no way to compute or even approximate that embedding using only information about the ``metric measure space'' structure of the Brownian surface.  Similarly, since a Brownian surface has a canonical conformal structure, we know that Brownian motion on a Brownian surface is well-defined, at least modulo a monotone reparameterization of time.\footnote{We will not consider the time parameterization in this paper, but we note that there is a canonical way to parameterize Brownian motion on a $\sqrt{8/3}$-LQG surface, called \emph{Liouville Brownian motion}~\cite{berestycki-lbm,grv-lbm}. We expect, but do not prove, that Liouville Brownian motion is the scaling limit of the parameterized walk in the setting of Theorem~\ref{thm-rw-conv0}; see Problem~\ref{prob-lbm}.}
But the argument tells us nothing about how to construct that Brownian motion: in particular, it does not tell us whether Brownian motion on a Brownian surface is the limit of simple random walk on some natural increasingly fine discretizations of that surface, the way Brownian motion on $\mathbb R^2$ is the $\epsilon \to 0$ limit of simple random walk on $\epsilon \mathbb Z^2$. 

The purpose of the present work is to construct the conformal structure of a Brownian surface in an {\em explicit} manner. To this end, we will start with an instance of a Brownian surface and then approximate it with the graph of cells associated with the Poisson-Voronoi tessellation.  More precisely, we will fix $\lambda > 0$ and then pick a Poisson point process with intensity measure given by $\lambda$ times the area measure on the Brownian surface.  The {\em cell} corresponding to a given point $x$ in the Poisson point process consists of those points that are at least as close to $x$ as they are to any other point of the Poisson point process (w.r.t.\ the metric on the Brownian surface).  Cells are considered to be adjacent if they have non-empty intersection.

Our main result (Theorem~\ref{thm-rw-conv0}) says that as $\lambda \rta \infty$, the random walk on the adjacency graph of Voronoi cells converges modulo parameterization to a limiting continuous path.  Moreover, this path is a Brownian motion (modulo parameterization) when it is embedded into $\BB C$ using the identification of Brownian surfaces with $\sqrt{8/3}$-LQG.  This gives an intrinsic way of constructing Brownian motion (modulo time parameterization) on a Brownian surface without any reference to LQG: indeed, the Brownian motion on the surface is simply the $\lambda\rta\infty$ limit of the random walk on the Voronoi cells.  

One can define the Tutte embedding of the adjacency graph of cells in terms of hitting probabilities for the simple random walk (in the same way that Riemann uniformization can be defined using the hitting probabilities for Brownian motion, i.e., harmonic measure).  Our results show that this Tutte embedding converges to $\sqrt{8/3}$-LQG in an appropriate sense as $\lambda\rta\infty$ (Theorem~\ref{thm-tutte-conv}).  This in particular gives a new, more explicit, proof of the main result of~\cite{lqg-tbm3}, which says that a $\sqrt{8/3}$-LQG surface is a.s.\ determined by its metric measure space structure.

Finally, we remark on some related works. Voronoi tessellations of Brownian surfaces have also been considered in other contexts: see, e.g.,~\cite{chapuy-voronoi-cells,guitter-voronoi-cells}. These papers raise interesting questions about the law of the partitioning of volume among the Voronoi cells, but they do not consider the adjacency graph of cells as we do here. 

The first two authors in~\cite{gm-uniqueness}, building on~\cite{dddf-lfpp,lqg-metric-estimates,local-metrics,gm-confluence}, recently constructed a metric on a $\gamma$-LQG surface for general $\gamma \in (0,2)$ using a completely different construction from the one in~\cite{lqg-tbm1,lqg-tbm2,lqg-tbm3}. 
It was shown in~\cite{gm-uniqueness} that the two constructions give the same metric. 
This paper will make no use of~\cite{gm-uniqueness}, but many of our estimates for the $\sqrt{8/3}$-LQG metric also work for general $\gamma \in (0,2)$; see the discussion just after Theorem~\ref{thm-tutte-conv} and also Remark~\ref{remark-general-gamma}.
\medskip

\noindent\textbf{Acknowledgements.} 
We thank two anonymous referees for helpful comments on an earlier version of this article.
E.G.\ was supported by a Herchel Smith fellowship and a Trinity College junior research fellowship.
S.S.\ was partially supported by NSF grants DMS-1712862 and DMS-1209044 and a Simons Fellowship with award number 306120.

\begin{figure}[ht!]
\begin{center}
\includegraphics[width=0.7\textwidth]{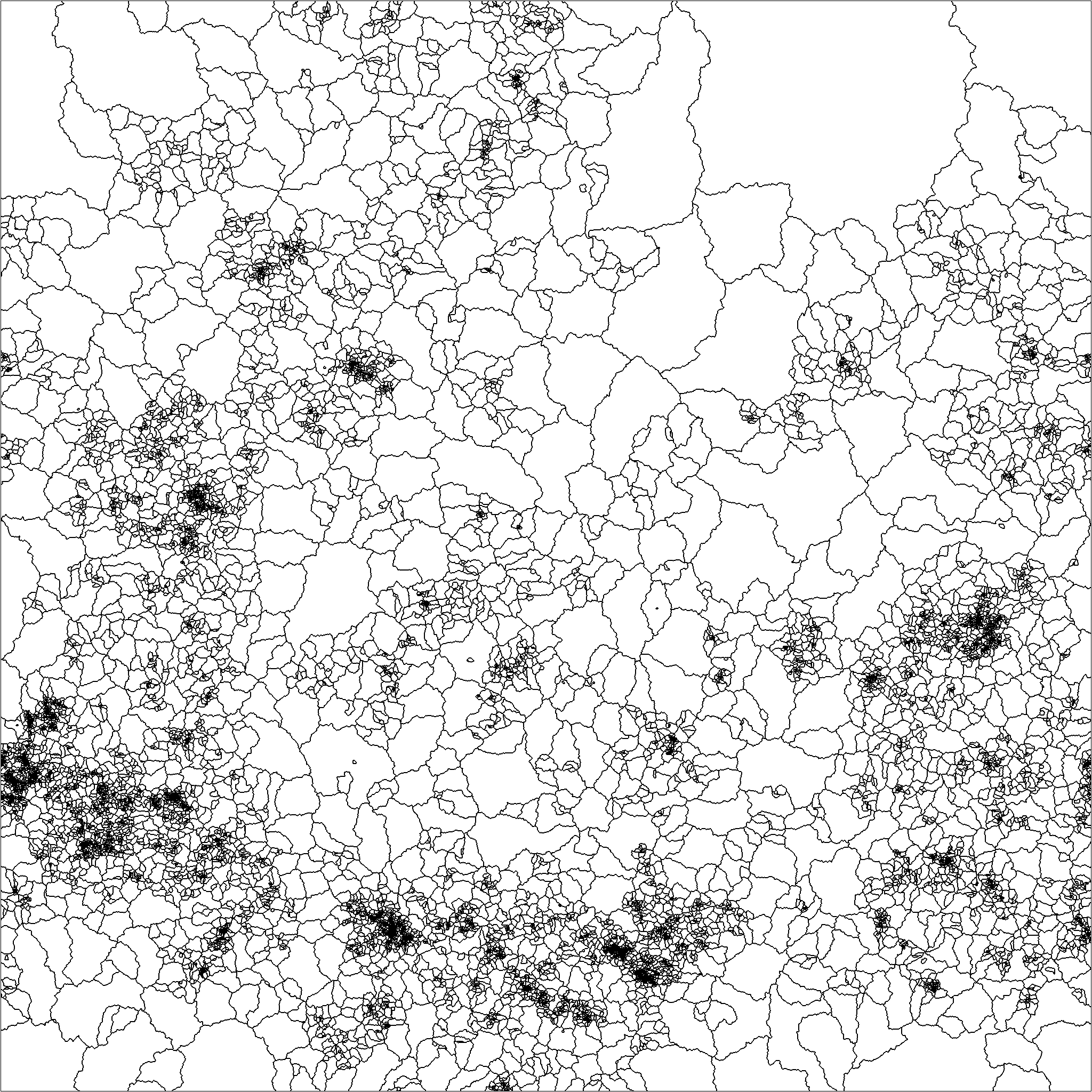}	
\end{center}
\caption{\label{fig:pv_sim} Simulation of the Poisson-Voronoi tessellation of $\sqrt{8/3}$-LQG. Although the cells differ greatly in Euclidean size, they appear to have moderate ``length-to-width ratios'' in the sense that a typical cell contains a Euclidean disk of diameter comparable to its own diameter.  A mathematical version of this observation (see Propositions~\ref{prop-lqg-ball-swallow} and~\ref{prop-lqg-ball-union}) plays a role in the proof that simple random walk on the adjacency graph of cells approximates Brownian motion.
}
\end{figure}

\subsection{Main results}
\label{sec-main-results}

Let $(\mcl X , D, \mu)$ be an instance of the Brownian map, disk, plane, or half-plane, equipped with its metric $D$ and its natural area measure $\mu$. 
Conditional on $(\mcl X,D,\mu )$, let $\mcl P^\lambda$ for $\lambda > 0$ be a Poisson point process on $\mcl X$ with intensity measure $\lambda \mu$.  
In the case of the Brownian map or disk (when $\mu$ is a finite measure) with area equal to $A$ we can equivalently define $\mcl P^\lambda$ as follows: sample $N \sim \op{Poisson}(\lambda A)$, independently from $(\mcl X , D , \mu  )$, and then conditional on $N$ and $(\mcl X , D , \mu  )$ sample $N$ points uniformly and independently from $\mu$. 
 
For $z\in\mcl P^\lambda$, let $H_z^\lambda$ be the \emph{Voronoi cell} which is the closure of the set of points in $\mcl X$ which are $D $-closer to $z$ than to any other point of $\mcl P^\lambda$.  There is a natural graph structure on $\mcl P^\lambda$ whereby $z,w\in\mcl P^\lambda$ are connected by an edge if and only if the cells $H_z^\lambda$ and $H_w^\lambda$ intersect along their boundaries. Equivalently, $z,w\in\mcl P^\lambda$ are joined by an edge if and only if there exists $u\in\mcl X$ such that $D(u,z) = D(u,w)$ and $D(u,x) \geq D(u,z)$ for each $x\in\mcl P^\lambda\setminus \{z,w\}$.  In the case of the Brownian disk or half-plane, we define $\bdy \mcl P^\lambda$ to be the set of $z\in\mcl P^\lambda$ for which the corresponding cell intersects $\bdy \mcl X$.

Our first main result (from which all of our other main results will follow) says that the simple random walk on $\mcl P^\lambda$ converges modulo time parameterization and that the limiting process coincides with Brownian motion under the embedding of $\mcl X$ into $\BB C$ which arises from its identification with a $\sqrt{8/3}$-LQG surface. 
In fact, the convergence is uniform over all choices of starting point in any given compact set $K\subset\mcl X$ (in the case of the Brownian map or disk, we can just take $K = \mcl X$). 
This gives us an explicit, intrinsic definition of Brownian motion on a Brownian surface which coincides with the definition which comes from LQG theory.
 
For $z \in \mcl X$, let $Y^{z,\lambda}$ be the simple random walk on $\mcl P^\lambda$ started from the point of $\mcl P^\lambda$ whose corresponding cell contains $z$ (if there is more than one such point, we choose one in an arbitrary manner).  
We extend $Y^{z,\lambda}$ from the integers to $[0,\infty)$ by declaring that for each $j\in\BB N$, the path $Y^{z,\lambda}|_{[j-1,j]}$ traverses the $D$-geodesic from $Y^{z,\lambda}_{j-1}$ to $Y^{z,\lambda}_j$ at constant speed.\footnote{It is easy to see that for any compact set $K\subset \mcl X$, the following is true. The supremum over all $j$ such that $Y_j^{z,\lambda} \in K$ of the $D$-diameter and the Euclidean diameter of $Y^{z,\lambda}|_{[j-1,j]}$ converges to zero in probability as $\lambda\rta\infty$. This is the only property of the continuous extension of the walk which is needed for our purposes. To see why this property holds, we observe that the $D$-diameter of $Y^{z,\lambda}|_{[j-1,j]}$ equals $D(Y^{z,\lambda}_{j-1}  , Y^{z,\lambda}_j)$. Since $Y_{j-1}^{z,\lambda}$ and $Y_j^{z,\lambda}$ are contained in adjacent Voronoi cells and the maximum $D$-diameter of all of the cells which intersect $K$ tends to zero in probability as $\lambda\rta\infty$ (by Lemma~\ref{lem-max-cell-diam}), we get the desired statement for the $D$-diameters. The desired property for Euclidean diameters follows since $D$ induces the same topology as the Euclidean metric.} 
In the case of the Brownian disk or half-plane, we stop $Y^{z,\lambda}$ at the first time it hits a point of $\bdy\mcl P^\lambda$.

\begin{thm}[Brownian motion on a Brownian surface] \label{thm-rw-conv0}
If $K\subset \mcl X$ is a compact set chosen in a manner which is measurable with respect to $(\mcl X , D , \mu)$, then as $\lambda\rta\infty$ the conditional law of the walk $Y^{z,\lambda}$ given $(\mcl X, D , \mu)$ converges in probability as $\lambda\rta 0$, uniformly over all $z\in K$, with respect to the topology on curves viewed modulo time parameterization (see Section~\ref{sec-cmp-metric} for a review of this topology). 
If we identify $\mcl X$ with the Riemann sphere, unit disk, complex plane, or upper half-plane using the correspondence between Brownian and $\sqrt{8/3}$-LQG surfaces, then the limit of the conditional law of $Y^{z,\lambda}$ given $(\mcl X,D,\mu)$ is standard planar Brownian motion started from $z$ (and stopped upon hitting the domain boundary in the case of a surface with boundary), viewed modulo time parameterization. 
\end{thm}

In light of Theorem~\ref{thm-rw-conv0}, one can define Brownian motion on $\mcl X$ to be the limit of the processes $Y^{z,\lambda}$ as $\lambda\rta\infty$. 
See Theorem~\ref{thm-rw-conv} below for a more general version of Theorem~\ref{thm-rw-conv0} which also applies to other $\sqrt{8/3}$-LQG surfaces.  Theorem~\ref{thm-rw-conv0} allows us to give intrinsic constructions of other objects on Brownian surfaces which can be derived from random walk.  For example, by combining it with the main result of~\cite{yy-lerw} we obtain an intrinsic definition of SLE$_2$ on a Brownian surface as the limit of the loop-erased random walk on Poisson-Voronoi tessellations.
Our techniques also imply variants of Theorem~\ref{thm-rw-conv0} in which the edges of $\mcl P^\lambda$ are assigned random conductances in a sufficiently nice way. For example, Theorem~\ref{thm-rw-conv0} is still true if the conductances are i.i.d.\ and the conductances and their reciprocals have finite expectation. Essentially, this is because the key step in the proof of Theorem~\ref{thm-rw-conv0} uses the main result of~\cite{gms-random-walk}, and the result in~\cite{gms-random-walk} allows for variable conductances. We will not state a general variable conductance theorem here, but we remark that if the reader wants to extend Theorem~\ref{thm-rw-conv0} to a particular variable conductance model, the key step will be checking that the hypotheses of~\cite{gms-random-walk} remain satisfied.

Theorem~\ref{thm-rw-conv0} also gives us an explicit construction of the embedding of a Brownian surface into $\BB C$. 
More precisely, one can define an explicit way of embedding the graphs $\mcl P^\lambda$ into $\BB C$ --- called the \emph{Tutte embedding} --- under which the metric and area measure on $\mcl P^\lambda$ inherited from $(\mcl X,D,\mu)$ converge to the $\sqrt{8/3}$-LQG metric and area measure, respectively, as $\lambda\rta\infty$.
For concreteness, let us focus on the case when $(\mcl X,D,\mu)$ is the Brownian disk (we do this since our embedding has a simpler definition when our surface has a boundary). 

We will now define the Tutte embedding $\Phi^\lambda$ into the closed unit disk $\ol{\BB D}$ of the Poisson point process $\mcl P^\lambda$ (equipped with the graph structure discussed above). 
Recall that $\bdy\mcl P^\lambda$ denotes the points in $\mcl P^\lambda$ whose corresponding cells intersect $\bdy\mcl X$. 
We first specify the points which will be sent (approximately) to 0 and 1.
Let $\BB z$ be sampled uniformly from $\mu$ and let $z_0$ be the a.s.\ unique (see Lemma~\ref{lem-zero-lqg-mass} below) element of $\mcl P^\lambda$ such that $\BB z \in H_{z_0}^\lambda$.  
Also let $\BB x$ be sampled uniformly from the canonical length measure on $\bdy\mcl X$ (which can be defined, e.g., by taking a limit of the re-scaled $\mu$-mass of $D$-neighborhoods of boundary arcs~\cite{legall-disk-snake}) and let $x_0$ be the a.s.\ unique element of $\bdy\mcl P^\lambda$ for which $\BB x \in H_{x_0}^\lambda$. 
Enumerate the other points in $\bdy \mcl P^\lambda$ as  $x_1,\dots,x_m$ so that if $j,k \in \{0,\dots,m\}$, then $j < k$ if and only if we hit the cell $H_{x_j}^\lambda$ before the cell $H_{x_k}^\lambda$ when we traverse $\bdy\mcl X$ in the counterclockwise direction started from $\BB x$. 

For $j\in \{0,\dots ,m\}$, we declare that $\Phi^\lambda(x_j) = e^{2\pi i p_j} \in\bdy\BB D$, where $p_j$ is the probability that a simple random walk on $\mcl P^\lambda$ started from $z_0$ first hits $\bdy\mcl P^\lambda$ at one of the points $x_0,\dots,x_j$. Note that this makes it so that the harmonic measure as seen from $z_0$ approximates the uniform measure on $\bdy\BB D$. 
This defines an embedding of $\bdy\mcl P^\lambda$. We then extend this embedding to be discrete harmonic on the rest of $\mcl P^\lambda$, equivalently we require that the position of each interior vertex of $\mcl P^\lambda$ under our embedding is the average of the positions of its neighbors. 

Note that under this embedding, the center point of the cell $x_0$ containing $\BB x$ maps to $e^0 = 1$.
Moreover, the $\mcl P^\lambda$-harmonic measure on $\bdy \mcl P^\lambda$ as viewed from $z_0$ approximates the uniform measure on $\bdy\BB D$, the average of which is zero. 
So, $\Phi^\lambda(z_0)$ is close to zero (but not necessarily exactly equal to zero) when $\lambda$ is large.

\begin{thm}[Tutte embedding convergence] \label{thm-tutte-conv} 
Let $(\mcl X , D , \mu)$ be a Brownian disk (with any fixed positive choice of area and boundary length) and let $h$ be the GFF-type distribution on $\BB D$ which parameterizes the $\sqrt{8/3}$-LQG surface\footnote{In particular, this surface is a quantum disk with a marked interior point at 0 and a marked boundary point at 1; see Section~\ref{sec-lqg-metric}.} corresponding to $(\mcl X , D,\mu)$ under the correspondence of~\cite{lqg-tbm1,lqg-tbm2,lqg-tbm3}.
Also let $ \mu_h$ and $D_h$ be the $\sqrt{8/3}$-Liouville quantum gravity area measure and metric, respectively, induced by $h$, so that $(\mcl X , D,\mu) = (\ol{\BB D} , D_h, \mu_h)$ as metric measure spaces.
If we identify $\mcl P^\lambda$ with its image under the Tutte embedding, then we have the following convergence in probability as $\lambda\rta\infty$. 
\begin{itemize}
\item The measure which assigns to each vertex of $\mcl P^\lambda$ a mass equal to the $\mu$-mass of its corresponding cell converges to $\mu_h$.  The same is true of the counting measure on vertices of $\mcl P^\lambda$, scaled by the factor~$\lambda^{-1}$. 
\item The maximum over all pairs of embedded vertices $z,w \in \mcl P^\lambda$ of the quantity $|D_h(z,w) - D(z,w)|$ converges to zero. 
\item The simple random walk on vertices of $\mcl P^\lambda$ started from $z_0$ converges modulo time parameterization to Brownian motion started from 0 and stopped upon hitting $\bdy\BB D$ in the quenched sense (i.e., its conditional law given $\mcl P^\lambda$ and $(\mcl X,D,\mu)$ converges weakly in probability as $\lambda\rta\infty$). 
\end{itemize}
\end{thm}

As we will explain in Section~\ref{sec-tutte-conv-proof}, Theorem~\ref{thm-tutte-conv} is a consequence of Theorem~\ref{thm-rw-conv0}.
Indeed, to prove Theorem~\ref{thm-tutte-conv} we just need to show that when $\lambda$ is large, the Tutte embedding of $\mcl P^\lambda$ is close to the image of $\mcl P^\lambda$ under the \emph{a priori} embedding of $(\mcl X , D , \mu)$ which comes from its identification with $(\ol{\BB D} , D_h, \mu_h)$.
Due to the manner in which the Tutte embedding is defined, this, in turn, follows from the statement that under the \emph{a priori} embedding, the random walk on $\mcl P^\lambda$ converges to Brownian motion modulo time parameterization, as asserted in Theorem~\ref{thm-rw-conv0}.  

We now briefly outline the proof of Theorem~\ref{thm-rw-conv0}. 
We want to show that under the \emph{a priori} embedding, the random walk on $\mcl P^\lambda$ converges to Brownian motion modulo time parameterization. 
This is a random walk in random environment (RWRE) problem: we have a random walk on $\mcl P^\lambda$ --- viewed as a random graph drawn in $\ol{\BB D}$ --- and we want to show that it approximates Brownian motion. 
However, this problem falls outside of the usual RWRE or random conductance model framework (as surveyed, e.g., in~\cite{bf-rwre-survey,biskup-rwre-survey}) because the Euclidean sizes of the Voronoi cells under the \emph{a priori} embedding vary dramatically from one location to another, so the environment is highly spatially inhomogeneous (see Figure~\ref{fig:pv_sim}) and in particular its law is not stationary with respect to spatial translations. 

Nevertheless, as explained in Section~\ref{sec-tutte-conv}, if we consider a certain special $\sqrt{8/3}$-LQG surface called a \emph{0-quantum cone} (which does not correspond to one of the ``standard" Brownian surfaces) then the associated adjacency graph of Poisson-Voronoi cells is in a certain sense ``translation invariant modulo a global rescaling.''
The paper~\cite{gms-random-walk} gives conditions under which random walk converges to Brownian motion modulo time parameterization in a random environment which is only required to satisfy this weaker form of translation invariance. 
We re-state the particular theorem from~\cite{gms-random-walk} which we will use as Theorem~\ref{thm-general-clt-uniform} below.
Once certain properties of our Voronoi cells have been established, this theorem shows that random walk on the Poisson-Voronoi tessellation of the 0-quantum cone converges to Brownian motion modulo time parameterization. 
One can then transfer to other $\sqrt{8/3}$-LQG surfaces (such as the ones corresponding to the Brownian map, disk, plane, and half-plane) via local absolute continuity considerations.

The key quantitative condition needed to apply the above RWRE theorem is that the quantity $\op{diam}(H_0)^2 \op{deg}(H_0) / \op{area}(H_0)$ has finite expectation, where $H_0$ is the Voronoi cell containing the origin for the $\gamma$-quantum cone with $\lambda=1$ and $\op{diam}$, $\op{deg}$, and $\op{area}$ denote its Euclidean diameter, degree (in the adjacency graph of Voronoi cells), and Lebesgue measure, respectively. 
We will show in Section~\ref{sec-finite-expectation}, using a mass-transport principle, that in fact it suffices to prove that
\eqb \label{eqn-ball-cell-moment0}
  \BB E\left[  \sum_{H\in\mcl H  : 0 \in B_H} \frac{\op{diam}(H)^2 \op{deg}(H)}{\op{area}(B_H)}   \right]  < \infty, 
\eqe
where $B_H$ is the smallest $D_h$-metric ball centered at the center point of $H$ (i.e., the point of $\mcl P$ which is in $H$) which contains $H$. 
This will be important for our purposes since it is easier to lower-bound the Lebesgue measure of an LQG metric ball than a Voronoi cell. 
In order to verify~\eqref{eqn-ball-cell-moment0}, we need to establish a number of estimates for $\sqrt{8/3}$-LQG metric balls and Voronoi cells which are of independent interest (see Section~\ref{sec-ball-estimates}). 

In particular, we show in Proposition~\ref{prop-lqg-ball-swallow} that a $\sqrt{8/3}$-LQG metric ball is extremely unlikely to be ``long and skinny" in the sense that it typically contains a Euclidean ball of radius comparable to its Euclidean diameter. 
This is done using a percolation argument for the GFF, similar to ones in~\cite{ding-dunlap-lqg-fpp,ding-goswami-watabiki,dzz-heat-kernel,dg-lqg-dim,df-lqg-metric,ding-dunlap-lgd,dddf-lfpp}.
We also show in Proposition~\ref{prop-ball-vol} that the $\sqrt{8/3}$-LQG mass of an LQG metric ball is highly concentrated around the fourth power of its $\sqrt{8/3}$-LQG radius. This is done by starting with estimates for the Brownian map~\cite{legall-geodesics}, then using the local independence properties of the GFF to establish a suitable concentration bound. 

The high-level strategy used in this paper (especially, the application of~\cite{gms-random-walk}) is similar to the strategy used in~\cite{gms-tutte} to prove the convergence to LQG of the Tutte embedding of the so-called \emph{mated-CRT map}. The mated-CRT map is a random planar map built by mating a pair of continuum random trees, which has an \emph{a priori} embedding into $\BB C$ due to the results of~\cite{wedges}. 
However, the proof of the needed bound for $\op{diam}(H_0)^2 \op{deg}(H_0) / \op{area}(H_0)$ in this paper, including the reduction to~\eqref{eqn-ball-cell-moment0} and the estimates for $\sqrt{8/3}$-LQG metric balls, is very different from the proof of the analogous bound in~\cite{gms-tutte}. 
The proof of this estimate comprises most of the technical work in this paper.

All of the arguments in this paper carry over verbatim to the $\gamma$-LQG metric for general $\gamma\in(0,2)$, as defined in~\cite{gm-uniqueness}, except for the proof of the ball volume concentration bound in Proposition~\ref{prop-ball-vol} (which uses estimates for the Brownian map, so only works for $\gamma=\sqrt{8/3}$). 
If we had an analog of Proposition~\ref{prop-ball-vol} for general $\gamma\in(0,2)$, we could immediately extend our results to Poisson-Voronoi tessellations of $\gamma$-LQG surfaces for all $\gamma\in(0,2)$. 
 
\begin{remark}[Embeddings of random planar maps] \label{remark-map-embedding}
Theorem~\ref{thm-tutte-conv} implies a scaling limit result for certain ``coarse-grained" embeddings of random planar maps toward $\sqrt{8/3}$-LQG, as we now explain. 
Suppose $\{M^n\}_{n\in\BB N}$ is a sequence of random planar maps with boundary which converge in law to the Brownian disk in the following sense. 
There are scaling constants $a_n,b_n,c_n > 0$ such that if we view the planar maps as curve-decorated metric measure spaces equipped with $a_n^{-1}$ times the graph distance, $b_n^{-1}$ times the counting measure on vertices, and the path which traces the boundary according to the natural ordering in such a way that each edge is traversed in $c_n^{-1}$ units of time, then the maps converge in law w.r.t.\ the Gromov-Hausdorff-Prokhorov-uniform (GHPU) topology, the analog of the Gromov-Hausdorff topology for curve-decorated metric measure spaces introduced in~\cite{gwynne-miller-uihpq}.

For $\lambda > 0$, we can define a Poisson-Voronoi tessellation of $M^n$ using a Poisson point process with respect to $\lambda b_n^{-1}$ times the counting measure on vertices of $M^n$. 
We can then define a ``coarse-grained" Tutte embedding of $M^n$ using this Poisson-Voronoi tessellation in exactly the same manner as in Theorem~\ref{thm-tutte-conv}. 
It can be seen from the GHPU convergence of $M^n$ to the Brownian disk that for each fixed $\lambda > 0$, the adjacency graph of Voronoi cells on $M^n$ converges in the total variation sense to the adjacency graph $\mcl P^\lambda$ defined above (here we emphasize that for fixed $\lambda$ the typical number of Voronoi cells is a tight random variable as $n\rta\infty$). 
Hence if we send $\lambda^n \rta\infty$ sufficiently slowly as $n\rta\infty$, we get an analog of Theorem~\ref{thm-tutte-conv} for the $\lambda^n$-coarse-grained Tutte embedding of $M^n$. 

More details regarding the above appeared in an earlier arXiv version of this paper, but were cut from the current version for brevity. 
\end{remark}

 \subsection{Outline}

The remainder of this article is structured as follows. In Section~\ref{sec-prelim}, we will fix some notation, state the scaling limit result from~\cite{gms-random-walk} which is used in the proof of our main results, and recall some facts about metric spaces and $\sqrt{8/3}$-LQG surfaces. 
In Section~\ref{sec-tutte-conv}, we prove Theorems~\ref{thm-rw-conv0} and \ref{thm-tutte-conv} assuming that a certain moment bound for Voronoi cells is satisfied.  
In Section~\ref{sec-ball-estimates} we will prove the required moment bound, along with a number of estimates for $\sqrt{8/3}$-LQG metric balls which are intermediate steps. 
Section~\ref{sec-open-problems} discusses several open problems related to the results of this paper.
Appendix~\ref{sec-voronoi-cells} contains the proofs of several elementary properties of Voronoi cells which follow from basic properties of Brownian surfaces and the GFF.

\section{Preliminaries}
\label{sec-prelim}

\subsection{Basic notation}
\label{sec-notation}

\noindent
We write $\BB N = \{1,2,3,\dots\}$ and $\BB N_0 = \BB N \cup \{0\}$. 
For $a < b$, we define the discrete interval $[a,b]_{\BB Z}:= [a,b]\cap\BB Z$. 
\medskip

\noindent
For $z\in\BB C$ and $r>0$, we write $B_r(0)$ for the open Euclidean ball of radius $r$ centered at $z$. 
We write $\op{diam} (\cdot) $ and $ \op{area}(\cdot) $ for Euclidean diameter and Lebesgue measure on $\BB C$, respectively.
\medskip

\noindent
For a metric space $(X,d)$, $x\in X$, and $r>0$, we write $B_r(x;d)$ for the open $d$-metric ball of radius $r$ centered at $x$. 
\medskip

\subsubsection{Asymptotics}
\label{sec-notation-asymp}

\noindent
If $f  :(0,\infty) \rta \BB R$ and $g : (0,\infty) \rta (0,\infty)$, we say that $f(\ep) = O_\ep(g(\ep))$ (resp.\ $f(\ep) = o_\ep(g(\ep))$) as $\ep\rta 0$ if $f(\ep)/g(\ep)$ remains bounded (resp.\ tends to zero) as $\ep\rta 0$. We similarly define $O(\cdot)$ and $o(\cdot)$ errors as a parameter goes to infinity. 
\medskip

\noindent
If $f,g : (0,\infty) \rta [0,\infty)$, we say that $f(\ep) \preceq g(\ep)$ if there is a constant $C>0$ (independent from $\ep$ and possibly from other parameters of interest) such that $f(\ep) \leq  C g(\ep)$. We write $f(\ep) \asymp g(\ep)$ if $f(\ep) \preceq g(\ep)$ and $g(\ep) \preceq f(\ep)$. 
\medskip

\noindent
Let $\{E^\ep\}_{\ep>0}$ be a one-parameter family of events. We say that $E^\ep$ occurs with
\begin{itemize}
\item \emph{polynomially high probability} as $\ep\rta 0$ if there is a $p > 0$ (independent from $\ep$ and possibly from other parameters of interest) such that  $\BB P[E^\ep] = 1 - O_\ep(\ep^p)$. 
\item \emph{superpolynomially high probability} as $\ep\rta 0$ if $\BB P[E^\ep] = 1 - O_\ep(\ep^p)$ for every $p>0$. 
\item \emph{exponentially high probability} as $\ep\rta 0$ if there exists $c >0$ (independent from $\ep$ and possibly from other parameters of interest) $\BB P[E^\ep] = 1 - O_\ep(e^{-c/\ep})$. 
\end{itemize}
We similarly define events which occur with polynomially, superpolynomially, and exponentially high probability as a parameter tends to $\infty$. 
\medskip

\noindent
We will often specify any requirements on the dependencies on rates of convergence in $O(\cdot)$ and $o(\cdot)$ errors, implicit constants in $\preceq$, etc., in the statements of lemmas/propositions/theorems, in which case we implicitly require that errors, implicit constants, etc., appearing in the proof satisfy the same dependencies.

\subsubsection{Metric spaces}
\label{sec-metric-prelim}

\noindent
Let $(X,d )$ be a metric space. 
For a curve $\gamma : [a,b] \rta X$, the \emph{$d $-length} of $\gamma$ is defined by 
\eqb \label{eqn-curve-length}
\op{len}\left( \gamma ; d  \right) := \sup_P \sum_{i=1}^{\# P} d (\gamma(t_i) , \gamma(t_{i-1})) 
\eqe 
where the supremum is over all partitions $P : a= t_0 < \dots < t_{\# P} = b$ of $[a,b]$. Note that the $d $-length of a curve may be infinite.
\medskip

\noindent
A $d$-\emph{geodesic} between two points $x,y \in X$ is a path from $x$ to $y$ of minimal $d$-length. 
\medskip

\noindent
A metric space $(X,d)$ is called a \emph{length space} if for each $z,w\in X$, the distance $d(z,w)$ is the infimum of the $d$-lengths of paths joining $z$ and $w$. 

\subsubsection{Metric on curves modulo time parameterization}
\label{sec-cmp-metric}

Our scaling limit results for random walk on embedded planar maps are with respect to the topology on curves modulo time parameterization, which we now recall. 
 If $\beta_1 : [0,T_{\beta_1}] \rta \BB C$ and $\beta_2 : [0,T_{\beta_2}] \rta \BB C$ are continuous curves defined on possibly different time intervals, we set 
\eqb \label{eqn-cmp-metric}
\BB d^{\op{CMP}} \left( \beta_1,\beta_2 \right) := \inf_{\phi } \sup_{t\in [0,T_{\beta_1} ]} \left| \beta_1(t) - \beta_2(\phi(t)) \right| 
\eqe 
where the infimum is over all increasing homeomorphisms $\phi : [0,T_{\beta_1}]  \rta [0,T_{\beta_2}]$ (the CMP stands for ``curves modulo parameterization"). It is shown in~\cite[Lemma~2.1]{ab-random-curves} that $\BB d^{\op{CMP}}$ induces a complete metric on the set of curves viewed modulo time parameterization. 

In the case of curves defined for infinite time, it is convenient to have a local variant of the metric $\BB d^{\op{CMP}}$. Suppose $\beta_1 : [0,\infty) \rta \BB C$ and $\beta_2 : [0,\infty) \rta \BB C$ are two such curves. For $r > 0$, let $T_{1,r}$ (resp.\ $T_{2,r}$) be the first exit time of $\beta_1$ (resp.\ $\beta_2$) from the ball $B_r(0)$ (or 0 if the curve starts outside $B_r(0)$). 
We define 
\eqb \label{eqn-cmp-metric-loc}
\BB d^{\op{CMP}}_{\op{loc}} \left( \beta_1,\beta_2 \right) := \int_1^\infty e^{-r} \left( 1 \wedge \BB d^{\op{CMP}}\left(\beta_1|_{[0,T_{1,r}]} , \beta_2|_{[0,T_{2,r}]} \right) \right) \, dr ,
\eqe 
so that $\BB d^{\op{CMP}}_{\op{loc}} (\beta^n , \beta) \rta 0$ if and only if for Lebesgue a.e.\ $r > 0$, $\beta^n$ stopped at its first exit time from $B_r(0)$ converges to $\beta$ stopped at its first exit time from $B_r(0)$ with respect to the metric~\eqref{eqn-cmp-metric}. 
Note that the definition~\eqref{eqn-cmp-metric} of $\BB d^{\op{CMP}}\left(\beta_1|_{[0,T_{1,r}]} , \beta_2|_{[0,T_{2,r}]} \right)$ makes sense even if one or both of $T_{1,r}$ or $T_{2,r}$ is infinite, provided we allow $\BB d^{\op{CMP}}\left(\beta_1|_{[0,T_{1,r}]} , \beta_2|_{[0,T_{2,r}]} \right) = \infty$ (this does not pose a problem due to the definition of the integrand in~\eqref{eqn-cmp-metric-loc}). 

If $(X,d,x_0)$ is a general metric space with a marked point, one can similarly define the metric on curves modulo time parameterization on $X$ but with $d$-distances in place of Euclidean distances and $d$-metric balls centered at $x_0$ in place of Euclidean balls centered at 0.

\subsection{Scaling limit for random walk on graphs of cells}
\label{sec-general-clt}

In this subsection we state a version of the main result of~\cite{gms-random-walk} which gives general conditions under which random walk on the adjacency graph of a random collection of cells (e.g., Voronoi cells) on $\BB C$ converges to Brownian motion.
This same result is also used in~\cite{gms-tutte} to prove an embedding convergence result for a different discretization of LQG. 
Let us first describe what we mean by an ``adjacency graph of cells". 

\begin{defn} \label{def-cell-config}
A \emph{cell configuration} on $\BB C$ consists of the following objects.
\begin{enumerate}
\item A locally finite collection $\mcl H$ of compact connected subsets of $\BB C$ (``cells") with non-empty interiors whose union is all of $\BB C$ and such that the intersection of any two elements of $\mcl H$ has zero Lebesgue measure.
\item A symmetric relation $\sim$ on $\mcl H\times \mcl H$ (``adjacency") such that if $H\sim H'$, then $H\cap H'\not=\emptyset$ and $H\not=H'$.  
\end{enumerate}
\end{defn}

We will typically slightly abuse notation by making the relation $\sim$ implicit, so we write $\mcl H$ instead of $(\mcl H,\sim )$.  
We view $\mcl H$ as a weighted graph whose vertices are the cells of $\mcl H$ and whose edge set is
\eqb \label{eqn-cell-edges}
\mcl E\mcl H := \left\{ \{H,H'\} \in \mcl H\times \mcl H : H\sim H' \right\} ,
\eqe
In~\cite{gms-random-walk}, one also allows for a conductance function on the edges of $\mcl H$. 
Here we will only consider cell configurations with unit conductances.
We note that the intersections of the cells of $\mcl H$ are required to have zero Lebesgue measure. 
We will check this condition for Voronoi cells in Lemma~\ref{lem-zero-lebesgue} below. 
   
We define a metric on the space of cell configurations by
\allb \label{eqn-cell-metric}
\BB d^{\op{CC}}(\mcl H,\mcl H') &:= \int_0^\infty e^{-r} \wedge \left( \inf_{f_r}  \sup_{z\in \BB C}  |z - f_r(z)| \right)     \,dr
\alle
where each of the infima is over all homeomorphisms $f_r : \BB C\rta \BB C$ such that $f_r$ takes each cell in $\mcl H $ which intersects $B_r(0)$ to a cell in $\mcl H'$ which intersects $B_r(0)$ and preserves the adjacency relation between these cells, and $f_r^{-1}$ does the same with $\mcl H$ and $\mcl H'$ reversed.  

In~\cite{gms-random-walk}, we proved that the simple random walk on a random cell configuration $\mcl H$ which satisfies the following hypotheses converges to Brownian motion.
Here, for $C>0$ and $z\in\BB C$ we write $C(\mcl H-z)$ for the cell configuration obtained by translating all of the cells by $-z$ then scaling all of the cells by $C$.  
\begin{enumerate}
\item \textbf{Translation invariance modulo scaling.} There is a (possibly random and $\mcl H$-dependent) increasing sequence of open sets $U_j \subset \BB C$, each of which is either a square or a disk, whose union is all of $\BB C$ such that the following is true. Conditional on $\mcl H$ and $U_j$, let $z_j$ for $j\in\BB N$ be sampled uniformly from Lebesgue measure on $U_j$. Then the shifted cell configurations $\mcl H - z_j$ converge in law to $\mcl H$ modulo scaling as $j\rta\infty$, i.e., there are random numbers $C_j >0$ (possibly depending on $\mcl H$ and $z_j$) such that $ C_j(\mcl H-z_j) \rta \mcl H$ in law with respect to the metric~\eqref{eqn-cell-metric}. \label{item-hyp-resampling} 
\item \textbf{Ergodicity modulo scaling.} Every real-valued measurable function $F = F(\mcl H)$ which is invariant under translation and scaling, i.e., $F(C(\mcl H-z)) = F(\mcl H)$ for each $z\in\BB C$ and $C>0$, is a.s.\ equal to a deterministic constant. \label{item-hyp-ergodic} 
\item \textbf{Finite expectation.} With $H_0$ the cell in $\mcl H$ containing 0,  \label{item-hyp-moment}
\eqb\label{eqn-hyp-moment}
\BB E\left[ \frac{\op{diam}(H_0)^2}{\op{area}(H_0)} \op{deg}(H_0)    \right] <\infty  
\eqe 
where $\op{diam}$, $\op{area}$, and $\op{deg}$ denote Euclidean diameter, Lebesgue measure, and vertex degree in $\mcl H$, respectively. 
\item \textbf{Connectedness along lines.} Almost surely, for each horizontal or vertical  line segment $L \subset \BB C$, the subgraph of $\mcl H$ induced by the set of cells which intersect $L$ is connected. \label{item-hyp-adjacency}
\end{enumerate}

The combination of hypotheses~\ref{item-hyp-resampling} and~\ref{item-hyp-ergodic} is referred to as \emph{ergodicity modulo scaling} in~\cite{gms-random-walk}. 
Several equivalent formulations of hypothesis~\ref{item-hyp-resampling} are given in~\cite[Definition 1.2]{gms-random-walk} (we will use a different formulation, in terms of a ``mass transport principle" in Section~\ref{sec-finite-expectation}).
The version of hypothesis~\ref{item-hyp-moment} given here is slightly simpler than the version in~\cite{gms-random-walk} since we are assuming unit conductances.
Hypothesis~\ref{item-hyp-adjacency} is automatically satisfied if any two cells which intersect are considered to be adjacent. This will always be the case for cell configurations considered in this paper. 
The following is~\cite[Theorem 3.10]{gms-random-walk}. 

\begin{thm}[\cite{gms-random-walk}] \label{thm-general-clt-uniform}
Let $\mcl H$ be a random cell configuration satisfying the above four hypotheses. 
For $z\in\BB C$, let $Y^z$ denote the simple random walk on $\mcl H$ started from $H_z$ (with conductances $\frk c$). For $j\in\BB N_0$, let $\wh Y_j^z$ be an arbitrarily chosen point of the cell $Y_j^z$ and extend $\wh Y^z$ from $\BB N_0$ to $[0,\infty)$ by piecewise linear interpolation. 
There is a deterministic covariance matrix $\Sigma$ with $\det\Sigma\not=0$ such that the following is true. For each fixed compact set $A \subset \BB C$, it is a.s.\ the case that as $\ep \rta 0$, the maximum over all $z\in A$ of the Prokhorov distance between the conditional law of $\ep \wh Y^{z/\ep}$ given $\mcl H$ and the law of Brownian motion started from $z$ with covariance matrix $\Sigma$, with respect to the topology on curves modulo time parameterization (as defined in Section~\ref{sec-cmp-metric}), tends to 0. 
\end{thm}

\subsection{Liouville quantum gravity surfaces} 
\label{sec-lqg-prelim}

Fix $\gamma \in (0,2)$ (in fact, we will always take $\gamma =\sqrt{8/3}$). For $k\in \BB N_0$, a \emph{$\gamma$-Liouville quantum gravity surface} with $k$ marked points is an equivalence class of $(k+2)$-tuples $(U,h,z_1,\dots,z_k)$ where $U\subset\BB C$ is an open domain, $h$ is a distribution on $U$ (which we will always take to be a realization of some variant of the GFF on $U$), and $z_1,\dots,z_k \in U\cup \bdy U$. Two such $(k+2)$-tuples $(U,h,z_1,\dots,z_k)$ and $(\wt U , \wt h , \wt z_1,\dots , \wt z_k)$ are declared to be equivalent if there is a conformal map $\phi : \wt U \rta U$ such that 
\eqb \label{eqn-lqg-coord}
\wt h = h\circ \phi + Q\log |\phi'| \quad \op{and} \quad \phi(\wt z_j) =  z_j  ,\quad \forall j\in [1,k]_{\BB Z}  \quad \text{where} \quad Q = \frac{2}{\gamma}  + \frac{\gamma}{2} .
\eqe
We think of two equivalent $(k+2)$-tuples as above as corresponding to different parameterizations of the same surface.
We refer to the distribution $h$ corresponding to an LQG surface as the \emph{embedding} of the surface. 
The above definitions first appeared in~\cite{shef-kpz}, and also play an important role, e.g., in~\cite{shef-zipper,wedges}.

If the law of the field $h$ is locally absolutely continuous with respect to the law of the Gaussian free field on $U$, then we can define the \emph{$\gamma$-LQG area measure} $\mu_h$ on $U$, which is the a.s.\ limit of regularized versions of $e^{\gamma h(z)} \,dz$ as well as the \emph{$\gamma$-LQG boundary length measure} $\nu_h$ on $\bdy U$ (in the case when $U$ has a boundary). There are several equivalent ways to construct these measures: see, e.g.,~\cite{kahane,shef-kpz,rhodes-vargas-review}. 
By~\cite[Proposition 2.1]{shef-kpz}, if $h$ and $\wt h$ are related by a conformal map as in~\eqref{eqn-lqg-coord}, then 
\eqbn
\mu_{\wt X}(X) = \mu_h(\phi(X)) \: \text{$\forall$ Borel set $X\subset \wt U$} \quad \text{and} \quad \nu_{\wt h}(Y) = \nu_h(\phi(Y)) \: \text{$\forall$ Borel set $Y \subset \bdy U$} .
\eqen
This means that $\mu_h$ and $\nu_h$ can be viewed as measures on the LQG surface. 

In the special case when $\gamma =\sqrt{8/3}$, an LQG surface also admits a metric $D_h$, as shown in~\cite{lqg-tbm1,lqg-tbm2,lqg-tbm3}, and this metric is compatible with coordinate changes of the form~\eqref{eqn-lqg-coord}. We will review the basic properties of this metric in Section~\ref{sec-lqg-metric}. 

Henceforth we fix $\gamma =\sqrt{8/3}$. We now discuss several different types of $\sqrt{8/3}$-LQG surfaces which are introduced in~\cite{wedges}.

\subsubsection{Quantum cones}

The LQG surface which we will work with must frequently is the \emph{$\alpha$-quantum cone} for $\alpha \in (-\infty,Q)$, which is defined in~\cite[Definition~4.10]{wedges}. 
The $\alpha$-quantum cone is a doubly marked surface $(\BB C ,h , 0, \infty)$ whose $\gamma$-LQG measure $\mu_h$ has infinite total mass, but assigns finite mass to every bounded subset of $\BB C$. 
One way to obtain an $\alpha$-quantum cone is to start with a whole-plane GFF plus $\alpha\log (1/|\cdot|)$ then ``zoom in" near the origin and re-scale (i.e., add a constant to the field) so that the LQG area of a fixed set remains of constant order. See~\cite[Proposition~4.13(i)]{wedges} for a precise statement.  

Quantum cones with parameter $\alpha \in \{0,\sqrt{8/3} \}$ are especially natural, and these will be the main types of quantum cones which we will consider. 
The case $\alpha = \sqrt{8/3}$ is special since a GFF a.s.\ has a $-\sqrt{8/3}$-log singularity at a point sampled from its $\sqrt{8/3}$-LQG measure~\cite[Section 3.3]{shef-kpz}, so this surface can be thought of as describing the behavior of a general $\sqrt{8/3}$-LQG surface near such a point. 
Moreover, the $\sqrt{8/3}$-quantum cone, equipped with its $\sqrt{8/3}$-LQG metric and area measure, is equivalent to the Brownian plane as defined in~\cite{curien-legall-plane} (see Section~\ref{sec-lqg-metric}). 
In a similar vein, the 0-quantum cone describes the local behavior of a $\sqrt{8/3}$-LQG surface near a \emph{Lebesgue} typical point. 
The 0-quantum cone can be used to construct cell configurations which satisfy the translation invariance modulo scaling condition of Theorem~\ref{thm-general-clt-uniform}, which says that the origin is in some sense ``Lebesgue typical". 
 
We will need some properties of quantum cones which follow from the definition in~\cite[Definition 4.10]{wedges}, so we now recall this definition. Let $\alpha < Q$ and let $A : \BB R \rta \BB R$ be the process such that $A_t =B_t  + \alpha t$ for $t\geq 0$, where $B$ is a standard linear Brownian motion; and for $t < 0$, let $A_t = \wh B_{-t} + \alpha t$, where $\wh B$ is a standard linear Brownian motion conditioned so that $\wh B_t  + (Q-\alpha) t > 0$ for all $t> 0$ and taken to be independent of $B$. 
We define $h$ to be the random distribution such that if $h_r(0)$ denotes the average of $h$ on $\partial B_r(0)$ (see~\cite[Section 3.1]{shef-kpz} for the definition and basic properties of the circle average), then $t\mapsto h_{e^{-t}}(0)$ has the same law as the process $A$; and $h - h_{|\cdot|}(0)$ is independent from $h_{|\cdot|}(0)$ and has the same law as the analogous process for a whole-plane GFF. 
  
Since a quantum cone has only two marked points, one can get a different choice of $h$ corresponding to the same LQG surface (i.e., a different embedding of the quantum cone) by re-scaling space and applying the LQG coordinate change formula~\eqref{eqn-lqg-coord}. 
We will almost always consider the particular choice of distribution $h$ defined just above. This choice of $h$ is called the \emph{circle average embedding}, and is characterized by the fact that $1 = \sup\{r > 0 : h_r(0) + Q\log r = 0\}$. If $h$ is the circle-average embedding of an $\alpha$-quantum cone, then $h|_{\BB D}$ agrees in law with the corresponding restriction of a whole-plane GFF plus $-\alpha\log |\cdot|$, normalized so that its average over $\bdy\BB D$ is 0. 
 
The $\alpha$-quantum cone possesses a certain special scale invariance property, which we now describe.
Let $\alpha < Q$, let $h$ be the circle-average embedding of an $\alpha$-quantum cone, and let $\{h_r(z) : r > 0, z\in\BB C\}$ be its circle average process. 
We define
\eqb \label{eqn-mass-hit-time}
R_b := \sup\left\{ r > 0 : h_r(0) + Q \log r = \frac{1}{\sqrt{8/3} } \log b \right\} ,\quad\forall b > 0, 
\eqe 
where here $Q$ is as in~\eqref{eqn-lqg-coord}. That is, $R_b$ gives the largest radius $r > 0$ so that if we scale spatially by the factor $r$ and apply the change of coordinates formula~\eqref{eqn-lqg-coord}, then the average of the resulting field on $\partial \BB D$ is equal to $\frac{1}{\sqrt{8/3} } \log b$. Note that $R_0 = 0$ in the case of the circle average embedding. It is easy to see from the above definition of $h$ (and is shown in~\cite[Proposition 4.13(i)]{wedges}) that for each fixed $b>0$, 
\eqb \label{eqn-cone-scale}
h \eqD h(R_b \cdot) + Q \log R_b -  \frac{1}{\sqrt{8/3} } \log b .
\eqe

It is immediate from the definitions of $\mu_h$ and the $\sqrt{8/3}$-LQG metric $D_h$ that adding $\frac{1}{\sqrt{8/3}} \log b$ to the field scales $\sqrt{8/3}$-LQG areas by $b$ and $\sqrt{8/3}$-LQG distances by $b^{1/4}$ (in the case of the metric, see~\cite[Lemma 2.2]{lqg-tbm2} or Lemma~\ref{lem-metric-f} below). 
By the $\sqrt{8/3}$-LQG coordinate change formulas for $\mu_h$ and $D_h$, we therefore see that~\eqref{eqn-cone-scale} implies that
\eqb \label{eqn-cone-scale-mm}
\left( \BB C , D_h , \mu_h \right) \eqD \left( \BB C , b^{1/4} D_h , b \mu_h \right) , \quad \forall b > 0 ,
\eqe
where here we mean equality in law as metric measure spaces. We note that the property~\eqref{eqn-cone-scale-mm} is \emph{not} true with, say, a whole-plane GFF in place of a quantum cone. This property is a major reason for considering quantum cones.

\subsubsection{Quantum disks, spheres, and wedges.}

We will also have occasion to consider other special quantum surfaces besides just quantum cones.
We will not need as many properties of these, so we just briefly mention their definitions and refer to the cited references for more details. 

A \emph{quantum disk} is a quantum surface $(\BB D , h)$ defined in~\cite[Definition~4.21]{wedges} which behaves locally like a free-boundary GFF on $\BB D$, but is defined in a slightly different way. 
Its $\sqrt{8/3}$-LQG area measure and boundary length measure each have finite total mass, and one can consider quantum disks with specified boundary length (and random area) or with specified boundary length and area. 
One can also define a quantum disk with any number of marked boundary points and interior points sampled uniformly from its $\sqrt{8/3}$-LQG boundary length and area measures, respectively.

A \emph{quantum sphere} is a quantum surface $(\BB C ,h  )$ introduced in~\cite[Definition~4.21]{wedges} with $\mu_h(\BB C ) < \infty$.
Typically one considers a unit-area quantum sphere, which means we fix $\mu_h(\BB C) = 1$. Quantum spheres with other areas are obtained by re-scaling (equivalently, adding a constant to $h$). 
As in the case of the quantum disk, one can consider quantum spheres with one or more marked points sampled uniformly from the $\sqrt{8/3}$-LQG area measure.

For $\alpha \leq Q$, an \emph{$\alpha$-quantum wedge} is a quantum surface $(\BB H , h , 0,\infty)$ defined in~\cite[Definition~4.5]{wedges} which has finite mass in every neighborhood of 0 but infinite total mass. It is the half-plane analog of the $\alpha$-quantum cone considered above and satisfies the same scaling property~\eqref{eqn-cone-scale-mm} as the $\alpha$-quantum cone.  

\subsection{The $\sqrt{8/3}$-Liouville quantum gravity metric}
\label{sec-lqg-metric}

Suppose that $U\subset\BB C$ is a connected open set and $h$ is a random distribution on $U$ which is locally absolutely continuous with respect to the GFF on $U$, in the sense that for every $z\in U$, there is a neighborhood $V$ of $z$ such that the law of $h|_V$ is absolutely continuous with respect to the corresponding restriction of the GFF.
The papers~\cite{lqg-tbm1,lqg-tbm2,lqg-tbm3} show that one can define a $\sqrt{8/3}$-LQG metric $D_h$ associated with $h$. 
We will not need the precise definition of this metric here. 
Rather, we will only use a small number of basic properties of $D_h$, which we now record.  
\begin{enumerate}
\item \textbf{Bi-H\"older with respect to Euclidean metric.} The identity map from $U$, equipped with the Euclidean metric, to $(U,D_h)$ and its inverse are each a.s.\ locally H\"older continuous with a (non-explicit) H\"older exponent. In particular, $D_h$ induces the same topology on $U$ as the Euclidean metric. \label{item-metric-holder}
\item \textbf{Existence of geodesics.} Almost surely, for any $z,w\in U$ with $D_h(z,w) < D_h(z,\bdy U)$ there is a $D_h$-geodesic from $z$ to $w$, i.e., a path from $z$ to $w$ of minimal $D_h$-length. If the law of $h$ is absolutely continuous with respect to that of a free-boundary GFF in a neighborhood of every point of $\bdy U$, then $D_h$ extends to a metric on $\ol U$ and a.s.\ for each $z,w\in \ol U$, there is a $D_h$-geodesic in $\ol U$ from $z$ to $w$. \label{item-metric-geodesic}
\item \textbf{LQG coordinate change formula.} If $U,\wt U\subset \BB C$ and $\phi : \wt U\rta  U$ is a conformal map, then 
\eqb \label{eqn-lqg-coord-metric}
D_{h\circ \phi + Q\log|\phi'|}(z,w) = D_h(\phi(z) , \phi(w)) ,\quad \forall z,w\in \wt U,  
\eqe
for $Q =  2/ \sqrt{8/3}  + \ \sqrt{8/3}/2  = 5 / \sqrt 6$. \label{item-metric-coord}
\item \textbf{Locality.} If $V\subset U$ and $z,w\in V$, then $D_{h|_V}(z,w)$ is the infimum of the $D_h$-lengths of paths \emph{in $V$} from $z$ to $w$. In particular, metric balls are locally determined by $h$.  \label{item-metric-local}
\end{enumerate}
Property~\ref{item-metric-holder} and the first statement of property~\ref{item-metric-geodesic} follow from~\cite[Theorems 1.2 and 1.3]{lqg-tbm2}, respectively, and local absolute continuity. 
The second statement of~\ref{item-metric-holder} follows from the equivalence of the quantum disk and the Brownian disk, the existence of geodesics in the latter (see, e.g.,~\cite{bet-mier-disk}), and local absolute continuity. 
Properties~\ref{item-metric-coord} and~\ref{item-metric-local} are easy consequences of the construction of $D_h$ in~\cite{lqg-tbm1,lqg-tbm2,lqg-tbm3}; see, e.g.,~\cite[Lemmas~2.3 and~2.5]{gwynne-miller-gluing}. 
 
We will also use the fact that certain special LQG surfaces $(U,h)$ are equivalent to Brownian surfaces, in the sense that the metric measure space $(\ol U , D_h , \mu_h)$ agrees in law with a Brownian surface.\footnote{The agreement between $\sqrt{8/3}$-LQG surfaces and Brownian surfaces is only up to two unknown positive, deterministic scaling constants which relate the metrics and measures. In this paper, we will always assume that we have re-scaled the metric and area measure on the Brownian surface by this scaling factor so that one has exact agreement.} See~\cite[Corollary~1.5]{lqg-tbm2} for the sphere, disk, and plane cases and~\cite[Proposition~1.10]{gwynne-miller-uihpq} for the half-plane case.
\begin{itemize}
\item The quantum sphere is equivalent to the Brownian map.
\item The quantum disk is equivalent to the Brownian disk.
\item The $\sqrt{8/3}$-quantum cone is equivalent to the Brownian plane.
\item The $\sqrt{8/3}$-quantum wedge is equivalent to the Brownian half-plane. 
\end{itemize}

We will now explain another elementary property of $D_h$ which allows us to define $D_{h+f}$ whenever $f : U\rta\BB R$ is a random continuous function coupled with $h$, even if the law of $h+f$ is not locally absolutely continuous with respect to the GFF.   
 
\begin{lem} \label{lem-metric-f}
Suppose $h$ is a random distribution on a connected open set $U\subset\BB C$ and let $f : U\rta\BB R$ be a random continuous function (not necessarily independent from $h$).
If the laws of $h$ and $h+f$ are both locally absolutely continuous with respect to the GFF on $U$, then a.s.\
\eqb \label{eqn-metric-f-max}
\exp\left( \frac{1}{\sqrt 6} \min_{x \in U} f(x) \right) D_h(z,w)   \leq D_{h+f}(z,w) \leq \exp\left( \frac{1}{\sqrt 6}  \max_{x \in U} f(x) \right) D_h(z,w)  ,\quad\forall z,w\in U .
\eqe
In fact, it is a.s.\ the case that for each $z,w\in U$, 
\eqb \label{eqn-metric-f}
D_{h+f}(z,w)  =  \inf_{\gamma : z\rta w} \int_0^{\op{len}(\gamma ; D_h)} e^{f(\gamma(t))/\sqrt 6} \,dt
\eqe
where the infimum is over all simple paths from $z$ to $w$ parameterized by their $D_h$-length. 
\end{lem}
\begin{proof}
If $f \equiv c$ is constant, then by~\cite[Lemma 2.2]{lqg-tbm2}, one has $D_{h+f} = e^{c/\sqrt 6} D_h$. In fact, the proof of~\cite[Lemma 2.2]{lqg-tbm2} shows that a.s.\ \eqref{eqn-metric-f-max} holds. 
We will now deduce~\eqref{eqn-metric-f} from~\eqref{eqn-metric-f-max}.
To this end, fix $\ep >0$. 
Since $f$ is continuous, we can find a (possibly random) $\delta >0$ such that $|f(z) - f(w)| \leq \ep$ whenever $|z-w| \leq \delta$. 

Now fix $z,w\in U$ and let $\gamma : [0,T] \rta U$ be a path from $z$ to $w$ whose $D_{h+f}$-length is at most $D_{h+f}(z,w) +\ep$. 
Choose finitely many times $0 = t_0 < t_1 < \dots < t_n = T$ such that $\max_{j\in [1,n]_{\BB Z} } \max_{t\in [t_{j-1} , t_j]} |\gamma(t) - \gamma(t_{j-1})| \leq \delta$. 
By~\eqref{eqn-metric-f-max} (applied with $B_\delta(\gamma(t_{j-1}))$ in place of $U$) and our choice of $\delta$, we have (in the notation~\eqref{eqn-curve-length})
\eqbn
\exp\left( \frac{f(t_{j-1})  - \ep }{\sqrt 6}    \right) \op{len}\left( \gamma|_{[t_{j-1}  ,t_j]} ; D_h\right) 
\leq \op{len}\left( \gamma|_{[t_{j-1}  ,t_j]} ; D_{h+f} \right) 
\leq \exp\left( \frac{f(t_{j-1})  + \ep }{\sqrt 6}    \right) \op{len}\left( \gamma|_{[t_{j-1}  ,t_j]} ; D_h\right)  .
\eqen
This shows that the $D_h$-length of $\gamma$ is finite and, if $\gamma$ is parameterized by $D_h$-length, that the $D_{h+f}$-length of $\gamma$ and the integral $\int_0^{\op{len}(\gamma ; D_h)} e^{f(\gamma(t))/\sqrt 6} \,dt$ differ by a factor of at most $e^{\ep/\sqrt{6}}$. Sending $\ep\rta 0$ shows that $D_{h+f}(z,w)$ is at least the right side of~\eqref{eqn-metric-f}. We similarly get the reverse inequality. 
\end{proof}


Let $h$ be a random distribution on a connected open set $U\subset\BB C$ whose law is locally absolutely continuous with respect to the GFF on $U$ and let $f : U\rta\BB R$ be a random continuous function (not necessarily independent from $h$).
If $D_h$ is defined, we define $D_{h+f}$ by the formula~\eqref{eqn-metric-f}. 
We need to make sure that $D_{h+f}$ is well-defined (i.e., we get the same metric if we make a different choice of $h$ and $f$ with $h' + f' = h+f$) and that it is a measurable function of $h+f$ (a priori we only know that $D_{h+f}$ is a measurable function of $(h,f)$).  
The following lemma is an easy consequence of Lemma~\ref{lem-metric-f}. We will give the proof just below.

\begin{lem} \label{lem-weighted-metric}
Let $h$ and $f$ be as above and define $D_{h+f}$ by the formula~\eqref{eqn-metric-f}. 
Then $D_{h+f}$ is a.s.\ determined by $h+f$. 
Moreover, if $(h',f')$ is another pair consisting of a random distribution on a connected open set $U\subset\BB C$ whose law is locally absolutely continuous with respect to the GFF on $U$ and a random continuous function such that $h+f \eqD h'+f'$. 
Then $(h+f, D_{h+f}) \eqD (h'+f', D_{h'+f'})$. 
\end{lem}

Once Lemma~\ref{lem-weighted-metric} is established, it follows from the above properties of $D_h$ that $D_{h+f}$ is locally bi-H\"older continuous with respect to the Euclidean metric in the sense of property~\ref{item-metric-holder} above (so induces the same topology on $U$ as the Euclidean metric) and satisfies the LQG coordinate change formula~\eqref{eqn-lqg-coord} and the locality property~\ref{item-metric-local}. 
We also note that Lemma~\ref{lem-weighted-metric} implies that for a given choice of $h$, a.s.\ $D_{h+f}$ can be defined via the formula~\eqref{eqn-metric-f-max} \emph{simultaneously} for every choice of continuous function $f : U \rta \BB R$. Indeed, this follows by considering a countable collection of functions $f$ which is dense in the space of all continuous functions $U\rta \BB R$ w.r.t.\ the local uniform topology. 

\begin{proof}[Proof of Lemma~\ref{lem-weighted-metric}]
For a length metric $D$ on $U$ and a continuous function $f : \ol U \rta \BB R$, write $e^{f/\sqrt 6} \cdot D$ for the metric defined by the formula~\eqref{eqn-metric-f} with $D$ in place of $D_h$. 
From the definition~\eqref{eqn-metric-f}, one immediately gets the following additivity property: for every metric $D$ on $U$ which induces the Euclidean topology and any two continuous functions $f,g : \ol U \rta \BB R$, 
\eqb \label{eqn-metric-commute}
e^{ g / \sqrt 6} \cdot (e^{ f /\sqrt 6} \cdot D) = e^{ (f+g) / \sqrt 6 } \cdot D .
\eqe
Indeed, this follows since the two metrics in~\eqref{eqn-metric-commute} induce the same length measure on each path in~$U$. 

Suppose now that we are given two couplings $(h,f)$ and $(h',f')$ of a distribution whose law is locally absolutely continuous w.r.t.\ the GFF and a random continuous function such that $h+f \eqD h'+f'$. We can couple $(h,f,h',f')$ in such a way that $h+f = h'+f'$ and $(h,f)$ and $(h',f')$ are conditionally independent given $h+f$. 
By~\eqref{eqn-metric-commute} applied to the functions $f$ and $f'-f$ together with Lemma~\ref{lem-metric-f} applied to the distributions $h'$ and $h = h' + f' - f$, we get that a.s.\
\eqb
D_{h'+f'} = e^{  f' / \sqrt 6} \cdot D_{h'} = e^{ f / \sqrt 6} \cdot \left( e^{  (f'-f) / \sqrt 6} \cdot D_{h'} \right) = e^{  f / \sqrt 6} \cdot D_h = D_{h+f} .
\eqe
Hence $(h+f , D_{h+f}) \eqD (h'+f', D_{h'+f'})$.
Moreover, since $D_{h+f}$ and $D_{h'+f'}$ are conditionally independent given $h+f$ (by our choice of coupling), it follows that $D_{h+f}$ is a.s.\ determined by $h+f$. 
\end{proof}

\begin{remark} \label{remark-general-gamma}
All of the properties of $D_h$ discussed in this section also hold for the $\gamma$-LQG metric for general $\gamma \in (0,2)$ from~\cite{gm-uniqueness}, except that $1/\sqrt 6$ is replaced by $\gamma/d_\gamma$, where $d_\gamma$ is the Hausdorff dimension of the $\gamma$-LQG metric (which is not known explicitly).
In fact,~\cite{gm-uniqueness} shows that a list of properties similar to the ones discussed in this section uniquely characterize the $\gamma$-LQG metric up to a deterministic multiplicative constant. 
\end{remark}

\section{Proof of main results, assuming finite expectation hypothesis}
\label{sec-tutte-conv}

For a GFF-type distribution $h$ on a connected open domain $\mcl D \subset \BB C$, we write $D_h$ and $\mu_h$ for its $\sqrt{8/3}$-LQG metric and area measure, respectively. 
Conditional on $h$, for $\lambda>0$ we let $\mcl P_h^\lambda$ be a Poisson point process on $\mcl D$ with intensity measure $\lambda \mu_h$.
For $z\in\mcl P_h^\lambda$, let $H_{h,z}^\lambda \subset\ol{\mcl D}$ be the \emph{Voronoi cell} which is the closed set of points in $\BB C$ which are (weakly) $D_h$-closer to $z$ than to any other point of $\mcl P_h^\lambda$. 
We view $\mcl P_h^\lambda$ as a graph with two points $z,w\in\mcl P_h^\lambda$ joined by an edge if and only if $H_{h,z}^\lambda\cap H_{h,w}^\lambda \not=\emptyset$, equivalently, if and only if 
\eqb \label{eqn-cell-adjacency}
\exists u\in \ol{\mcl D} \: \text{such that} \: D_h(u,z) = D_h(u,w) \: \text{and} \: D_h(u,z)\leq D_h(u,x) ,\quad\forall x\in \mcl P_h^\lambda\setminus \{z,w\} .
\eqe

Extending the notation above, for $w\in \mcl D$ we write $H_{h,w}^\lambda$ for the (a.s.\ unique for $w$ deterministic, by Lemma~\ref{lem-zero-lebesgue} just below) Voronoi cell which contains $w$.
We define
\eqb
\mcl H_h^\lambda := \left\{ H_{h,z}^\lambda : z\in\mcl P_h^\lambda\right\} .
\eqe
We will often omit the subscript $h$ and/or the superscript $\lambda$ when these objects are clear from the context. 

In this subsection, we will prove all of our main results conditional on the following proposition. 
 
\begin{prop} \label{prop-cell-moment}
Suppose $(\BB C , h , 0, \infty)$ is a 0-quantum cone. 
Define the Voronoi cell configuration $\mcl H = \mcl H_h^1$ as above with $\lambda = 1$ and let $H_0$ be the cell which contains the origin. 
Then 
\eqb \label{eqn-cell-moment}
\BB E\left[ \frac{\op{diam}(H_0)^2 \op{deg}(H_0)}{\op{area}(H_0)}  \right] < \infty,
\eqe
where here $\op{deg}(H_0)$ denotes the degree of $H_0$ as a vertex of $\mcl H$. 
\end{prop}

Proposition~\ref{prop-cell-moment} is used in Section~\ref{sec-0cone} to check the finite expectation hypotheses of Theorem~\ref{thm-general-clt-uniform} for the Voronoi cell configuration associated with 0-quantum cone. 
The proof of Proposition~\ref{prop-cell-moment} is given in Section~\ref{sec-ball-estimates}. 
This proof is the most difficult step in the proofs of our main results, and requires us to establish several estimates for $\sqrt{8/3}$-LQG metric balls which are of independent interest.

The rest of this section is structured as follows. 
In Section~\ref{sec-0cone}, we explain why Proposition~\ref{prop-cell-moment} together with Theorem~\ref{thm-general-clt-uniform} implies a scaling limit result for random walk on the adjacency graph of Voronoi cells associated with a 0-quantum cone.
In Section~\ref{sec-rw-conv}, we transfer this result to random walk on Voronoi cells on other types of $\sqrt{8/3}$-LQG surfaces (including the ones corresponding to the Brownian map, disk, plane, and half-plane) using local absolute continuity and thereby prove Theorem~\ref{thm-rw-conv0}.
In Section~\ref{sec-tutte-conv-proof}, we deduce Theorem~\ref{thm-tutte-conv} from our scaling limit result for random walk.
The arguments in these three subsections are similar to the analogous arguments in~\cite[Section 3]{gms-tutte}.
In Section~\ref{sec-finite-expectation}, we give a re-formulation of Proposition~\ref{prop-cell-moment} which involves bounds for $\sqrt{8/3}$-LQG metric balls instead of Voronoi cells, and which turns out to be easier to prove than Proposition~\ref{prop-cell-moment} itself. 

Throughout this section, we will use several elementary properties of Voronoi cells whose proofs are collected in Appendix~\ref{sec-voronoi-cells} to avoid interrupting the main argument.

\subsection{Cell configuration corresponding to a 0-quantum cone}
\label{sec-0cone}

Let $(\BB C , h , 0, \infty)$ be a 0-quantum cone and write $\mcl P = \mcl P_h^1$ and $\mcl H =\mcl H_h^1$ for its associated Poisson point process and collection of Voronoi cells with $\lambda = 1$. 

\begin{prop} \label{prop-0cone}
The conclusion of Theorem~\ref{thm-general-clt-uniform} holds for the cell configuration $\mcl H$ above. Moreover, the covariance matrix $\Sigma$ of the limiting Brownian motion is a positive scalar multiple of the identity matrix.
\end{prop}
\begin{proof}
By Lemma~\ref{lem-cell-basic}, the cells of $\mcl H$ are a.s.\ compact with non-empty interior and $\mcl H$ is locally finite. 
By Lemma~\ref{lem-zero-lebesgue}, a.s.\ the intersection of any two cells of $\mcl H$ has zero Lebesgue measure, and by definition any two cells which are adjacent in $\mcl H$ intersect. Therefore $\mcl H$ satisfies the conditions of Definition~\ref{def-cell-config}. We will now check the conditions of Theorem~\ref{thm-general-clt-uniform}.
\medskip

\noindent\textbf{Translation invariance modulo scaling.} For $j\in\BB N$, let $R_j$ be the largest $r> 0$ for which $h_r(0) + Q\log r =  \gamma^{-1} \log j$, where $h_{r}(0)$ denotes the circle average, as in~\eqref{eqn-mass-hit-time}. We will check the needed resampling property for $U_j = B_{R_j}(0)$. By~\eqref{eqn-cone-scale}, the field $h^j := h(R_j\cdot) + Q\log R_j - \gamma^{-1} \log j$ agrees in law with $h$. In particular, by the discussion just after~\cite[Definition 4.10]{wedges}, $h^j|_{\BB D}$ agrees in law with the corresponding restriction of a whole-plane GFF, normalized so that its circle average over $\bdy\BB D$ is 0. Consequently, if we sample $z_j$ uniformly from Lebesgue measure on $B_{R_j}(0)$, then the proof of~\cite[Proposition 4.13(ii)]{wedges} along with the translation invariance of the law of the whole-plane GFF, modulo additive constant, shows that the there is a sequence of random constants $C_j\rta\infty$ such that the law of $ h (C_j(\cdot-z_j)) + Q\log C_j $, restricted to any compact subset $K\subset \BB C$, converges to the law of $h|_K$ in the total variation sense as $j\rta\infty$.
By the LQG coordinate change formula, this implies that the joint law of $\mu_h(C_j(\cdot-z_j))|_K$ and $D_h(C_j(\cdot-z_j) , C_j(\cdot-z_j))|_K$ converges in the total variation sense to the joint law of $\mu_h|_K$ and $D_h|_K$ as $j\rta\infty$. 
This implies that $C_j(\mcl H-z_j) \rta \mcl H$ in law as $j\rta\infty$.
\medskip

\noindent\textbf{Ergodicity modulo scaling.} 
It is easily checked that $\bigcap_{R >0} \sigma\left(h|_{\BB C\setminus B_R(0)} \right)$ is the trivial $\sigma$-algebra (see, e.g.,~\cite[Lemma 2.2]{hs-euclidean} for the case of the whole-plane GFF; the case of $h$ can be treated in an identical manner due to~\cite[Definition 4.10]{wedges}). 
From this, it follows that also the intersection over all $R>0$ of the $\sigma$-algebra generated by $h|_{\BB C\setminus B_R(0)}$ and the set of points of $\mcl P$ which are contained in $\BB C\setminus B_R(0)$ is trivial.
For any $R>0$, there is an $R' = R'(R)  > 0$ such that each cell of $\mcl H$ which intersects $\BB C\setminus B_{R }(0)$ is contained in $\BB C\setminus B_{R'}(0)$, and we have $R'\rta\infty$ as $R\rta\infty$. 
It therefore follows that $\bigcap_{R>0} \sigma\left( \mcl H(B_R(0)) \right) $ is the trivial $\sigma$-algebra. 

To deduce condition~\ref{item-hyp-ergodic} from this, consider a real-valued function $F = F(\mcl H)$ satisfying $F(C(\mcl H-z))  = F(\mcl H)$ for each $C>0$ and $z\in\BB C$. If $F$ is determined by $\mcl H(B_R(0))$ for any $R>0$, then $F$ is equal to a deterministic constant a.s.\ since $F = F(B_R(0) - z)$ for every $z\in\BB C$ so $F$ is measurable with respect to the $\sigma$-algebra $\bigcap_{R>0} \sigma\left( \mcl H(B_R(0)) \right) $. In general, the conditional law of $F$ given $\mcl H(B_R(0))$ must be deterministic by the preceding sentence, so $F$ is independent from $\mcl H(B_R(0))$, whence the above claim implies that $F$ is equal to a deterministic constant a.s. 
\medskip

\noindent\textbf{Finite expectation.} This is the content of Proposition~\ref{prop-cell-moment}, which will be proven in Section~\ref{sec-ball-estimates}.
\medskip

\noindent\textbf{Connectedness along lines.} This follows since by definition two cells of $\mcl H$ are connected by an edge of $\mcl E\mcl H$ if and only if they intersect and the collection of cells is locally finite. 
\medskip

\noindent The covariance matrix $\Sigma$ is a scalar multiple of the identity since the law of $h$, and therefore the law of $\mcl H$, is invariant under rotations around the origin. 
\end{proof}

\subsection{Random walk on cells converges to Brownian motion}
\label{sec-rw-conv}

The following theorem is a generalization of Theorem~\ref{thm-rw-conv0} (recall the correspondence between Brownian and $\sqrt{8/3}$-LQG surfaces as described in Section~\ref{sec-lqg-metric}).

\begin{thm} \label{thm-rw-conv}
Suppose that we are in one of the following situations.
\begin{itemize}
\item $\mcl D = \BB C$, $\alpha < Q$, and $(\BB C  , h , 0, \infty)$ is an $\alpha$-quantum cone. 
\item $\mcl D = \BB C$ and $(\BB C , h , 0, \infty)$ is a doubly marked quantum sphere, e.g., with fixed area.
\item $\mcl D = \BB H$, $\alpha < Q$, and $(\BB H , h  , 0, \infty)$ is an $\alpha$-quantum wedge.
\item $\mcl D = \BB D$ and $(\BB D , h )$ is a quantum disk with fixed boundary length or fixed boundary length and area.
\end{itemize}
For $z\in \ol{\mcl D}$ and $\lambda > 0$, let $Y^{z,\lambda}  $ be the simple random walk on the adjacency graph of the Voronoi cell configuration $\mcl H_h^\lambda$.
Let $\wh Y^{z,\lambda}$ be the image of $Y^{z,\lambda}$ under the map which sends each Voronoi cell to its center point and extend $\wh Y^{z,\lambda}$ to a function from $[0,\infty)$ to $\ol{\mcl D}$ by piecewise linear interpolation at constant speed.  

For each deterministic compact set $K\subset \ol{\mcl D}$, the supremum over all $z\in K$ of the Prokhorov distance between the conditional law of $\wh Y^{z,\lambda} $ given $(h , \mcl P_h^\lambda)$ and the law of a standard two-dimensional Brownian motion started from $z$ (and stopped when it hits the boundary in the case of a quantum wedge or quantum disk), with respect to the metric on curves viewed modulo time parameterization (i.e., the metric~\eqref{eqn-cmp-metric} in the disk or half-plane case or the metric~\eqref{eqn-cmp-metric-loc} in sphere or whole-plane case) converges to 0 in probability as $\lambda\rta \infty$. 
\end{thm}

We note that in Theorem~\ref{thm-rw-conv}, the walk is extended by piecewise linear interpolation whereas in Theorem~\ref{thm-rw-conv0} it follows $D  $-geodesics between the points of $\mcl P^\lambda$. This does not affect the conclusion of the theorem: indeed, by Lemma~\ref{lem-max-cell-diam} and the fact that $D_h$ induces the Euclidean topology, for any fixed compact set $K\subset\ol{\mcl D}$, the maximum over all adjacent pairs of vertices $z,w\in\mcl P_h^\lambda \cap K$ of the Euclidean diameter of every $D_h$-geodesic from $z$ to $w$ tends to zero in law as $\lambda\rta\infty$. The same is true with $D_h$-diameters in place of Euclidean diameters and/or line segments in place of $D_h$-geodesics.

We first prove Theorem~\ref{thm-rw-conv} in the case of the 0-quantum cone, using Proposition~\ref{prop-0cone}. 
This is the step in the proof where we go from a.s.\ converges to convergence in probability. 

\begin{lem} \label{lem-rw-conv-0cone}
Theorem~\ref{thm-rw-conv} is true with a 0-quantum cone in place of a $\gamma$-quantum cone. 
\end{lem}
\begin{proof}
By Brownian scaling the statement of the lemma is invariant under the operation of changing the embedding $h$ (i.e., replacing $h$ by $h(r\cdot)  + Q\log r$ for some possibly random $r>0$), so we can assume without loss of generality that $h$ has the circle-average embedding, as described in Section~\ref{sec-lqg-prelim} and~\cite[Definition 4.10]{wedges} (we could also, e.g., embed so that $\mu_h(\BB D ) =1$). 
 
Proposition~\ref{prop-0cone} together with Theorem~\ref{thm-general-clt-uniform} tells us that a.s.\ the conditional law given $\mcl H_h^1$ of the random walk on $\ep\mcl H_h^1$ converges in law as $\ep\rta0$ to standard two-dimensional Brownian motion modulo time parameterization, and the convergence is uniform over all starting points in any fixed compact subset of $\BB C$. 

For $\lambda  > 1$, we typically do not have $ \mcl H_h^\lambda  = \ep \mcl H_h^1$ for any $\ep > 0$ since $\mcl H_h^\lambda$ is defined by scaling the intensity measure of the Poisson point process rather than by scaling space. 
Nevertheless, we have $\ep \mcl H_h^1 \eqD\mcl H_h^\lambda$ for a certain random choice of $\ep$, as we now explain.

For $b  > 0$, let $R_b > 0$ be as in~\eqref{eqn-mass-hit-time} and let $h^b := h(R_b\cdot) + Q\log R_b - \frac{1}{\sqrt{8/3}} \log b$, so that by~\eqref{eqn-cone-scale}, $h^b\eqD h$. 
By the LQG coordinate change formula~\cite[Proposition 2.1]{shef-kpz}, $\mu_{h^b}(\cdot) =  b \mu_h(R_b^{-1}\cdot)$.  
Hence if $\mcl P_{h^b}^1$ is a Poisson point process with intensity measure $\mu_{h^b}$, then $R_b^{-1} \mcl P_{h^b}^1$ is a Poisson point process with intensity measure $b \mu_h$. 
Therefore, for $\lambda >0$, 
\eqb \label{eqn-cell-law}
\mcl H_{h^\lambda}^1 \eqD \mcl H_h^1 \quad \text{and} \quad \mcl H_h^\lambda   \eqD  R_\lambda^{-1} \mcl H_{h^\lambda}^1 . 
\eqe
Since $R_\lambda\rta \infty$ as $\lambda\rta\infty$, we now get the desired convergence in probability from Proposition~\ref{prop-0cone} and Theorem~\ref{thm-general-clt-uniform}.
\end{proof}

Using local absolute continuity, we can now transfer to other quantum surfaces, starting with the case of quantum cones.

\begin{lem} \label{lem-rw-conv-cone}
Theorem~\ref{thm-rw-conv} is true in the case of the $\alpha$-quantum cone for $\alpha <Q$. 
\end{lem}
\begin{proof} 

As in the proof of Lemma~\ref{lem-rw-conv-0cone}, we work with the circle-average embedding of the $\alpha$-quantum cone, which has the property that $h|_{\BB D}$ agrees in law with the corresponding restriction of a whole-plane GFF plus $-\alpha \log |\cdot|$, normalized so that its circle average over $\bdy\BB D$ is 0. We also let $\wt h$ be the circle-average embedding of a 0-quantum cone in $(\BB C, 0,\infty)$, so that $\wt h|_{\BB D} \eqD (h +\alpha\log|\cdot|)|_{\BB D}$. 

The statement of the lemma is essentially a consequence of Lemma~\ref{lem-rw-conv-0cone} and local absolute continuity (in the form of~\cite[Proposition 3.4]{ig1}), but a little care is needed since we only have local absolute continuity between the laws of a $h$ and $\wt h$ on domains at positive distance from 0 (due to the $\alpha$-log singularity of $h$) and from $\bdy \BB D$ (due to our choice of embedding). Throughout the proof, the Prokhorov distance is always taken with respect to the metric on curves viewed modulo time parameterization.
 
For $\rho  > 0$ and $z\in B_\rho(0)$, let $J_\rho^{z,\lambda}$ for $n\in\BB N$ be the exit time from $B_\rho(0)$ of the embedded walk $\wh Y^{z,\lambda}$ on $\mcl H_h^\lambda$. Also let $\mcl B^z$ be a standard two-dimensional Brownian motion started from $z$ and let $\tau_\rho^z$ be its exit time from $B_\rho(0)$. We need to show that for each $\rho > 0$, the supremum over all $z\in B_\rho(0)$ of the Prokhorov distance between the conditional laws of $\wh Y^{z,n}|_{[0,J_\rho^{z,n}]}$ and $\mcl B^{z}|_{[0,\tau_\rho^z]}$ given $\left(h,\mcl P_h^\lambda\right)$ converges to zero in probability as $\lambda\rta\infty$. 

We first consider a radius $\rho  \in (0,1)$ and deal with the log singularity at 0. 
For $\delta \in (0,\rho)$, choose $\zeta = \zeta(\delta) \in (0,\delta)$ such that the probability that a Brownian motion started from any point of $\BB D\setminus B_\delta(0)$ hits $B_\zeta(0)$ before leaving $\BB D$ is at most $\delta$. By Lemma~\ref{lem-rw-conv-0cone} and local absolute continuity it holds with probability tending to 1 as $\lambda \rta\infty$ that for each $z\in B_\rho(0) \setminus B_\delta(0)$, the Prokhorov distance between the conditional laws of $\wh Y^{z,\lambda}|_{[0,J_\rho^{z,\lambda}]}$ and $\mcl B^z|_{[0,\tau_\rho^z]}$ given $h$ is at most $\delta$. 
Since the law of $\mcl B^z|_{[0,\tau_\rho^z]}$ depends continuously on $z$, the Prokhorov-distance diameter of the set of laws of the curves $\mcl B^z|_{[0,\tau_\rho^z]}$ for $z\in B_\delta(0)$ tends to 0 as $\delta\rta 0$.  
 
By the last two sentences of the preceding paragraph and the strong Markov property of $Y^{z,\lambda}$ and of $\mcl B^z$, it holds with probability tending to 1 as $\lambda\rta\infty$ that for each $z\in B_\delta(0)$, the Prokhorov distance between the conditional laws of $Y^{z,\lambda}|_{[J_\delta^{z,\lambda} ,J_\rho^{z,\lambda}]}$ and $\mcl B^z|_{[0,\tau_\rho^z]}$ given $(h, \mcl P_h^\lambda )$ is $o_\delta(1)$, at a deterministic rate depending only on $\rho$. The distance between the curves $Y^{z,\lambda}|_{[J_\delta^{z,\lambda} ,J_\rho^{z,\lambda}]}$ and $Y^{z,\lambda}|_{[0,J_\rho^{z,\lambda}]}$, viewed modulo time parameterization, is at most $2\delta$. Sending $\delta \rta 0$ now gives the theorem statement in the case $\rho < 1$. 

The case when $\rho\geq 1$ follows from the case when $\rho \in (0,1)$ and the scale invariance property of the $\alpha$-quantum cone~\cite[Proposition 4.13(i)]{wedges}, applied similarly as in Proposition~\ref{prop-0cone}. 
\end{proof}

\begin{lem} \label{lem-rw-conv-sphere}
Theorem~\ref{thm-rw-conv} is true in the case of the quantum sphere. 
\end{lem}
\begin{proof} 
This is immediate from Lemma~\ref{lem-rw-conv-cone} and local absolute continuity. 
\end{proof}

\begin{lem} \label{lem-rw-conv-disk}
Theorem~\ref{thm-rw-conv} is true in the case of the quantum disk.
\end{lem}
\begin{proof}
Let $(\BB C , h , 0, \infty)$ be a doubly marked quantum sphere conditioned on the event that the $D_{h}$-distance from 0 to $\infty$ is at least 1 and let $U$ be the connected component of $\BB C\setminus B_1(\infty; D_{h})$ which contains 0.
Then the conditional law of the quantum surface $(U , h|_U , 0)$ given $\nu_{h}(\bdy U)$ is that of a quantum disk with one marked point in its interior, with given boundary length (this follows, e.g., from the construction of $D_{h}$ using QLE in~\cite{lqg-tbm1}).

If we let $\mcl P_{h}^\lambda$ be a Poisson point process with intensity measure $\lambda \mu_{h}$, then $\mcl P_{h}^\lambda\cap U$ is a Poisson point process on $U$ with intensity measure $\lambda\mu_{h|_U}$. 
Let $\mcl H_{h|_U}^\lambda$ be the configuration of Voronoi cells defined using the set of points $\mcl P_{h}^\lambda\cap U$ and the metric $D_{h|_U}$.
Then each cell of $\mcl H_{h|U}^\lambda$ which does not intersect $\bdy U$ is identical to the corresponding cell of $\mcl H_{h}^\lambda $ with the same center point.
It therefore follows from Lemma~\ref{lem-rw-conv-sphere} that the maximum over all $z\in U$ of Prokhorov distance between the following two laws, with respect to the topology on curves viewed modulo time parameterization, tends to 0 as $\lambda \rta \infty$: 
\begin{itemize}
\item The conditional law given $(h , \mcl H_{h}^\lambda)$ of the random walk on $\mcl H_{h|_U}^\lambda$ stopped upon hitting a cell which intersects $\bdy U$, embedded into $U$ and linearly interpolated as in Theorem~\ref{thm-rw-conv}.
\item The law of Brownian motion a started from 0 and stopped upon hitting $\bdy U$.
\end{itemize}
By the conformal invariance of Brownian motion and the first paragraph, this gives the statement of the lemma for the quantum disk with random boundary length $\nu_{h}(\bdy U)$. By scale invariance, this implies the statement of Theorem~\ref{thm-rw-conv} for a doubly marked quantum disk with any fixed boundary length. 
By conditioning on the area of such a quantum disk, we also get the statement for a quantum disk with fixed area and boundary length. 
\end{proof}

\begin{proof}[Proof of Theorem~\ref{thm-rw-conv}]
Lemmas~\ref{lem-rw-conv-cone}, \ref{lem-rw-conv-sphere}, and~\ref{lem-rw-conv-disk} give the theorem statement in the quantum cone, quantum sphere, and quantum disk cases, respectively. The case of the quantum wedge follows from the case of the quantum disk and the same argument as in the proof of Lemma~\ref{lem-rw-conv-cone}. 
\end{proof}

\subsection{Proof of Tutte embedding convergence result}
\label{sec-tutte-conv-proof}

\begin{proof}[Proof of Theorem~\ref{thm-tutte-conv}]
Let $(\BB D , h , 0,1)$ be a quantum disk with fixed boundary length and area, with one marked boundary point and one marked interior point.
Let $D_h$ and $\mu_h$ be the $\sqrt{8/3}$-LQG metric and area measure and let $\xi_h$ be the path which traverses $\bdy\BB D$ counterclockwise from 1 to 1 in such a way that it traverses one unit of $\sqrt{8/3}$-LQG length in one unit of time. 
By~\cite[Corollary 1.5]{lqg-tbm2}, we know that the curve-decorated metric measure space $(\ol{\BB D} , D_h , \mu_h , \xi_h )$ is a Brownian disk with unit area and boundary length.
Furthermore, by the definition of a marked quantum disk, if we condition on this curve-decorated metric measure space then the marked point 0 is a uniform sample from $\mu_h$.

For $\lambda > 0$, define the Poisson point process $\mcl P^\lambda$, the Voronoi tessellation $\mcl H^\lambda$, and the Tutte embedding $\Phi^\lambda : \mcl P^\lambda \rta \ol{\BB D}$ as in the discussion just above Theorem~\ref{thm-tutte-conv} for the Brownian disk $(\mcl X , D , \mu , \xi) = (\ol{\BB D} , D_h , \mu_h , \xi_h )$.
Note that here the space $\mcl X$ is identified with $\ol{\BB D}$, so in particular $\mcl P^\lambda\subset \ol{\BB D}$.  

We will now argue that 
\eqb \label{eqn-tutte-conv}
\max_{z\in\mcl P^\lambda} |\Phi^\lambda(z) - z| \rta 0 ,\quad \text{in probability as $\lambda\rta\infty$}. 
\eqe
Indeed, Theorem~\ref{thm-rw-conv} implies that the maximum over all vertices $z \in \mcl P_h^\lambda$ of the Prokhorov distance between the Euclidean harmonic measure on $\bdy\BB D$ as viewed from $z$ and the $\mcl P^\lambda$-harmonic measure on $\bdy\mcl P^\lambda$ as viewed from $z$ tends to zero in probability as $n\rta\infty$. 
From this and the definition of $\Phi^\lambda$, we get~\eqref{eqn-tutte-conv}.

The first two convergence statements in the theorem statement are immediate from~\eqref{eqn-tutte-conv} (for the convergence of re-scaled counting measure, we use that $\lambda^{-1}$ times the counting measure on $\mcl P^\lambda$ converges in probability to $\mu_h$ since the intensity measure of $\mcl P^\lambda$ is $\lambda \mu_h$).
The convergence statement for the random walk on $\mcl P^\lambda$ follows from~\eqref{eqn-tutte-conv} and Theorem~\ref{thm-rw-conv}. 
\end{proof}

\subsection{A reformulation of the finite expectation hypothesis}
\label{sec-finite-expectation}

As in Section~\ref{sec-0cone}, let $(\BB C , h , 0, \infty)$ be a 0-quantum cone and write $\mcl P = \mcl P_h^1$ and $\mcl H =\mcl H_h^1$ for its associated Poisson point process and collection of Voronoi cells with $\lambda = 1$. 
Proving Proposition~\ref{prop-cell-moment} (i.e., the finite expectation hypothesis in Theorem~\ref{thm-general-clt-uniform} for $\mcl H$) directly turns out to be difficult since Voronoi cells depend on the field in a rather delicate way, so it is not clear how to lower-bound the Lebesgue measure of the origin-containing cell $H_0$. 
Instead, we will use the following lemma which allows us to lower-bound the Lebesgue measure of an LQG metric ball instead.

\begin{lem} \label{lem-ball-cell-moment}
For a Voronoi cell $H\in\mcl H$, we write $B_H$ for the smallest $D_h$-metric ball centered at the center point of $H$ (i.e., the point of $\mcl P$ which is in $H$) which contains $H$. We have
\eqb \label{eqn-ball-cell-moment}
\BB E\left[ \frac{\op{diam}(H_0)^2}{\op{area}(H_0)} \op{deg}(H_0) \right] 
 = \BB E\left[  \sum_{H\in\mcl H : 0 \in B_H} \frac{\op{diam}(H)^2 \op{deg}(H)}{\op{area}(B_H)}   \right]  .
\eqe
\end{lem} 

Lemma~\ref{lem-ball-cell-moment} is essentially a consequence of the ``mass transport" definition of translation invariance modulo scaling in~\cite[Definition 1.2]{gms-random-walk}. However, the balls $B_H$ are not functions of the cell configuration $\mcl H$ itself (they depend on additional randomness from the field) so we will need a trivial reformulation of the mass transport condition which allows for this. 

A \emph{decorated cell configuration} is a cell configuration $\mcl H$ together with a compact set $K_H\subset \BB C$ associated with each cell $H\in\mcl H$. 
We can define a topology on the space of decorated cell configurations by the obvious extension of~\eqref{eqn-cell-metric}: 
\allb \label{eqn-decorated-cell-metric}
&\BB d^{\op{DCC}}\left( (\mcl H , \{K_H\}_{H\in\mcl H} ) ,(\mcl H' , \{K_{H'}'\}_{H'\in\mcl H'} ) \right) \notag \\ 
&\qquad := \int_0^\infty e^{-r} \wedge \inf_{f_r} \big\{ \max_{z\in \BB C} |z - f_r(z)|   
  + \max_{H \in  \mcl H(B_r(0))} \BB d^{\op{Haus}}(K_{H} , K_{f_r(H)}' ) \big\} \,dr
\alle  
where $\BB d^{\op{Haus}}$ denotes the Hausdorff distance and each of the infima is over all homeomorphisms $f_r : \BB C\rta \BB C$ such that $f_r$ takes each cell in $\mcl H(B_r(0))$ to a cell in $\mcl H'(B_r(0))$ and preserves the adjacency relation, and $f_r^{-1}$ does the same with $\mcl H$ and $\mcl H'$ reversed.

\begin{defn} \label{def-decorated-tims}
We say that a random decorated cell configuration $(\mcl H , \{K_H\}_{H\in\mcl H})$ is \emph{translation invariant modulo scaling} if it satisfies the following obvious extension of the definition of translation invariance modulo scaling for cell configurations.
There is a (possibly random and $(\mcl H , \{K_H\}_{H\in\mcl H})$-dependent) increasing sequence of open sets $U_j \subset \BB C$, each of which is either a square or a disk, whose union is all of $\BB C$ such that the following is true. Conditional on $(\mcl H , \{K_H\}_{H\in\mcl H})$ and $U_j$, let $z_j$ for $j\in\BB N$ be sampled uniformly from Lebesgue measure on $U_j$. Then there are random numbers $C_j >0$ (possibly depending on $(\mcl H , \{K_H\}_{H\in\mcl H})$ and $z_j$) such that 
\eqbn
\left( C_j(\mcl H-z_j) , \{ C_j (K_H - z_j) \}_{H\in\mcl H}\right) \rta  (\mcl H , \{K_H\}_{H\in\mcl H}) 
\eqen
in law with respect to the metric~\eqref{eqn-decorated-cell-metric}.
\end{defn} 

Exactly as in~\cite[Definition 1.2]{gms-random-walk}, one can formulate various equivalent definitions of translation invariance modulo scaling for cell configurations and prove that these definitions are equivalent via exactly the same arguments as in~\cite[Appendix A]{gms-random-walk}. 
For our purposes, we will need the ``mass transport" formulation of translation invariance modulo scaling for decorated cell configurations.

\begin{lem}[Mass transport condition] \label{lem-mass-transport}
A random decorated cell configuration is $(\mcl H , \{K_H\}_{H\in\mcl H})$ is translation invariant modulo scaling in the sense of Definition~\ref{def-decorated-tims} if and only if it satisfies the following condition. 
Suppose that $F(\mcl H , \{K_H\}_{H\in\mcl H} , x,y)$ is a non-negative measurable function on the space of decorated cell configurations with two marked points in $\BB C$ such that $F$ is covariant with respect to dilations and translations of the plane in the sense that for each $C>0$ and $z\in\BB C$, 
\eqb \label{eqn-mass-transport-commutation}
F\left( C(\mcl H -z) , \{C(K_H-z) \}_{H\in\mcl H} , C(x-z) , C(y-z)  ) \right) = C^{-2} F\left( \mcl H , \{K_H\}_{H\in\mcl H} , x,y \right) .
\eqe
Then
\eqb \label{eqn-mass-transport} 
\BB E\left[ \int_{\BB C} F\left(\mcl H , \{K_H\}_{H\in\mcl H} , x, 0 \right)\, dx \right] = \BB E\left[ \int_{\BB C} F\left(\mcl H , \{K_H\}_{H\in\mcl H} , 0, y \right) \,dy  \right].
\eqe  
\end{lem}
\begin{proof}
This follows from exactly the same argument used for undecorated cell configurations in~\cite[Appendix A]{gms-random-walk}. 
\end{proof}
 
\begin{proof}[Proof of Lemma~\ref{lem-ball-cell-moment}]
For a decorated cell configuration $(\mcl H , \{K_H\}_{H\in\mcl H})$ and $x,y\in\BB C$, define 
\eqbn
F\left( \mcl H , \{K_H\}_{H\in\mcl H} ,  x ,y \right) := \frac{\op{diam}(H_y)^2  \op{deg}(H_y) }{ \op{area}(H_y)  \op{area}(K_{H_y})  } \BB 1_{x \in K_{H_y}} .
\eqen
Obviously, this choice of $F$ satisfies the condition~\eqref{eqn-mass-transport-commutation}. 

Now let $\mcl H$ be the particular cell configuration consisting of Voronoi cells on the 0-quantum cone.
It is easily verified that, with $B_H$ as in the statement of the lemma, the decorated cell configuration $(\mcl H , \{B_H\}_{H\in\mcl H})$ is translation invariant modulo scaling in the sense of Definition~\ref{def-decorated-tims}. 
By Lemma~\ref{lem-mass-transport}, we therefore have
\eqb \label{eqn-ball-mass-transport}
\BB E\left[ \int_{\BB C} F\left(\mcl H , \{B_H\}_{H\in\mcl H} , x , 0\right) \,dx  \right]  = \BB E\left[ \int_{\BB C} F\left(\mcl H , \{B_H\}_{H\in\mcl H} , 0 , y\right) \,dy  \right]  .
\eqe 

Clearly, 
\eqb \label{eqn-mass-transport-int1}
\int_{\BB C} F\left(\mcl H , \{B_H\}_{H\in\mcl H} , x , 0\right) \,dx 
= \frac{\op{diam}(H_0)^2}{\op{area}(H_0)} \op{deg}(H_0) .
\eqe 
By breaking up the integral into a sum of the integrals over each of the cells $H\in\mcl H$, we get
\eqb \label{eqn-mass-transport-int2}
\int_{\BB C} F\left(\mcl H , \{B_H\}_{H\in\mcl H} , 0, y\right) \,dy 
= \sum_{H\in\mcl H : 0 \in B_H}   \frac{\op{diam}(H )^2  \op{deg}(H ) }{  \op{area}(B_{H })  }  .
\eqe 
Plugging~\eqref{eqn-mass-transport-int1} and~\eqref{eqn-mass-transport-int2} into~\eqref{eqn-ball-mass-transport} gives~\eqref{eqn-ball-cell-moment}
\end{proof}

In light of Lemma~\ref{lem-ball-cell-moment}, we only need to prove that the expectation on the right side of~\eqref{eqn-ball-cell-moment} is finite.
Actually, we will prove the following much stronger statement. 

\begin{prop} \label{prop-ball-moment}
For each $p  >0$, we have
\eqb \label{eqn-ball-moment}
\BB E\left[ \left( \sum_{H\in\mcl H : 0 \in B_H}   \frac{\op{diam}(H )^2  \op{deg}(H ) }{  \op{area}(B_{H })  }  \right)^p \right] < \infty .
\eqe
\end{prop}

We emphasize that Proposition~\ref{prop-ball-moment} does \emph{not} imply that $ \frac{\op{diam}(H_0)^2}{\op{area}(H_0)} \op{deg}(H_0) $ has finite moments of all positive orders (rather, we only get that it has a finite first moment) since Lemma~\ref{lem-ball-cell-moment} does not allow us to compare moments of order greater than 1. 
The rest of the paper is devoted to the proof of Proposition~\ref{prop-ball-moment}.

\section{Estimates for LQG metric balls}
\label{sec-ball-estimates}

The goal of this section is to establish Proposition~\ref{prop-ball-moment}, which together with Lemma~\ref{lem-ball-cell-moment} will conclude the proof of our main results.
Along the way, we will establish a number of estimates for LQG metric balls which are of independent interest (see in particular Propositions~\ref{prop-lqg-ball-upper}, \ref{prop-lqg-ball-swallow}, \ref{prop-lqg-ball-union}, and~\ref{prop-ball-vol}).

We start out in Section~\ref{sec-wn-decomp} by introducing two random distributions on $\BB C$ which are defined using the white noise decomposition of the GFF and which have certain nice properties that the GFF itself does not. 
The first of these distributions, which we call $\wh h$, possesses exact scale and translation invariance properties (not just scale and translation invariance modulo additive constant, like the whole-plane GFF). 
The second, which we call $\wh h^\tr$, has the property that its restrictions to two sets at distance at least $1/5$ from each other are independent. 
We then state a general lemma (Lemma~\ref{lem-gff-compare}) which allows us to compare $\wh h$, $\wh h^\tr$, and the whole-plane GFF.

In Section~\ref{sec-eucl-ball}, we prove an upper bound for the LQG distance across a Euclidean annulus (Proposition~\ref{prop-lqg-ball-upper}) as well as estimates to the effect that an LQG metric ball $B$ typically contains a Euclidean ball of radius comparable to the Euclidean diameter of $B$ (Propositions~\ref{prop-lqg-ball-swallow} and~\ref{prop-lqg-ball-union}). These are proven using percolation arguments which rely crucially on the local independence property of $\wh h^\tr$.
In Section~\ref{sec-ball-vol}, we prove that the LQG area of an LQG metric ball of radius $r$ is tightly concentrated around $r^{4+o_r(1)}$ (Proposition~\ref{prop-ball-vol}).
This is proven by starting with known estimates for metric balls in the Brownian map, then transferring to the GFF using the equivalence of Brownian surfaces and LQG surfaces~\cite{lqg-tbm2}, and finally using the local independent property of $\wh h^\tr$ to go from events of high probability to events of superpolynomially high probability.

In Section~\ref{sec-0cone-estimates}, we transfer the estimates of the preceding subsections from the whole-plane GFF to the 0-quantum cone. 
In Section~\ref{sec-moment-proof}, we conclude the proof of Proposition~\ref{prop-ball-moment}. 

\subsection{White-noise approximation of the Gaussian free field}
\label{sec-wn-decomp}

In this subsection we will introduce various white-noise approximations of the GFF which are often more convenient to work with than the GFF itself. Similar approximations to the ones used here were also studied in~\cite{ding-goswami-watabiki,dzz-heat-kernel,dg-lqg-dim}. Let $W$ be a space-time white noise on $\BB C\times [0,\infty)$, i.e., $\{(W,f) : f\in L^2(\BB C\times [0,\infty))\}$ is a centered Gaussian process with covariances $\BB E[(W,f) (W,g) ]  = \int_\BB C\int_0^\infty f(z,s) g(z,s) \, ds\,dz$. For $f\in L^2(\BB C\times [0,\infty))$ and Borel measurable sets $A\subset\BB C$ and $I\subset [0,\infty)$, 
we slightly abuse notation by writing 
\eqbn
\int_B\int_I f(z,s) \, W(ds,dz) := (W , f \BB 1_{A\times I} ) .
\eqen

For an open set $U \subset \BB C$, we write $p_U(s ; z,w)$ for the transition density of Brownian motion killed upon exiting $U$, so that for $s\geq 0$, $z\in \BB C$, and $A\subset \ol U$, the integral $\int_A p_U(s;z,w) \,dw$ gives the probability that a standard planar Brownian motion $\mcl B$ started from $z$ satisfies $\mcl B([0,s]) \subset U$ and $\mcl B_s \in A$. We also write 
\eqbn
p(s;z,w) := p_{\BB C}(s;z,w) =  \frac{1}{2\pi s} \exp\left( - \frac{|z-w|^2}{2s} \right) .
\eqen  

We define the centered Gaussian process
\eqb \label{eqn-wn-decomp}
\wh h_t (z) := \sqrt\pi \int_{\BB C} \int_{t^2}^1 p (s/2 ;z,w) \, W(ds,dw)  ,\quad \forall t \in [0,1] , \quad \forall z\in \BB C .
\eqe  
We set $\wh h : =\wh h_0$. 
By~\cite[Lemma 3.1]{ding-goswami-watabiki} and Kolmogorov's criterion, each $\wh h_t$ for $t\in (0,1]$ admits a continuous modification. Henceforth whenever we work with $\wh h_t$ we will assume that it has been replaced by such a modification. 
The process $\wh h$ does not admit a continuous modification, but its integral against any smooth compactly supported test function has finite variance, so it makes sense as a distribution. 
We record for reference the formula
\eqb \label{eqn-wn-var}
\op{Var}\left( \wh h_{\wt t}(z) - \wh h_t(z) \right) = \log (\wt t/t),\quad\forall z \in \BB C, \quad\forall 0 < t <\wt t < 1 ,
\eqe
which is immediate from~\eqref{eqn-wn-decomp}.

The distribution $\wh h$ is often more convenient to work with than the GFF thanks to the following symmetries, which are immediate from the definition. 
\begin{itemize}
\item \textit{Rotation/translation/reflection invariance.} The law of $\wh h$ is invariant with respect to rotation, translation, and reflection of the plane.
\item \textit{Scale invariance.} For $\delta \in (0,1]$, one has $ \wh h(\delta \cdot) - \wh h_\delta(\delta \cdot)  \eqD  \wh h$. 
\item \textit{Independent increments.} For $\delta \in (0,1)$, $\wh h  - \wh h_\delta $ is independent from $\wh h_\delta$. 
\end{itemize}
   
One property which $\wh h$ does not possess is spatial independence. To get around this, we will sometimes work with a truncated variant of $\wh h $ where we only integrate over a ball of finite radius. We define
\eqb \label{eqn-wn-decomp-tr}
\wh h^\tr(z) := \sqrt\pi \int_0^1 \int_{\BB C} p_{B_{1/10}(z)}(s/2; z,w)   \, W(dw,dt)  
\eqe 
and we interpret $\wh h^\tr$ as a random distribution.  
The key property enjoyed by $\wh h^\tr$ is spatial independence: if $A,B\subset \BB C$ with $\op{dist}(A,B) \geq 1/5$, then $ \wh h^\tr|_A $ and $\wh h^\tr|_B$ are independent. Indeed, this is because $\wh h^\tr|_A $ and $\wh h^\tr|_B $ are determined by the restrictions of the white noise $W$ to the disjoint sets $B_{1/10}(A) \times \BB R_+ $ and $B_{1/10}(B)\times \BB R_+ $, respectively.  
Unlike $\wh h$, the distribution $\wh h^\tr$ does not possess any sort of scale invariance but its law is still invariant with respect to rotations, translations, and reflections of $\BB C$.

The following lemma is proven using elementary calculations for the transition density $p_U(t;z,w)$ together with the Kolmogorov continuity theorem (see, e.g.,~\cite[Lemma 3.1]{dg-lqg-dim}). It will allow us to use $\wh h^\tr$ or $\wh h$ in place of the GFF in many of our arguments. 

\begin{lem} \label{lem-gff-compare}

For any compact set $K\subset\BB C$, there is a coupling $(h , \wh h , \wh h^\tr)$ of a whole-plane GFF normalized so that $h_1(0) = 0$ and the distributions from~\eqref{eqn-wn-decomp} and~\eqref{eqn-wn-decomp-tr} such that the following is true. For $h^1,h^2 \in \{h ,  \wh h , \wh h^\tr\}$, the distribution $(h^1-h^2)|_K$ a.s.\ admits a continuous modification 
and there are constants $c_0,c_1 > 0$ depending only on $K$ such that for $A>1$, 
\eqb \label{eqn-gff-compare}
\BB P\left[\max_{z\in K} |(h^1-h^2)(z)| \leq A \right] \geq 1 - c_0 e^{-c_1 A^2} .
\eqe
In fact, in this coupling one can arrange so that $\wh h$ and $\wh h^\tr$ are defined using the same white noise $W$. 
\end{lem}

The existence of continuous modifications in Lemma~\ref{lem-gff-compare} allows us to define the $\sqrt{8/3}$-LQG metrics $D_{\wh h}$ and $D_{\wh h^\tr}$. Indeed, this is because we know how to define $D_{h + f}$ when $h$ is a GFF and $f$ is a continuous function (see the discussion just after Lemma~\ref{lem-metric-f}). 
Moreover, we get that $D_{\wh h}$ and $D_{\wh h^\tr}$ each a.s.\ induces the Euclidean topology on $\BB C$.  
Due to Lemma~\ref{lem-metric-f}, the estimate~\eqref{eqn-gff-compare} will allow us to compare $D_{\wh h}$ and $D_{\wh h^\tr}$ to the $\sqrt{8/3}$-LQG metrics induced by a GFF. 

The following lemma will be used when we apply the scaling property of $\wh h$. 

\begin{lem} \label{lem-use-btis} 
For each bounded domain $U \subset \BB C$, there are constants $c_0 , c_1 > 0$ depending only on $U$ such that for $\delta\in (0,1)$ and $C>0$, 
\eqb \label{eqn-use-btis}
\BB P\left[ \max_{z,w \in U : |z-w| \leq \delta} |\wh h_\delta(z) -\wh h_\delta(w)| \leq C \right] \geq 1 - c_0 \delta^{-2} e^{-c_1 C^2  } .
\eqe 
\end{lem}
\begin{proof}
It is easily seen (see~\cite[Lemma 3.1]{ding-goswami-watabiki}) that for $\delta > 0$, $\op{Var}(\wh h_\delta(z) - \wh h_\delta(w)) \leq |z-w|^2/\delta^2$, which is of course smaller than $|z-w| /\delta$ whenever $|z-w| \leq \delta$.  
By Fernique's criterion~\cite{fernique-criterion} (see~\cite[Theorem 4.1]{adler-gaussian} or~\cite[Lemma 2.3]{dzz-heat-kernel} for the version we use here), we find that for each square $S\subset\BB C$ with side length $\delta/2$, 
\eqbn
\BB E\left[ \max_{z,w\in S} |\wh h_\delta(z) - \wh h_\delta(w)| \right] \leq A ,
\eqen
for a universal constant $A>0$.
Combining this with the Borell-TIS inequality~\cite{borell-tis1,borell-tis2} (see, e.g.,~\cite[Theorem 2.1.1]{adler-taylor-fields}), we get that for each such square $S$, 
\eqbn
\BB P\left[ \max_{z,w\in S} |\wh h_\delta(z) - \wh h_\delta(w)| \leq C \right]  \geq 1 - c_0 e^{- c_1 C^2  }  
\eqen
for universal constants $c_0 ,c_1 > 0$. 
A union bound over $O_\delta(\delta^{-2})$ such squares whose union contains $U$ concludes the proof. 
\end{proof}

\subsection{Comparing LQG metric balls and Euclidean balls}
\label{sec-eucl-ball}

Throughout this subsection, we let $h$ be a whole-plane GFF normalized so that $h_1(0) = 0$. 
We will prove the following three propositions, which relate $D_h$-metric balls and Euclidean balls. 
Our first estimate implies in particular that a $D_h$-metric ball is extremely unlikely to have an unusually large Euclidean diameter. 
This estimate is related to the fact that $\mu_h(\BB D)$ has negative moments of all orders (see~\cite[Lemma 4.5]{shef-kpz} or~\cite[Theorem 2.12]{rhodes-vargas-review}) but is proven in a very different way. 

\begin{prop} \label{prop-lqg-ball-upper}
For each fixed $\rho \in (0,1)$, it holds with superpolynomially high probability as $\ep\rta 0$ that 
\eqb \label{eqn-lqg-ball-upper}
 D_h\left(  B_\rho(0) , \bdy \BB D \right) \geq \ep  .
\eqe
\end{prop}

We next state two closely related estimates to the effect that a $D_h$-metric ball typically contains a Euclidean ball of radius comparable to its Euclidean diameter.
These are analogs for $D_h$-metric balls of estimates for space-filling SLE cells from~\cite[Section 3]{ghm-kpz} and~\cite[Section 4]{gms-tutte}, and will be used to control the ratio $\op{diam}(H_0)^2/\op{area}(B_{H_0})$ appearing in Proposition~\ref{prop-ball-moment}. 

\begin{prop} \label{prop-lqg-ball-swallow}
With superpolynomially high probability as $\ep\rta 0$, each $D_h$-metric ball which intersects both $\bdy B_\rho(0)$ and $\bdy \BB D$ contains a Euclidean ball of radius at least $\ep$. 
\end{prop}

\begin{prop} \label{prop-lqg-ball-union}
Fix $\zeta \in (0,1)$. 
With superpolynomially high probability as $\delta \rta 0$, each $D_h$-metric ball $B \subset \BB D$ with $\op{diam}(B) \leq \delta$ contains a Euclidean ball of radius at least $\op{diam}(B)^{1+\zeta}$. 
\end{prop}

We will prove Propositions~\ref{prop-lqg-ball-upper} and~\ref{prop-lqg-ball-swallow} simultaneously using a percolation argument which is similar to ones from~\cite{ding-dunlap-lqg-fpp,ding-goswami-watabiki,dzz-heat-kernel,dg-lqg-dim,df-lqg-metric,ding-dunlap-lgd,dddf-lfpp}.  Proposition~\ref{prop-lqg-ball-union} will be deduced from Proposition~\ref{prop-lqg-ball-swallow} and a union bound. 

For $\rho  >0$, define the square annulus $\mcl A_\rho$ and its inner and outer boundaries by
\eqb \label{eqn-square-annulus-def}
\mcl A_\rho :=  [-2\rho, 2\rho]^2 \setminus (-\rho,\rho)^2 ,
\quad  \bdy_{\op{in}}\mcl A_{\rho} := \bdy ([-\rho,\rho]^2) ,\quad \op{and} \quad
 \bdy_{\op{out}}\mcl A_{\rho} := \bdy([-2\rho ,2\rho]^2)  .
\eqe 
The main step in the proof of the above propositions is Lemma~\ref{lem-annulus-perc-tr} just below.
For $n\in\BB N$, we consider the restriction to $\mcl A_n$ of the truncated white noise field $\wh h^\tr$ of~\eqref{eqn-wn-decomp-tr}. 
The reason for considering $\mcl A_n$ instead of $\mcl A_1$, say, is that two sets need to be at Euclidean distance at least $1/5$ from each other for the restrictions of $\wh h^\tr$ to be independent and we want to define lots of independent events. 

Basic properties of the $\sqrt{8/3}$-LQG metric show that for each $1\times 1$ square $S\subset \mcl A_n$, it holds with probability tending to 1 as $C\rta\infty$ that the $D_{\wh h^\tr}$-distance from the boundary of the $1/2$-neighborhood $\bdy B_{1/2}(S)$ to $S$ is at least $1/C$ and each Euclidean ball of radius $e^{-C n^{2/3}}$ which intersects $S$ has $D_{\wh h^\tr }$-diameter at most $e^{-n^{2/3}}$ (this last condition would also hold with $e^{-C n^{2/3}}$ and $e^{-n^{2/3}}$ replaced by, e.g., $n^{-C}$ and $n^{-1}$ or $1$ and $1/C$, but we use $e^{-n^{2/3}}$ since we will get an error of order $e^{n^{1/2}}$ in Lemma~\ref{lem-annulus-perc-wn} below).
The restrictions of $\wh h^\tr$ to squares which lie at distance at least $1/5$ from one another are independent, so the adjacency graph of ``good" squares which satisfy the above properties looks like a very supercritical percolation on $\BB Z^2$ when $C$ is large. 
Hence with exponentially high probability in $n$, there is path of such good squares which separates the inner and outer boundaries of $\mcl A_n$. 
This implies analogs of Propositions~\ref{prop-lqg-ball-upper} and~\ref{prop-lqg-ball-swallow} for $\wh h^\tr|_{\mcl A_n}$.  In Lemma~\ref{lem-annulus-perc-wn}, we set $n \asymp (\log\ep^{-1})^{3/2}$ and transfer from $\wh h^\tr|_{\mcl A_n}$ to $\wh h|_{\mcl A_\rho}$ (and thereby to $h|_{\mcl A_\rho}$) using Lemma~\ref{lem-gff-compare} and the scale invariance properties of the field $\wh h$. 

\begin{lem} \label{lem-annulus-perc-tr} 
Define $\mcl A_n$ for $n\in\BB N$ as in~\eqref{eqn-square-annulus-def}. 
There are universal constants $a_0 , a_1   > 0$ and $C>1$ such that for each $n\in\BB N$, it holds with probability at least $1-a_0 e^{-a_1 n}$ that the following is true.
\begin{enumerate}
\item The $D_{\wh h^\tr}$-distance from $ \bdy_{\op{in}}\mcl A_n$ to $\bdy_{\op{out}}\mcl A_n$ is at least $1/C$. \label{item-annulus-perc-dist}
\item Each path from  $ \bdy_{\op{in}}\mcl A_n$ to $\bdy_{\op{out}}\mcl A_n$ intersects a Euclidean ball with Euclidean radius $e^{-C n^{2/3}}$ and $D_{\wh h^\tr}$-diameter at most $e^{-n^{2/3}}$. \label{item-annulus-perc-ball}
\end{enumerate} 
\end{lem}  

See Figure~\ref{fig-annulus-perc} for an illustration of the statement and proof of Lemma~\ref{lem-annulus-perc-tr}. 
We will eventually apply the lemma with $n \asymp (\log\ep^{-1})^{3/2}$, so that $e^{-n}$ is smaller than any power of $\ep$ and $e^{-n^{2/3}} \approx \ep$.

\begin{figure}[ht!]
\begin{center}
\includegraphics[scale=.65]{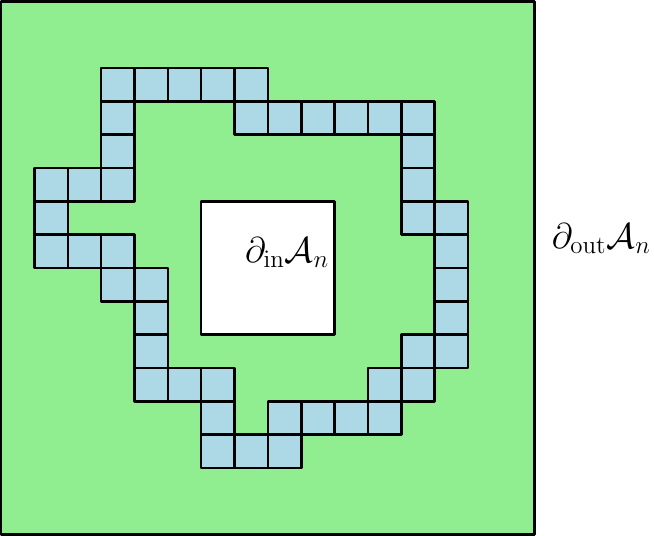} 
\caption{\label{fig-annulus-perc} The square annulus $\mcl A_n$ of~\eqref{eqn-square-annulus-def} is shown in light green. To prove Lemma~\ref{lem-annulus-perc-tr}, we use a percolation argument (based on the local independence property of $\wh h^\tr$) to show that with extremely high probability, we can find a collection of $1\times 1$ squares $S$ (light blue) whose union disconnects the inner and outer boundaries of $\mcl A_n$ and such that each path which crosses one of these squares has to have $D_{\wh h^\tr}$-length at least $1/C$ and has to hit a Euclidean ball of radius $e^{-C n^{2/3}}$ and $D_{\wh h^\tr}$-diameter at most $e^{-n^{2/3}}$. 
After re-scaling by $1/n$, this statement is used to prove Propositions~\ref{prop-lqg-ball-upper}, \ref{prop-lqg-ball-swallow}, and~\ref{prop-lqg-ball-union}. 
}
\end{center}
\end{figure}

\begin{proof}[Proof of Lemma~\ref{lem-annulus-perc-tr}]
Let $p\in (0,1)$ be a small universal constant to be chosen later, in a universal manner. 
For $n\in\BB N$, let $\mcl S(\mcl A_n)$ be the set of unit side length squares with corners in $\BB Z^2$ whose Euclidean $1$-neighborhood satisfies $B_1(S) \subset \mcl A_n$. For $S\in \mcl S(\mcl A_n)$ and $C>1$, let $E_S(C)$ be the event that the following is true. 
\begin{enumerate}
\item $D_{\wh h^\tr }  \left( S , \bdy B_{1/2}(S)  \right) \geq 1/C$.
\item Each Euclidean ball of radius $e^{-C n^{2/3}}$ which intersects $S$ has $D_{\wh h^\tr|_{B_{1/2}(S)}}$-diameter at most $e^{-n^{2/3}}$. 
\end{enumerate}
Then $E_S(C)$ is a.s.\ determined by $\wh h^\tr|_{B_{1/2}(S)}$. By~\cite[Theorem 1.2]{lqg-tbm2} and Lemma~\ref{lem-gff-compare} (see the discussion at the end of Section~\ref{sec-lqg-metric}), the identity map from $B_{1/2}(S)$, equipped with the Euclidean metric, to $B_{1/2}(S)$, equipped with $D_{\wh h^\tr|_{B_{1/2(S)}} }$, and its inverse are a.s.\ locally H\"older continuous. 
In particular, $D_{\wh h^\tr|_{B_{1/2}(S)}}$ induces the same topology on $B_{1/2}(S)$ as the Euclidean metric.  
Since the law of $\wh h^\tr$ is invariant under spatial translation, it follows that there exists $C = C(p) > 1$ such that $\BB P[E_S(C)] \geq 1-p$ for every $S\in \mcl S(\mcl A_n)$. Henceforth fix such a $C$. 
  
View $\mcl S(\mcl A_n)$ as a graph with two squares considered to be adjacent if they share an edge. 
We claim that if $p$ is chosen sufficiently small, then for appropriate constants $a_0,a_1 > 0$ as in the statement of the lemma, it holds for each $n\in\BB N$ that with probability at least $1- a_0 e^{-a_1 n}$, we can find a path $\mcl P$ of squares in $\mcl S(\mcl A_n)$ which disconnects $\bdy_{\op{in}}\mcl A_n$ from $\bdy_{\op{out}} \mcl A_n$ such that $E_S(C)$ occurs for each $S\in\mcl P$. 

Assume the claim for the moment. If a path $\mcl P$ as in the claim exists, then each Euclidean path from $\bdy_{\op{in}}\mcl A_n$ to $\bdy_{\op{out}} \mcl A_n$ must pass through one of the squares $S\in\mcl P$. Since $E_S(C)$ occurs for each such square, each such path must hit a Euclidean ball of radius $e^{-C n^{2/3}}$ centered at a point of $S$ which has $D_{\wh h^\tr} \leq  D_{\wh h^\tr|_{B_{1/2}(S)}}$-diameter at most $e^{-n^{2/3}}$, i.e., condition~\ref{item-annulus-perc-ball} in the lemma statement holds. Furthermore, since $B_{1/2}(S) \subset\mcl A_n$ for each $S\in \mcl S(\mcl A_n)$, any path from $\bdy_{\op{in}}\mcl A_n$ to $\bdy_{\op{out}} \mcl A_n$ must cross one of the annuli $B_{1/2}(S) \setminus S $ for some $S\in \mcl P$. Since $E_S(C)$ occurs for each such $S$, condition~\ref{item-annulus-perc-dist} in the lemma statement holds. 
   
It remains only to prove the claim.
Let $\mcl S^*(\mcl A_n)$ be the graph whose squares are the same as the squares of $\mcl S(\mcl A_n)$, but with two squares considered to be adjacent if they share a corner or an edge, instead of only considering squares to be adjacent if they share an edge.
We define the inner and outer boundaries of $\mcl S^*(\mcl A_n)$ to be the set of squares which lie at Euclidean distance 1 from $\bdy_{\op{in}}\mcl A_n$ and $\bdy_{\op{out}}\mcl A_n$, respectively (recall that squares in $\mcl S(\mcl A_n)$ satisfy $B_1(S)\subset \mcl A_n$).
By planar duality, it suffices to show that if $p ,a_0, a_1$ are chosen appropriately, then it holds with probability at least $1-a_0 e^{-a_1 n}$ that there does \emph{not} exist a simple path in $\mcl S^*(\mcl R_n)$ from the inner boundary of $\mcl A_n$ to the outer boundary of $\mcl A_n$ consisting of squares for which $E_S(C)$ does not occur. This will be proven by a standard argument for subcritical percolation.

By the definition~\eqref{eqn-wn-decomp-tr} of $\wh h^{\tr}$, the event $E_S (C)$ is a.s.\ determined by the restriction of the white noise $W$ to $B_{1}(S) \times \BB R_+ $. In particular, $E_S(C)$ and $E_{\wt S}(C)$ are independent whenever $B_1(S)\cap B_1(\wt S) = \emptyset$. 
For each fixed deterministic simple path $ P$ in $\mcl S^*(\mcl R_n)$, we can find a set of at least $|P|/100$ squares hit by $P$ for which the neighborhoods $B_1(S)$ are disjoint. 
Since the events $E_S(C)$ for these $|P|/100$ squares are independent and each has probability at least $1-p$, the probability that $E_S(C)$ fails to occur for every square in $ P$ is at most $p^{| P|/100}$.

We now take a union bound over all simple paths $P$ in $\mcl S^*(\mcl A_n)$ connecting the inner and outer boundaries. For $k \in [n,16n^2]_{\BB Z}$, the number of such paths with $|P| = k$ is at most $ 4 n 8^{k }$ since there are $4 n$ possible initial squares along the inner boundary of $\mcl A_n$ and 8 choices for each step of the path. Combining this with the estimate in the preceding paragraph, we find that the probability of an inner-outer crossing in $\mcl S^*(\mcl A_n)$ consisting of squares for which $E_S(C)$ does not occur is at most
\eqbn
  4 n \sum_{k=n}^{16n^2} p^{k/100} 8^{k+1} , 
\eqen
which is bounded above by an exponential function of $n$ provided we take $p < 8^{-100}$.  
\end{proof}

We now transfer from $\wh h^\tr|_{\mcl A_n}$ to $\wh h|_{\mcl A_\rho}$.

\begin{lem} \label{lem-annulus-perc-wn} 
There is a universal constant $C >1$ such that for each $\rho \in (0,1)$, it holds with superpolynomially high probability as $\ep\rta 0$ that the following is true. 
\begin{enumerate}
\item The $D_{\wh h }$-distance from $ \bdy_{\op{in}}\mcl A_\rho$ to $\bdy_{\op{out}}\mcl A_\rho$ is at least $\ep^{1/2}$. \label{item-annulus-perc-dist-wn}
\item Each path from  $ \bdy_{\op{in}}\mcl A_\rho$ to $\bdy_{\op{out}}\mcl A_\rho$ intersects a Euclidean ball with Euclidean radius at least $ \ep^C $ and $D_{\wh h }$-diameter at most $\ep$.  \label{item-annulus-perc-ball-wn}
\end{enumerate}
\end{lem}  
\begin{proof} 
We will apply Lemma~\ref{lem-annulus-perc-tr} with $n \asymp (\log \ep^{-1})^{3/2}$ together with a scaling argument. 
We first establish an estimate for $D_{\wh h}$-distances in $\mcl A_n$. 
By Lemma~\ref{lem-gff-compare} and a union bound over $O_n(n^2)$ Euclidean balls of unit radius which cover $\mcl A_n$, we can find constants $c_0  ,c_1 > 0$ and a coupling of $\wh h$ and $\wh h^\tr$ such that
\eqbn
\BB P\left[ \max_{z\in \mcl A_n} |(\wh h - \wh h^\tr)(z) | \leq A \right] \geq 1 - c_0 n^2 e^{-c_1 A^2} ,\quad \forall A > 0. 
\eqen
If $\max_{z\in \mcl A_n} |(\wh h - \wh h^\tr)(z) | \leq A$, then by the scaling property of LQG distances (Lemma~\ref{lem-metric-f}), 
\eqbn
e^{-A/\sqrt 6} D_{\wh h^\tr}(z,w) \leq D_{\wh h}(z,w) \leq e^{A/\sqrt 6} D_{\wh h^\tr}(z,w) ,\quad \forall z,w\in \mcl A_n .
\eqen
Setting $A = \sqrt 6 n^{1/2} $ and applying Lemma~\ref{lem-annulus-perc-tr}, we see that there is a universal constant $C>1$ such that with exponentially high probability as $n\rta\infty$, the following is true.
\begin{enumerate}
\item The $D_{\wh h }$-distance from $ \bdy_{\op{in}}\mcl A_n$ to $\bdy_{\op{out}}\mcl A_n$ is at least $C^{-1} e^{- n^{1/2} }$.  
\item  Each path from  $ \bdy_{\op{in}}\mcl A_n$ to $\bdy_{\op{out}}\mcl A_n$ intersects a Euclidean ball with Euclidean radius $e^{-C n^{2/3}}$ and $D_{\wh h }$-diameter at most $e^{-n^{2/3} + O_n(n^{1/2})}$,  
\end{enumerate} 
with the rate of the $O_n(n^{1/2})$ universal.

We now use a scaling argument to transfer from $\mcl A_n$ to $\mcl A_\rho$. 
Recall that $(\wh h  - \wh h_{\rho/n})  ((\rho/n)\cdot)   \eqD \wh h$.
By the LQG coordinate change formula~\eqref{eqn-lqg-coord-metric} and Lemma~\ref{lem-metric-f},
\alb
&(\rho/n)^{-  Q / \sqrt 6 } \exp\left(  - \frac{1}{\sqrt 6} \max_{x \in \mcl A_n} \wh h_{\rho/n}(x) \right) D_{\wh h}(z,w)  \notag \\
&\qquad \leq D_{(\wh h  - \wh h_{\rho/n})  ((\rho/n)\cdot) }((n/\rho) z , (n/\rho) w ) \notag \\
&\qquad \qquad \leq (\rho/n)^{- Q / \sqrt 6} \exp\left(- \frac{1}{\sqrt 6}  \min_{x \in \mcl A_n} \wh h_{\rho/n}(x) \right) D_{\wh h}(z,w) ,
\quad \forall z,w \in \mcl A_\rho .
\ale
Choose a finite set $\mcl Z_n$ of $O_n(n^4)$ points $z\in\mcl A_n$ such that $\mcl A_n\subset \bigcup_{z\in\mcl Z_n} B_{\rho/n}(z)$. 
By Lemma~\ref{lem-use-btis} (applied with $\delta = \rho/n$ and $C=(1/2) n^{1/2}$), the Gaussian tail bound applied to the $O_n(n^4)$ centered Gaussian random variables $\wh h_{\rho/n}(z)$ for $z\in \mcl Z_n$, each of which has variance $\log(\rho/n)$, and a union bound, we can find constants $c_0' ,c_1' >0$, depending only on $\rho$, such that  
\eqbn
\BB P\left[ \max_{x\in \mcl A_n} |\wh h_{\rho/n}(x)| \leq n^{1/2}  \right] \geq 1 - c_0' n^4 e^{-c_1' n / \log n } .
\eqen 
Hence, with probability at least $1 -  c_0' n^4 e^{-c_1' n /\log n }$, 
\eqb \label{eqn-annulus-perc-scale}
e^{-O_n(n^{1/2})} D_{\wh h}(z,w) 
\leq D_{(\wh h  - \wh h_{\rho/n})  ((\rho/n)\cdot)}((n/\rho) z , (n/\rho) w ) 
\leq e^{O_n(n^{1/2})} D_{\wh h}(z,w) ,
\quad \forall z,w \in \mcl A_\rho ,
\eqe 
with the rate of the $O_n(n^{1/2})$ deterministic and depending only on $\rho$. 

We know that $(\wh h  - \wh h_{\rho/n})  ((\rho/n)\cdot) \eqD \wh h$, so by combining~\eqref{eqn-annulus-perc-scale} and the conclusion of the first paragraph with $(\wh h  - \wh h_{\rho/n})  ((\rho/n)\cdot)$ in place of $h$, we get that (after possibly shrinking $c_0'$ and $c_1'$) it holds with probability at least $1 -  c_0' n^4 e^{-c_1' n  }$ that the following is true.
\begin{enumerate}
\item The $D_{\wh h }$-distance from $ \bdy_{\op{in}}\mcl A_\rho$ to $\bdy_{\op{out}}\mcl A_\rho$ is at least $e^{-O_n(n^{1/2})} $.  
\item Each path from  $ \bdy_{\op{in}}\mcl A_\rho$ to $\bdy_{\op{out}}\mcl A_\rho$ intersects a Euclidean ball with Euclidean radius $(\rho/n) e^{-C n^{2/3}}$ and $D_{\wh h }$-diameter at most $e^{-n^{2/3} + O_n(n^{1/2})}$. 
\end{enumerate} 
We now choose $n  = \lfloor (\log \ep^{-1})^{3/2} \rfloor$. This makes it so that $n^4 e^{-c_1' n /\log n }$ decays faster than any positive power of $\ep$, $e^{-n^{1/2}}$ decays slower than any positive power of $\ep$, $(\rho/n) e^{-C n^{2/3}} = \ep^{C+o_\ep(1)}$, and $e^{-n^{2/3} + O_n(n^{1/2})} = \ep^{1 + o_\ep(1)}$. Making this choice of $n$ and possibly slightly adjusting $C$ and $\ep$ concludes the proof.  
\end{proof}

\begin{proof}[Proof of Proposition~\ref{prop-lqg-ball-upper}]
By Lemma~\ref{lem-gff-compare}, we can couple $h$ and $\wh h$ in such a way that $(h-\wh h)|_{\BB D}$ is a continuous function and with superpolynomially high probability as $\ep\rta 0$, one has $\max_{z\in \BB D} |(h -\wh h)(z)| \leq (\log\ep^{-1})^{2/3}$. Combining this with Lemma~\ref{lem-annulus-perc-wn} and the scaling property of LQG distances shows that for each fixed square annulus $A\subset \BB D$ such that the ratio of its inner and outer side lengths is 4, it holds with superpolynomially high probability as $\ep\rta 0$ (at a rate depending on $A$) that the $D_h$-distance between the inner and outer boundaries of $A$ is at least $\ep$. 
We can find finitely many such square annuli contained in $\BB D\setminus B_\rho(0)$ such that the union of their inner boundaries disconnects the inner and outer boundaries of $\BB D\setminus B_\rho(0)$. 
Each path between the inner and outer boundaries of $\BB D\setminus B_\rho(0)$ must cross between the inner and outer boundaries of one of these square annuli, so applying the preceding estimate once to each such annulus and taking a union bound concludes the proof. 
\end{proof}

\begin{proof}[Proof of Proposition~\ref{prop-lqg-ball-swallow}]
Via the same argument as in the proof of Proposition~\ref{prop-lqg-ball-upper}, Lemma~\ref{lem-annulus-perc-wn} implies that there is a universal constant $C>1$ such that with superpolynomially high probability as $\ep\rta 0$, each path from  $B_{\rho}(0)$ to $\bdy B_{(1+\rho)/2}(0)$ intersects a Euclidean ball with Euclidean radius at least $ \ep^C $ and $D_h$-diameter at most $\ep$. 
In particular, each $D_h$-ball $B$ which intersects both $\bdy B_\rho(0)$ and $\bdy \BB D$ intersects a Euclidean ball of radius at least $ \ep^C $ and $D_h$-diameter at most $\ep$ which is contained in $B_{(1+\rho)/2}(0)$.  
On the other hand, Proposition~\ref{prop-lqg-ball-upper} shows that with superpolynomially high probability as $\ep\rta 0$, the $D_h$-distance from $B_{(1+\rho)/2}(0)$ to $\bdy\BB D$ is at least $2\ep$, in which case the aforementioned Euclidean ball is contained in the $D_h$-metric ball $B$. 
Replacing $\ep^C$ by $\ep$ concludes the proof.
\end{proof}

\begin{proof}[Proof of Proposition~\ref{prop-lqg-ball-union}]
Observe that the conclusion of Proposition~\ref{prop-lqg-ball-swallow} does not depend on the choice of additive constant for $h$.
By the scale and translation invariance of the law of $h$, modulo additive constant, we see that Proposition~\ref{prop-lqg-ball-swallow} implies that for each $\rho\in (0,1)$, $r > 0$, and $z\in\BB C$, it holds with superpolynomially high probability as $\ep\rta 0$, at a rate which is uniform in $r$ and $z$, that each $D_h$-metric ball which intersects both $\bdy B_{\rho r}(z)$ and $\bdy B_r(z)$ contains a Euclidean ball of radius at least $\ep r$. 

By a union bound, with superpolynomially high probability as $\delta \rta 0$ that the following is true.
For each $k\in\BB N$ with $2^{-k} \leq \delta$ and each $z\in  \BB D\cap (2^{-100 k} \BB Z^2)$, each $D_h$-metric ball which intersects both $\bdy B_{2^{-k-1}}(z)$ and $\bdy B_{2^{-k}}(z)$ contains a Euclidean ball of radius at least $2^{-(1+\zeta/2) k}$. 
If $B\subset \BB D$ is a $D_h$-metric ball with $\op{diam}(B) \leq \delta$, then there exists $k\in\BB N$ with $2^{-k} \leq \op{diam}(B) \leq 2^{-k+1}$ and $z\in \BB D\cap ( 2^{-100 k} \BB Z^2 )$ 
such that $B$ intersects both $\bdy B_{2^{-k-1}}(z)$ and $\bdy B_{2^{-k}}(z)$. 
Therefore, $B$ contains a Euclidean ball of radius at least $2^{-(1+\zeta/2) k} \geq \op{diam}(B)^{1+\zeta}$, as required.
\end{proof}

\subsection{Volume estimates for LQG metric balls}
\label{sec-ball-vol}

The goal of this subsection is to establish the following estimate for the LQG mass of LQG metric balls. 

\begin{prop} \label{prop-ball-vol}
Let $h$ be a whole-plane GFF normalized so that $h_1(0) = 0$. 
For each $\zeta \in (0,1)$, it holds with superpolynomially high probability as $\ep\rta 0$ that
\eqb \label{eqn-ball-vol}
s^{4+\zeta} \leq \mu_{h}\left(B_s(z ; D_h) \right) \leq   s^{4 - \zeta }, \quad \forall s \in (0,\ep]  , \quad \forall z \in \BB D .
\eqe
\end{prop}

We will extract Proposition~\ref{prop-ball-vol} from known ball volume estimates for the Brownian map, which say that a.s.\ the volume of every ball of radius $s$ in the Brownian map simultaneously is bounded above and below by constants times $s^{4-\zeta}$ and $s^{4+\zeta}$ (stated as Lemma~\ref{lem-bm-ball-bounds}). 
These estimates together with the equivalence of the Brownian map and the quantum sphere do not immediately imply~\eqref{eqn-ball-vol} since we are working with a whole-plane GFF instead of a quantum sphere. 
One could attempt to transfer the estimates using some sort of quantitative local absolute continuity, but we instead take a different approach which we find to be easier.
We note that Proposition~\ref{prop-ball-vol} has not previously appeared in the Brownian map literature, although closely related results have been established (see the proof of Lemma~\ref{lem-bm-ball-bounds} below).

Local absolute continuity (without any quantitative Radon-Nikodym derivative bound) shows that a.s.\ the $\mu_h$-mass of every $D_h$-ball of radius $s \in (0,1)$ which is contained in $\BB D$ is bounded above and below by constants times $s^{4-\zeta}$ and $s^{4+\zeta}$. 
To turn this into a bound which holds with superpolynomially high probability instead of just a.s., we first use Lemma~\ref{lem-gff-compare} and scale invariance considerations to transfer from $h|_{\BB D}$ to the restriction of the truncated white-noise field $\wh h^\tr$ of~\eqref{eqn-wn-decomp-tr} to $B_R(0)$ for a large value of $R$ (Lemma~\ref{lem-good-ball-tr}).
The restrictions of $\wh h^\tr$ to radius-1 Euclidean balls contained in $B_R(0)$ which lie at distance at least $1/5$ from one another are independent. 
Hence, the fact that an event (in our setting, bounds for the $\mu_h$-mass of $D_h$-balls contained in the Euclidean ball) holds \emph{simultaneously} for all such Euclidean balls with high probability shows that in fact the event for a single Euclidean ball has to hold with \emph{extremely} high probability (Lemma~\ref{lem-ball-vol-upper-tr}). 
We then transfer back to $h$ to conclude the proof. 

Let us first record what we get from Brownian map estimates.

\begin{lem} \label{lem-bm-ball-bounds}
If $h$ is a whole-plane GFF normalized so that $h_1(0) = 0$, then a.s.\ 
\eqb \label{eqn-ball-sup-h}
\sup_{s \in (0,1)}  \sup_{z\in \BB D} \frac{\mu_h (B_s(z ; D_h)}{s^{4-\zeta }} < \infty 
\quad \op{and} \quad 
\inf_{s \in (0,1)}  \inf_{z\in \BB D} \frac{\mu_h (B_s(z ; D_h)}{s^{4+\zeta }} > 0   .
\eqe 
The same is true with the field $\wh h$ of~\eqref{eqn-wn-decomp} in place of $h$. 
\end{lem}
\begin{proof}
We will use estimates for the Brownian map, so we need to work with a quantum sphere due to the equivalence of the Brownian map and quantum sphere~\cite[Corollary 1.4]{lqg-tbm2}.
Let $h^{\op{Sph}}$ be an embedding into $\BB C$ of the quantum sphere (say, conditioned to have LQG area at least 1), normalized so that two marked points sampled uniformly from $\mu_{h^{\op{Sph}}}$ are sent to $0$ and $\infty$ and so that $1 = \sup\{r > 0 : h_r^{\op{Sph}}(0) + Q \log r = 0\}$, provided $h_r^{\op{Sph}}(0) + Q \log r =0$ for some $r > 0$.  
This choice of normalization makes it so that the laws of $h^{\op{Sph}}|_{\BB D \setminus B_{1/2}(0) }$ and $h|_{\BB D \setminus B_{1/2}(0)}$ are mutually absolutely continuous on the event $\{h_r^{\op{Sph}}(0) = 0\}$ (the laws of the restrictions of the fields to $\BB D$ are not absolutely continuous since $h^{\op{Sph}}$ has a $\gamma$-log singularity at 0).

By~\cite[Corollary 6.2]{legall-geodesics} and the equivalence of the Brownian map and the quantum sphere, 
\eqbn
\sup_{s > 0}  \sup_{z\in \BB C} \frac{\mu_{h^{\op{Sph}}} (B_s(z ; D_{h^{\op{Sph}}})}{s^{4-\zeta }} < \infty .
\eqen
Furthermore, since the continuum label process (the ``head of the Brownian snake") used to define the Brownian map is a.s.\ H\"older continuous with any exponent less than $1/4$, a.s.\ 
\eqbn
\inf_{s > 0}  \inf_{z\in \BB C} \frac{\mu_{h^{\op{Sph}}} (B_s(z ; D_{h^{\op{Sph}}})}{s^{4+\zeta }} > 0 .
\eqen
By the local absolute continuity between $h$ and $h^{\op{Sph}}$, we obtain~\eqref{eqn-ball-sup-h} with $\BB D\setminus B_{1/2}(0)$ in place of $\BB D$. By the translation invariance of the law of $h$, modulo additive constant, we get~\eqref{eqn-ball-sup-h}. By Lemma~\ref{lem-gff-compare}, the same is true with $\wh h$ in place of $h$. 
\end{proof}

We will now transfer to an estimate for $\wh h^\tr$ restricted to a large ball. 

\begin{lem} \label{lem-good-ball-tr}
Fix $p > 1$ and $\zeta \in (0,1)$. 
There is a universal constant $c>0$ and a random set $\mcl Z_{r^p} \subset B_{r^p}(0) \cap (3\BB Z^2)$ independent from $\wh h^\tr$ such that with probability tending to 1 as $r\rta\infty$, one has $\#\mcl Z_{r^p}  \geq c r^{2p} $ and for each $z\in \mcl Z_{r^p}$, 
\eqb \label{eqn-good-ball-tr}
s^{4 + \zeta } \leq \mu_{\wh h^\tr}\left(B_s(z ; D_{\wh h^\tr}) \right) \leq   s^{4 - \zeta }, \quad \forall s \leq r^{-\zeta} \min\left\{ 1 ,    D_{\wh h^\tr}(w , \bdy B_1(z)) \right\} , \quad \forall w \in B_1(z) .
\eqe 
\end{lem}

The reason why we need to restrict to the set $\mcl Z_{r^p}$ in Lemma~\ref{lem-good-ball-tr}, instead of looking at all points in $B_{r^p}(0)\cap (3\BB Z^2)$, is as follows.
To transfer from an estimate on $\BB D$ to an estimate on $B_{r^p}(0)$, we will use the scale invariance property of the white noise field $\wh h$, which says that $\wh h(r^{-p}\cdot) -\wh h_{r^{-p}}(r^{-p}\cdot) $ has the same law as $\wh h$ and is independent from $\wh h_{r^{-p}}$. 
We will restrict attention to the set of points where $\wh h_{r^{-p}}$ is not unusually large, which is independent from $\wh h(r^{-p}\cdot) -\wh h_{r^{-p}}(r^{-p}\cdot) $, then couple $\wh h^\tr$ with $\wh h(r^{-p}\cdot) -\wh h_{r^{-p}}(r^{-p}\cdot) $ using Lemma~\ref{lem-gff-compare}.

\begin{proof}[Proof of Lemma~\ref{lem-good-ball-tr}] 
\noindent\textit{Step 1: re-scaling from $\BB D$ to $B_{r^p}(0)$.} We will re-scale with the eventual aim of transferring Lemma~\ref{lem-bm-ball-bounds} to an estimate with $B_{r^p}(0)$ in place of $\BB D$ and $\wh h$ in place of $h$. This will lead to the definition of $\mcl Z_{r^p}$.
If we set $\wh h^{r^p} := (\wh h - \wh h_{r^{-p}})(r^{-p} \cdot)  $, then $\wh h^{r^p} \eqD \wh h$ and $\wh h^{r^p}$ is independent from $\wh h_{r^{-p} }$.  
Let  
\eqb \label{eqn-good-ball-def}
\mcl Z_{r^p} := \left\{ z\in B_{r^p}(0) \cap (3\BB Z^2) : \sup_{w \in B_1(z)} |\wh h_{r^{-p} }(r^{-p} w)| \leq \zeta^2 \log r \right\} .
\eqe 
We emphasize that $\mcl Z_{r^p}$ is determined by $\wh h_{r^{-p}}$, so is independent from $\wh h^{r^p}$. 

We will now argue that there is a universal constant $c>0$ such that 
\eqb \label{eqn-good-balls}
\BB P\left[\#\mcl Z_{r^p}  \geq c r^{2p} \right] \rta 1 \quad \text{as} \quad r \rta\infty.
\eqe
To see this, we observe that each $\wh h_{r^{-p}}(r^{-p} z)$ for $z\in 3\BB Z^2$ is Gaussian with variance $\log r^p$. By the Gaussian tail bound, $\BB P[|\wh h_{r^{-p}}(r^{-p} z)| \leq \frac{\zeta^2}{2} \log r]$ tends to 1 as $r\rta\infty$, uniformly over all $z\in 3\BB Z^2$.
By Markov's inequality, it holds with probability tending to 1 as $r\rta\infty$ that the number of $z\in B_{r^p}(0) \cap (3\BB Z^2)$ for which $|\wh h_{r^{-p}}(r^{-p} z)| \leq \frac{\zeta^2}{2} \log r$ is at least $\frac12 \#\left[  B_{r^p}(0) \cap (3\BB Z^2) \right]$, say.
This last quantity is at least $c r^{2p}$ for some universal constant $c>0$. 
By Lemma~\ref{lem-use-btis}, it holds with probability tending to 1 as $r \rta\infty$ that
\eqbn
\sup_{z \in B_{r^p}(0) \cap (3\BB Z^2)} \sup_{w\in B_1(z)} |\wh h_{r^{-p}}(r^{-p} w) - \wh h_{r^{-p}}(r^{-p} z)| \leq \frac{\zeta^2}{2} \log r .
\eqen
By combining these estimates, we get~\eqref{eqn-good-balls}. 
\medskip

\noindent\textit{Step 2: estimate for LQG balls centered at points of $\mcl Z_{r^p}$.}
By the LQG coordinate change formula~\eqref{eqn-lqg-coord-metric} and Lemma~\ref{lem-metric-f}, for each $z\in \mcl Z_{r^p}$ and each $x,y\in B_1(z)$, 
\eqbn
 r^{(p Q-\zeta^2) /\sqrt 6} D_{\wh h|_{r^{-p} B_1(z)} }(r^{-p}  x , r^{-p} y)  
\leq D_{\wh h^{r^p}|_{B_1(z)}}(x,y) 
\leq r^{(p Q+\zeta^2) /\sqrt 6} D_{\wh h|_{r^{-p} B_1(z)} }(r^{-p} x , r^{-p} y)
\eqen
Moreover, the analogous properties for the LQG measure show that
\eqbn
 r^{\sqrt{8/3} ( p Q-\zeta^2) } \mu_{\wh h  }(r^{-p} A)  
\leq \mu_{\wh h^{r^p} }(A) 
\leq r^{\sqrt{8/3} ( p Q+\zeta^2) } \mu_{\wh h  }(r^{-p} A)  ,
\quad \forall A \subset B_1(z) \quad \text{Borel}.
\eqen
Combining this with Lemma~\ref{lem-bm-ball-bounds} (with $\wh h$ in place of $h$) shows that with probability tending to 1 as $r \rta\infty$, it holds for each $z\in \mcl Z_{r^p}$ that
\eqb \label{eqn-good-ball-ratio}
s^{4 + a \zeta} \leq \mu_{\wh h}(B_s(w ; D_{\wh h}) ) \leq   s^{4 - a \zeta } ,\quad \forall s \leq r^{-\zeta} \min\{1 , D_{\wh h}(w , \bdy B_1(z)) \} ,\quad \forall w \in B_1(z) ,
\eqe 
where here $a  > 0$ is a universal constant. Note that we used that $4/\sqrt 6 = \sqrt{8/3}$ to cancel two large powers of $r$ and we used that $r^{\zeta^2} \leq s^{-\zeta}$ for $s\leq r^{-\zeta }$ to absorb a small power of $r$ into a power of $s$.  
\medskip

\noindent\textit{Step 3: transferring from $\wh h$ to $\wh h^\tr$.} 
By Lemma~\ref{lem-gff-compare} and a union bound over $O_{r^p}(r^{2p} )$ Euclidean balls of radius 1 which cover $B_{r^p}(0)$, we can couple $\wh h^{r^p}$ and $\wh h^\tr$ in such a way that with probability tending to 1 as $r \rta\infty$, we have $\max_{z\in B_{r^p +1}(0)} |(\wh h^{r^p} - \wh h^\tr)(z)| \leq (\log r)^{2/3}$. 
Since $\wh h^{r^p}$ is independent from $\mcl Z_{r^p}$, we can take $\wh h^\tr$ to be independent from $\mcl Z_{r^p}$ in this coupling. 
Our choice of coupling together with~\eqref{eqn-good-balls} and~\eqref{eqn-good-ball-ratio} shows that with probability tending to 1 as $r\rta\infty$, it holds for each $z\in \mcl Z_{r^p}$, each $w\in B_1(z)$, and each $s \leq r^{-\zeta}  e^{- \frac{1}{\sqrt 6} (\log r)^{2/3} } \min\left\{ 1 ,   D_{\wh h^\tr}(w , \bdy B_1(z)) \right\}$ that
\eqbn
e^{-\frac{1}{\sqrt 6} (\log r)^{2/3} } s^{4 -  a \zeta } \leq  \mu_{\wh h^\tr}(B_s(w ; D_{\wh h^\tr}) ) \leq  e^{\frac{1}{\sqrt 6} (\log r)^{2/3} } s^{4 -  a \zeta }   .
\eqen 
After adjusting $\zeta$ appropriately, this gives~\eqref{eqn-good-ball-tr}.
\end{proof}

We can now go from events with probability tending to 1 to events with superpolynomially high probability.

\begin{lem} \label{lem-ball-vol-upper-tr}
For each $\zeta \in (0,1)$, it holds with superpolynomially high probability as $\ep\rta 0$ that
\eqb \label{eqn-ball-upper-tr}
s^{4+\zeta} \leq \mu_{\wh h^\tr}\left(B_s(w ; D_{\wh h^\tr}) \right) \leq   s^{4 - \zeta }, 
\quad \forall s \leq  \ep \min\left\{ 1 ,  D_{\wh h^\tr}(w , \bdy \BB D) \right\}, 
\quad \forall w \in \BB D .
\eqe
\end{lem}
\begin{proof}
Fix $p >1$, which we will eventually send to $\infty$. 
For $z\in  3\BB Z^2 $ and $r > 1$, let 
\eqbn
E_r(z) :=  \left\{s^{4+\zeta} \leq \mu_{\wh h^\tr}\left(B_s(w ; D_{\wh h^\tr}) \right) \leq  s^{4 - \zeta }, \: \forall s \leq r^{-\zeta} \min\left\{ 1 ,  D_{\wh h^\tr}(w , \bdy B_1(z)) \right\} ,\: \forall w \in B_1(z)  \right\} .
\eqen
Note that the event in~\eqref{eqn-ball-upper-tr} is the same as $E_{\ep^{-1/\zeta}}(0)$. 
Furthermore, if we let $\mcl Z_{r^p} \subset B_{r^p}(0)\cap (3\BB Z^2)$ be the random set independent from $\wh h^\tr$ from Lemma~\ref{lem-good-ball-tr}, then that lemma tells us that with probability tending to 1 as $r\rta\infty$, we have $\#\mcl Z_{r^p} \geq c r^{2p}$ and $E_r(z)$ occurs for every $z\in \mcl Z_{r^p}$. 

The fields $\wh h^\tr|_{B_1(z)}$ for different choices of $z \in 3\BB Z^2$ are independent and the law of $\wh h^\tr$ is invariant with respect to spatial translations. 
Consequently, the events $E_r(z)$ for different choices of $z\in 3\BB Z^2$ are independent. 
Since $\wh h^\tr$ is independent from $\mcl Z_{r^p}$, we get that 
\eqb
1 - o_r(1) \leq \BB P\left[ E_r(z) ,\: \forall z\in \mcl Z_r \,|\, \mcl Z_r \geq c r^{2p} \right] \leq   \BB P\left[ E_r(0) \right]^{c r^{2p} } .
\eqe
If $r$ is large enough that this $1-o_r(1)$ is at least $e^{-c}$, then re-arranging gives  
\eqbn
\BB P[E_r(0)] \geq e^{-1/r^{2p} } \geq  1 - r^{-2p}
\eqen
where here we have used the elementary inequality $1-e^{-x} \leq x$. Since $p> 1$ can be made arbitrarily large, we get that $E_r(0)$ occurs with superpolynomially high probability as $r\rta\infty$. Setting $r = \ep^{-1/\zeta}$ now concludes the proof. 
\end{proof}

\begin{proof}[Proof of Proposition~\ref{prop-ball-vol}]
By Lemmas~\ref{lem-gff-compare} and~\ref{lem-ball-vol-upper-tr} together with the scale invariance of the law of $h$, modulo additive constant, and the fact that the law of $h_2(0)$ is Gaussian with constant-order variance, it holds with superpolynomially high probability as $\ep\rta 0$ that
\eqbn
s^{4+\zeta} \leq \mu_{h}\left(B_s(w ; D_{h}) \right) \leq   s^{4 - \zeta }, 
\quad \forall s \leq  \ep^{1/2} \min\left\{ 1 ,  D_{h}(w , \bdy  B_2(0) ) \right\}, 
\quad \forall w \in B_2(0)   .
\eqen
By Proposition~\ref{prop-lqg-ball-upper}, it holds with superpolynomially high probability as $\ep\rta 0$ that $D_h(\bdy\BB D,\bdy B_2(0)) \geq \ep^{1/2}$. 
Combining these estimates gives~\eqref{eqn-ball-vol}. 
\end{proof}

\subsection{Estimates for the 0-quantum cone}
\label{sec-0cone-estimates}

We now want to shift attention from the whole-plane GFF to the 0-quantum cone, with a view toward proving Proposition~\ref{prop-ball-moment}.
To this end, we will transfer the main results of the preceding subsections to the case of a 0-quantum cone. We start with estimates for the LQG areas of LQG metric balls which follows from Proposition~\ref{prop-ball-vol}.

\begin{prop} \label{prop-0cone-ball-bounds}
Let $(\BB C , h , 0, \infty)$ be a 0-quantum cone. 
\begin{enumerate}
\item With superpolynomially high probability as $C\rta\infty$, one has $C^{-1} \leq \mu_h(B_1(0 ; D_h)) \leq C $.   \label{item-0cone-mass} 
\item For $\zeta\in (0,1)$, it holds with superpolynomially high probability as $C\rta \infty$ that $C^{-\zeta} \leq \mu_h(B_1(z; D_h)) \leq C^\zeta$ for each $z\in B_C(0 ; D_h)$. \label{item-0cone-union} 
\end{enumerate}
\end{prop}

By the scaling property~\eqref{eqn-cone-scale-mm} of the 0-quantum cone, the law of $(\BB C , D_h,\mu_h)$ as a metric measure space is invariant under scaling distances by $b^{1/4}$ and areas by $b$, for any $b>0$. We will often use this fact in conjunction with Proposition~\ref{prop-0cone-ball-bounds} without comment. 

\begin{proof}[Proof of Proposition~\ref{prop-0cone-ball-bounds}]
The proposition statement does not depend on the choice of embedding for $h$, so we can assume without loss of generality that $h$ is given the circle-average embedding. 
Recall that $h|_{\BB D}$ agrees in law with the corresponding restriction of a whole-plane GFF normalized so that its circle average over $\bdy\BB D$ is zero. 
By Proposition~\ref{prop-lqg-ball-upper}, it holds with superpolynomially high probability as $C\rta\infty$ that $B_{C^{-1}}(0 ; D_h) \subset \BB D$.
Hence Proposition~\ref{prop-ball-vol} (applied with $\zeta = 1$, say) shows that with superpolynomially high probability as $C\rta\infty$,
\eqb \label{eqn-0cone-unscaled}
C^{-5 } \leq \mu_h(B_{C^{-1}}(0 ; D_h) ) \leq C^{ -3 } .
\eqe 
By the scale invariance property of the 0-quantum cone~\eqref{eqn-cone-scale-mm}, we have $(\BB C , D_h , \mu_h) \eqD (\BB C , C  D_h , C^{ 4} \mu_h)$ as metric measure spaces. 
Therefore,~\eqref{eqn-0cone-unscaled} implies that with superpolynomially high probability as $C \rta\infty$, one has $C^{-1} \leq \mu_h(B_1(0 ; D_h)) \leq C $, which is assertion~\ref{item-0cone-mass}.

We now prove assertion~\ref{item-0cone-union} via a similar argument.
By Proposition~\ref{prop-ball-vol} and our above description of the law of $h|_{\BB D}$, it holds with superpolynomially high probability as $C\rta\infty$ that
\eqbn
C^{-8- \zeta} \leq \mu_h(B_{C^{-2}}( z ; D_h)) \leq C^{-8+\zeta}, \quad \forall z \in \BB D  .
\eqen
Furthermore, by Proposition~\ref{prop-lqg-ball-upper}, it holds with superpolynomially high probability as $C\rta\infty$ that $B_{C^{-1} }(0;D_h) \subset \BB D$.
Hence with superpolynomially high probability as $C\rta\infty$, 
\eqb \label{eqn-0cone-union0}
C^{-8- \zeta} \leq \mu_h(B_{C^{-2}}( z ; D_h)) \leq C^{-8+\zeta}, \quad \forall z \in B_{C^{-1}}(0 ; D_h) .
\eqe
We now scale distances by $C^2$ and areas by $C^8$ and apply~\eqref{eqn-cone-scale-mm} as above to deduce assertion~\ref{item-0cone-union} from~\eqref{eqn-0cone-union0}. 
\end{proof}

We next record an estimate to the effect that $D_h$-metric balls have to contain Euclidean metric balls of radius comparable to their Euclidean diameters.

\begin{prop} \label{prop-0cone-swallow}
Let $(\BB C ,h , 0, \infty)$ be a 0-quantum cone and let $\zeta\in (0,1)$. 
With superpolynomially high probability as $C\rta\infty$, each $D_h$-ball $B$ which is contained in $B_C(0)$ and which has $D_h$-radius at least $C^{-1}$ contains a Euclidean ball of radius at least $C^{-\zeta} \op{diam}(B)$.
\end{prop}

We will deduce Proposition~\ref{prop-0cone-swallow} from Propositions~\ref{prop-lqg-ball-upper} and~\ref{prop-lqg-ball-union}.
Before we can do so, however, we need some basic polynomial tail estimates for the minimal and maximal radii of Euclidean balls.
This is because Proposition~\ref{prop-0cone-swallow} only holds for LQG balls with sufficiently small Euclidean diameter and because $\op{diam}(B)^{1+\zeta}$ can be much smaller than $C^{-\zeta} \op{diam}(B)$ if $\op{diam}(B)$ is tiny. 
 
\begin{lem} \label{lem-min-ball-diam}
Let $h$ be a whole-plane GFF normalized so that $h_1(0) = 0$.  
For each $q > \tfrac{8}{(2-\sqrt{8/3})^2}$ and each $\ep \in (0,1)$, 
\eqb \label{eqn-min-ball-diam}
\BB P\left[ \op{diam}\left( B_\ep(z; D_h) \right) \geq \ep^q ,\: \forall z \in \BB D \right] \geq 1 -  \ep^{\alpha(q ) + o_\ep(1)}  ,
\eqe  
where the rate of the $o_\ep(1)$ depends only on $q$ and
\eqb \label{eqn-min-ball-exponent}
\alpha(q ) :=      \frac{3q}{16} \left(\frac{10}{3} - \frac{4}{q} \right)^2 - 2q  .
\eqe 
\end{lem}
\begin{proof}
By standard estimates for the $\sqrt{8/3}$-LQG measure (see, e.g., the proof of~\cite[Lemma 3.7]{dg-lqg-dim} with $\gamma=\sqrt{8/3}$), for $p  > 2\sqrt{8/3}$, 
it holds with probability at least $1-\delta^{ 3p^2/16  - 2}    $ that each Euclidean ball centered at a point of $\BB D$ with radius at least $\delta $ has $\mu_h$-mass at most $\delta^{ 10/3 - p   }$. 
We now fix $\zeta \in (0,1)$, which we will eventually send to 0. 
Applying the above estimate with $p = 10/3 - (4+\zeta)/q$ and $\delta = \ep^q =  \ep^{(4+\zeta)/(10/3 - p)}$ and shows that with probability at least $1 - \ep^{\alpha(q) + o_\zeta(1) + o_\ep(1)}$ (with the $o_\zeta(1)$ deterministic and independent of $\ep$), 
each Euclidean ball centered at a point of $\BB D$ with radius $\ep^{ q}$ has $\mu_h$-mass at most $\ep^{4+\zeta}$. 

By Proposition~\ref{prop-ball-vol}, with superpolynomially high probability as $\ep\rta 0$, $\mu_h(B_\ep(z;D_h))  > \ep^{4+\zeta}$ for each $z\in \BB D$. 
In particular, no such ball $B_\ep(z;D_h)$ can be contained in a Euclidean ball with $\mu_h$-mass at most $\ep^{4+\zeta}$. 
Combining this with the preceding paragraph and sending $\zeta\rta 0$ concludes the proof.
\end{proof}

\begin{lem} \label{lem-max-ball-diam}
Let $h$ be a whole-plane GFF normalized so that $h_1(0) = 0$.  
For each $q \in \left( 0 , \tfrac{8}{(2+\sqrt{8/3})^2} \right)$ and each $\ep \in (0,1)$, 
\eqb \label{eqn-max-ball-diam}
\BB P\left[ \op{diam}\left( B_\ep(z; D_h) \right) \leq \ep^q ,\: \forall z \in \BB D \right] \geq 1 -  \ep^{\beta(q ) + o_\ep(1)}  ,
\eqe  
where the rate of the $o_\ep(1)$ depends only on $q$ and
\eqbn 
\beta(q ) := \frac{3 q}{16  } \left( \frac{10}{3} + \frac{4}{q} \right)^2 - 2 q  .
\eqen
\end{lem}
\begin{proof} 
Fix $\wt q \in  \left( q , \tfrac{8}{(2+\sqrt{8/3})^2} \right)$, which we will eventually send to $q$, and $\zeta \in (0,1)$, which we will eventually send to 0. 
By standard estimates for the $\sqrt{8/3}$-LQG measure (see, e.g.,~\cite[Lemma 2.5]{gms-harmonic} with $\gamma=\sqrt{8/3}$), for $p  > 2\sqrt{8/3}$ 
it holds with probability at least $1-\delta^{3p^2/16 - 2}    $ that each Euclidean ball centered at a point of $\BB D$ with radius at least $\delta  $ has $\mu_h$-mass at least $\delta^{10/3 + p}$. 
Applying the above estimate with $p =   (4-\zeta)/\wt q - 10/3$ and $\delta = \ep^{\wt q} =  \ep^{(4-\zeta)/(10/3 + p)}$ and shows that with probability at least $1 - \ep^{\beta(\wt q) + o_\zeta(1) + o_\ep(1)}$ (with the $o_\zeta(1)$ deterministic and independent of $\ep$), 
each Euclidean ball centered at a point of $\BB D$ with radius at least $\ep^{ \wt q}$ has $\mu_h$-mass at least $\ep^{4-\zeta}$. 

By Proposition~\ref{prop-ball-vol} (applied with $B_2(0)$ in place of $\BB D$ and $\ep^{1/2}$ in place of $\ep$), we see that with superpolynomially high probability as $\ep\rta 0$, $\mu_h(B_\ep(z;D_h)) <  \ep^{4 - \zeta}$ for each $z\in \BB D$. 
By Proposition~\ref{prop-lqg-ball-union}, it holds with superpolynomially high probability as $\ep\rta 0$ that each $D_h$-ball centered at a point of $\BB D$ which has Euclidean diameter at least $\ep^q$ contains a Euclidean ball of radius at least $\ep^{\wt q}$.
By the preceding estimates, with probability at least $1 - \ep^{\beta(\wt q) + o_\zeta(1) + o_\ep(1)}$, each such $D_h$-ball has $\mu_h$-mass at least $\ep^{4-\zeta}$ and hence $D_h$-radius strictly larger than $\ep$.  
Sending $\zeta \rta 0$ and then $\wt q\rta q$ concludes the proof.
\end{proof}

\begin{proof}[Proof of Proposition~\ref{prop-0cone-swallow}] 
Let $\ul q \in \left( 0 , \tfrac{8}{(2+\sqrt{8/3})^2}\right)$ and $  \ol q  > \tfrac{8}{(2-\sqrt{8/3})^2}$. We will eventually send $\ul q$ to 0 and $\ol q$ to $\infty$. 
By Lemmas~\ref{lem-min-ball-diam} and~\ref{lem-max-ball-diam}, it holds with probability at least $1 - C^{- (3 \alpha(\ol q) ) \wedge \beta(\ul q)  + o_C(1)}$ that each $D_h$-ball contained in $\BB D$ with $D_h$-diameter between $C^{-3}$ and $C^{-1}$ has Euclidean diameter between $C^{-3 \ol q}$ and $C^{-\ul q}$. 
By Proposition~\ref{prop-lqg-ball-union}, it holds with superpolynomially high probability as $C\rta\infty$ that each $D_h$-ball contained in $\BB D$ with Euclidean diameter at most $C^{-\ul q}$ contains a Euclidean ball of radius at least $\op{diam}(B)^{1+\zeta/(2\ol q)}$. 
By Proposition~\ref{prop-lqg-ball-upper}, it holds with superpolynomially high probability as $C\rta\infty$ that $B_{C^{-1}}(0;D_h)\subset \BB D$.
Hence with probability at least $1 - C^{-(3\alpha(\ol q)) \wedge \beta(\ul q)  + o_C(1)}$, each $D_h$-ball contained in $B_{C^{-1}}(0;D_h)$ with $D_h$-diameter in $[C^{-3} , C^{-1}]$ contains a Euclidean ball of radius at least $\op{diam}(B)^{1+\zeta/(3 \ol q)} \geq C^{-\zeta} \op{diam}(B)$. 
This statement does not depend on the choice of embedding $h$, so we can add $\frac{2}{\sqrt 6} \log C$ to $h$ (i.e., scale distances by $C^2$) to get that the event in the statement of the lemma holds with probability at least $1 - C^{-(3\alpha(\ol q)) \wedge \beta(\ul q)  + o_C(1)}$. 
Since $\alpha(\ol q) , \beta(\ul q) \rta\infty$ as $\ol q\rta\infty$ and $\ul q \rta 0$, this concludes the proof.
\end{proof}

\subsection{Proof of the moment estimate}
\label{sec-moment-proof}

Throughout this subsection, we let $h$ be the circle-average embedding of a 0-quantum cone in $(\BB C , 0, \infty)$. 
We also fix $\lambda = 1$ and define the Poisson point process $\mcl P := \mcl P_h^1$ and the collection of Voronoi cells $\mcl H = \mcl H_h^1$. 
We recall that $H_0$ is the a.s.\ unique cell in $\mcl H$ which contains 0. 

To prove Proposition~\ref{prop-ball-moment} (and thereby Proposition~\ref{prop-cell-moment}), we first establish an upper bound for the $D_h$-diameter of a Voronoi cell (Lemma~\ref{lem-cell-in-ball}) by building a ``wall" of Voronoi cells in the annulus between two concentric $D_h$-balls (Lemma~\ref{lem-separating-balls}). 
Using this and Proposition~\ref{prop-0cone-swallow} allows us to simultaneously bound $\op{diam}(B_H)^2 / \op{area}(B_H)$ for all of the Voronoi cells $H$ with $0\in B_H$, where here we recall that $B_{H }$ is the smallest $D_h$-ball centered at the center point of $H$ which contains $H$ (Lemma~\ref{lem-ball-ratio}). 
We will then prove an upper bound for the number of cells with $0\in B_H$ and for the maximal degree of these cells (Lemma~\ref{lem-cell-intersect}) and combine these estimates to get Proposition~\ref{prop-ball-moment}. 

\begin{lem} \label{lem-ball-covering}
Let $\zeta\in (0,1)$. 
With superpolynomially high probability as $C\rta \infty$, the ball $B_C(0;D_h)$ is contained in the union of at most $C^{4+\zeta}$ $D_h$-metric balls of radius 1.
\end{lem}
\begin{proof}
Let $\mcl Z$ be a maximal collection of points in $B_C(0;D_h)$ such that the balls $B_{1/2}(z;D_h)$ for $z\in\mcl Z$ are disjoint. 
By Proposition~\ref{prop-0cone-ball-bounds}, it holds with superpolynomially high probability as $C\rta\infty$ that $\min_{z\in \mcl Z} \mu_h(B_{1/2}(z;D_h)) \geq C^{-\zeta/2}$ and $\mu_h(B_C(z;D_h)) \leq C^{4+\zeta/2}$, which implies that $\#\mcl Z\leq C^{4+\zeta}$. 
By the maximality of $\mcl Z$, each point of $B_C(0;D_h)$ is contained in $B_1(z;D_h)$ for some $z\in\mcl Z$.
\end{proof}

The following lemma shows that Voronoi cells are extremely unlikely to have a larger quantum diameter than one would expect. 

\begin{lem} \label{lem-cell-in-ball}
Fix $\zeta\in (0,1)$. 
With superpolynomially high probability as $C\rta\infty$, each cell in $\mcl H$ which intersects $B_C(0;D_h)$ has $D_h$-diameter at most $C^\zeta$. 
\end{lem}

To prove Lemma~\ref{lem-cell-in-ball}, we will use the following lemma to build a ``wall" of Voronoi cells which separate the boundaries of two concentric $D_h$-balls.

\begin{lem} \label{lem-separating-balls}
For $\zeta\in(0,1)$, it holds with superpolynomially high probability as $C\rta\infty$ that the following is true.
For each $z\in B_C(0;D_h)$, we can find a finite collection of at most $C^\zeta$ $D_h$-metric balls of radius $1/2$ which are contained in $ B_{3 }(z;D_h) \setminus B_1(z;D_h) $ and whose union disconnects $B_1(z;D_h)$ from $\BB C\setminus B_{3}(z;D_h)$.
\end{lem}
\begin{proof} 
By Proposition~\ref{prop-0cone-ball-bounds}, it holds with superpolynomially high probability as $C\rta\infty$ that 
\eqb \label{eqn-separating-balls-bound}
\mu_h(B_3(z;D_h)) \leq C^{\zeta/2} \quad\op{and} \quad \mu_h(B_{1/4}(z;D_h) ) \geq C^{-\zeta/2} ,\quad \forall z \in B_C(0;D_h) .
\eqe 
Henceforth assume that~\eqref{eqn-separating-balls-bound} holds and fix $z\in B_C(0;D_h)$. We will construct a collection of $D_h$-balls as in the statement of the lemma. 
 
Let $\mcl C $ be a maximal collection of points in $ \bdy B_2(z;D_h)$ such that the balls $B_{1/4}(w;D_h)$ for $w\in \mcl C $ are disjoint. 
The union of the balls $B_{1/2}(w;D_h)$ for $w\in\mcl C$ covers $ \bdy B_2(z;D_h)$ (otherwise, we could find a point in $ \bdy B_2(z;D_h)$ which lies at distance at least $1/2$ from each $w\in \mcl C $, which contradicts the maximality of $\mcl C $). 
Consequently, $\bigcup_{w \in \mcl C} B_{1/2}(w;D_h)$ disconnects $B_1(z;D_h)$ from $\BB C\setminus B_3(z;D_h)$. 
Furthermore, each of the balls $B_{1/2}(w;D_h)$ for $w\in\mcl C$ is centered at a point of $\bdy B_2(z;D_h)$, so is contained in $B_3(z;D_h)\setminus B_1(z;D_h)$. 
Finally, since the balls $B_{1/4}(w;D_h)$ for $w\in  \mcl C $ are disjoint and contained in $B_3(z;D_h)$, we see from~\eqref{eqn-separating-balls-bound} that $\#\mcl C\leq C^\zeta$. 
\end{proof}

\begin{proof}[Proof of Lemma~\ref{lem-cell-in-ball}]
By Lemma~\ref{lem-ball-covering}, with superpolynomially high probability as $C\rta\infty$, we can find a collection $\mcl Z_C$ of at most $C^{4+\zeta}$ points of $B_C(0;D_h)$ such that the union of the balls $B_1(z;D_h)$ for $z\in\mcl Z_C$ covers $B_C(0;D_h)$. 
By Lemma~\ref{lem-separating-balls} and the scaling property~\eqref{eqn-cone-scale-mm} of the 0-quantum cone, with superpolynomially high probability as $C\rta\infty$ we can find for each $z\in \mcl Z_C$ a finite collection $\mcl C(z)$ of at most $C^\zeta$ $D_h$-balls of radius $\frac12 C^\zeta $ which are contained in $\ol{B_{3C^\zeta}(z;D_h) \setminus B_{C^\zeta}(z;D_h)}$ and whose union disconnects $B_{ C^\zeta}(z;D_h)$ from $\BB C\setminus B_{3C^\zeta}(z;D_h)$. By Proposition~\ref{prop-0cone-ball-bounds}, it holds with superpolynomially high probability as $C\rta\infty$ that the $\mu_h$-mass of each ball in each of the collections $\mcl C(z)$ is at least $C^{4\zeta(1-\zeta)}$

By the formula for the Poisson distribution, if this is the case then the conditional probability given $h$ that each of the balls in $\bigcup_{z\in\mcl Z_C} \mcl C(z)$ contains a point of $\mcl P$ is at least $1 - \exp\left( - C^{4\zeta(1-\zeta)} \right)$. 
By a union bound over the at most $C^{4+2\zeta}$ balls in $\bigcup_{z\in\mcl Z_C} \mcl C(z)$,
we find that with superpolynomially high probability as $C\rta\infty$, each of the balls this collection contains a point of $\mcl P$. 
Similarly, it holds with superpolynomially high probability as $C\rta\infty$ that $B_{C^\zeta/4}(z ;D_h)$ contains a point $w_z \in \mcl P$ for each $z\in \mcl Z_C$. 
Since each point of $B_1(z ; D_h)$ lies within $D_h$-distance $C^\zeta/4+1$ of $w_z$, this means that the center point of each Voronoi cell which intersects $B_1(z;D_h)$ is contained in $B_{C^\zeta/4+2}(z;D_h) \subset B_{C^\zeta/2}(z;D_h)$. 

If the events described in the preceding paragraph are satisfied, then for $z \in \mcl Z_C$, each point of $\BB C\setminus B_{3C^\zeta}(z;D_h)$ is $D_h$-closer to a point of $\mcl P$ which is contained in one of the balls in $\mcl C(z)$ than it is to any point of $B_{C^\zeta/2}(z;D_h)$. 
This means that no such point can be contained in a cell whose center point is in $B_{C^\zeta/2}(z;D_h)$, hence no such point can be contained in a cell which intersects $B_1(z;D_h)$. 
Hence each Voronoi cell which intersects $B_1(z;D_h)$ is contained in $B_{3C^\zeta}(z;D_h)$ with superpolynomially high probability as $C\rta\infty$. 
Since the union of the balls $B_1(z;D_h)$ for $z\in\mcl Z_C$ covers $B_C(0;D_h)$, this gives the statement of the lemma with $3C^\zeta$ in place of $C^\zeta$, which is sufficient. 
\end{proof}

We can now prove our main moment estimates. Recall that $B_H$ denotes the smallest $D_h$-metric ball containing the cell $H$ which is centered at the center point of $H$. 
 
\begin{lem} \label{lem-ball-ratio}
With superpolynomially high probability as $C\rta \infty$, 
\eqb \label{eqn-ball-ratio}
  \max_{H \in \mcl H : 0 \in B_H} \frac{\op{diam}(B_{H})^2}{\op{area}(B_{H})}   \leq C .
\eqe 
\end{lem}
\begin{proof}
Fix $\zeta\in (0,1)$, which we will eventually send to 0. 
Lemma~\ref{lem-cell-in-ball} shows that with superpolynomially high probability as $C\rta\infty$, each $H\in\mcl H$ with $0 \in B_H$ has $D_h$-diameter at most $C^\zeta$, so is contained in $B_{2C^\zeta}(0;D_h)$. 
We will now argue that with probability at least $1-O_C(C^{-4/\zeta + \zeta})$, the $D_h$-radius of $B_H$ for each such cell $H$ is at least $C^{-1/\zeta}$. Indeed, Proposition~\ref{prop-0cone-ball-bounds} shows that with superpolynomially high probability as $C\rta\infty$, we have $\mu_h(B_{2C^{-1/\zeta}}(0;D_h)) \leq C^{-4/\zeta + \zeta }$. 
Since the number of points of $\mcl P$ which belong to $B_{2C^{-1/\zeta}}(0;D_h)$ is Poisson with mean $\mu_h(B_{2C^{-1/\zeta}}(0;D_h)) $ conditional on $h$, it follows that with probability at least $1-O_C(C^{-4/\zeta + \zeta})$, no point of $\mcl P$ is contained in $B_{2C^{-1/\zeta}}(0)$. In particular, no cell $H \in \mcl H$ is contained in $B_{2C^{-1/\zeta}}(0)$, so if $0 \in B_H$ then the $D_h$-radius of $B_H$ must be at least $C^{-1/\zeta}$, as required. 

By Proposition~\ref{prop-0cone-swallow}, it holds with superpolynomially high probability as $C\rta\infty$ that each $D_h$-ball $B$ centered at a point of $B_C(0;D_h)$ with $D_h$-radius at least $C^{-1/\zeta}$ contains a Euclidean ball of radius at least $C^{-1/2} \op{diam}(B)$. Combining this with the preceding paragraph shows that~\eqref{eqn-ball-ratio} holds with probability at least $1-O_C(C^{-4/\zeta + \zeta})$. Sending $\zeta\rta 0$ concludes the proof.
\end{proof}

To bound the number of cells with $0 \in B_H$ and their degrees, we will need the following lemma.

\begin{lem} \label{lem-cell-intersect}
For each $\zeta \in (0,1)$, it holds with superpolynomially high probability as $C\rta\infty$ that 
\eqb \label{eqn-cell-intersect}
\#\left\{H\in\mcl H : B_H\cap B_C(0; D_h) \not=\emptyset \right\} \leq C^{4 + \zeta}  
\eqe 
and the $D_h$-diameter of each of the balls $B_h$ which intersects $B_C(0;D_h)$ is at most $C^\zeta$. 
\end{lem}
\begin{proof}
Fix $\zeta\in (0,1)$. 
By Lemma~\ref{lem-cell-in-ball} and a union bound over dyadic values of $C$, it holds with superpolynomially high probability as $C\rta\infty$ that for each $k \in \BB N$ with $2^k \geq C/2 $, each cell $H\in\mcl H$ which intersects $B_{2^k}(0;D_h)$ has $D_h$-diameter at most $2^{\zeta k -1}$. 
Henceforth assume that this is the case.

For a cell $H\in\mcl H$, let $k_H \in\BB N$ be the smallest integer for which $H \cap B_{2^{k_H}}(0;D_h) \not=\emptyset$. 
If $2^{k_H} \geq 4C$, then the $D_h$-diameter of $H$ is at most $2^{\zeta k_H-1} $, so the ball $B_H$ has $D_h$-diameter at most $2^{\zeta k_H }  < 2^{k_H-2}$. Since this ball intersects $\BB C\setminus B_{2^{k_H }}(0;D_h)$ (by the minimality of $k_H$), it cannot intersect $B_{2^{k_H-2}}(0; D_h)$, so must be disjoint from $B_C(0;D_h)$. 

Consequently, each $H\in\mcl H$ with $B_H\cap B_{C}(0;D_h) \not=\emptyset$ must intersect $B_{4C}(0;D_h)$ and hence must each have $D_h$-diameter at most $4 C^\zeta$. 
In particular, each such cell is contained in $B_{5C}(0;D_h)$. We are thus left to bound the number of cells contained in $B_{5C}(0;D_h)$. 
By Proposition~\ref{prop-0cone-ball-bounds}, it holds with superpolynomially high probability as $C\rta\infty$ that $\mu_h(B_{5C}(0;D_h)) \leq C^{4+\zeta/2}$.
Conditioned on this event, the number of points of $\mcl P$ which belong to $B_{5C}(0;D_h)$ is Poisson with mean at most $C^{4+\zeta/2}$. 
By the elementary estimate
\eqbn
\BB P\left[ X  > x\right] \leq \frac{e^{-\lambda} \lambda^x}{x^x} ,\quad \text{for} \quad x\sim \op{Poisson}(\lambda) ,
\eqen
we see that the probability that $B_{5C}(0;D_h)$ contains more that $C^{4+ \zeta}$ points of $\mcl P$ decays superpolynomially in $C$.
This gives~\eqref{eqn-cell-intersect}.
\end{proof}

\begin{lem} \label{lem-cell-deg}
With superpolynomially high probability as $C\rta\infty$, 
\eqb \label{eqn-cell-deg}
\#\left\{ H\in \mcl H : 0 \in B_H \right\} \leq C \quad \text{and} \quad
\max_{H\in\mcl H : 0 \in B_H} \op{deg}(H) \leq C .
\eqe
\end{lem}
\begin{proof}
By Lemma~\ref{lem-cell-intersect} (applied with any choice of $\zeta \in (0,1)$) it holds with superpolynomially high probability as $C\rta\infty$ that each $H\in\mcl H$ with $ 0 \in B_H$ has $D_h$-diameter at most $C $, so is contained in $B_{2C}(0;D_h)$. This means that each neighbor of each such cell intersects $B_{2C}(0;D_h)$. 
Therefore, Lemma~\ref{lem-cell-intersect} implies that with superpolynomially high probability as $C\rta\infty$, the total number of cells such that either $0 \in B_H$ or $H \sim H'$ for some cell $H'$ with $0 \in B_{H'}$ is at most $C^5$. Replacing $C$ with $C^{1/5}$ concludes the proof. 
\end{proof}

\begin{proof}[Proof of Proposition~\ref{prop-ball-moment}]
By Lemmas~\ref{lem-ball-ratio} and~\ref{lem-cell-deg}, it holds with superpolynomially high probability as $C\rta\infty$ that
\eqb
\sum_{H\in \mcl H : 0 \in B_H} \frac{\op{diam}(H)^2 \op{deg}(H)}{\op{area}(B_H)} \leq C^3 .
\eqe
Consequently, this sum has finite moments of all positive orders.  
\end{proof}

\section{Open problems}
\label{sec-open-problems}

Perhaps the most natural question to ask about Brownian motion on the Brownian map is the following. 

\begin{prob}
\label{prob-rw-conv}
Show that random walk on uniform random planar maps (e.g., uniform quadrangulations or triangulations) converges to Brownian motion on the Brownian map with respect to the Gromov-Hausdorff-Prokhorov-uniform topology, the natural topology for curve-decorated metric measure spaces~\cite{gwynne-miller-uihpq}. 
\end{prob}

It is known that self-avoiding walk and percolation interfaces on uniform random planar maps converge to SLE$_{8/3}$ and SLE$_6$, respectively~\cite{gwynne-miller-saw,gwynne-miller-perc}. In contrast to the case of $\BB Z^2$, however, random walk on a random planar map seems harder to analyze than SAW or percolation interfaces since the random walk can re-trace its past, so one cannot explore the curve and the planar map simultaneously using peeling. 
 
Our results only concern Brownian motion on Brownian surfaces viewed modulo time parameterization. 
The natural way to parameterize Brownian motion on a $\sqrt{8/3}$-LQG surface, equivalently a Brownian surface, is called \emph{Liouville Brownian motion} and is constructed in~\cite{berestycki-lbm,grv-lbm}. 

\begin{prob}
\label{prob-lbm}
Show that in the setting of Theorem~\ref{thm-rw-conv0}, the random walk on $\mcl P^\lambda$, parameterized so that it traverses one edge in one unit of time, converges to Liouville Brownian motion with respect to the uniform topology as $\lambda\rta\infty$. 
\end{prob}

There are other natural types of random walks that one can consider on Brownian surfaces which one would expect to converge to Brownian motion in the scaling limit.  For example, one can generate a random walk which at each step moves to a point sampled uniformly at random from the metric ball of radius $\epsilon$ centered at its current position.  As a second example, the Brownian snake construction of the Brownian map involves describing the Brownian map as a gluing of the tree of geodesics back to the root together with a dual tree (and instance of the CRT) rooted at the dual root.  The peanocurve which ``snakes between these two trees'' is a space-filling curve, which one may use to give a graph approximation analogous to the mated-CRT map considered in \cite{gms-tutte,gms-harmonic}.  In particular, if $(\mcl X,D,\mu)$ denotes the unit-area Brownian map, the Brownian snake construction gives a quotient map $p : [0,1] \rta \mcl X$. We then fix $\ep > 0$ and consider the random walk on the adjacency graph of $\mu$-mass $\ep$ cells $p([x-\ep,x])$ for $x\in [0,1] \cap (\ep\BB Z)$. 

One of the appeals of this construction is that one can sample from it in linear time (one just needs to generate an instance of the Brownian snake) and then compute its Tutte embedding efficiently using a sparse matrix package (c.f.~\cite[Remark 1.2]{gms-tutte}).

\begin{prob}
\label{prob-other-approximations}
Show that Theorem~\ref{thm-rw-conv0} holds for random walk on other graph approximations of Brownian surfaces, such as the two mentioned just above.	
\end{prob}

\begin{prob}
\label{prob-bm}
Show that the complementary connected components of a Brownian motion on the Brownian map (run for a fixed amount of time) are independent Brownian disks conditional on their boundary length.  
\end{prob}

The analog of the property of Problem~\ref{prob-bm} for Brownian motion on certain random planar maps (like the UIPT or the UIPQ) follows from the so-called spatial Markov property, a.k.a.\ peeling; see, e.g.,~\cite{benjamini-curien-uipq-walk}. 
It is known that the complementary connected components of SLE$_6$ on a Brownian surface are Brownian disks (this follows from the results of~\cite{wedges,sphere-constructions} and the equivalence of Brownian and $\sqrt{8/3}$-LQG surfaces), so it may be possible to solve Problem~\ref{prob-bm} using the relationship between SLE$_6$ and Brownian motion. 
An alternative approach to Problem~\ref{prob-bm} is via Theorem~\ref{thm-rw-conv0}. 
Indeed, we know that the complementary connected components of a metric ball on the Brownian map are Brownian disks conditional on their boundary lengths~\cite{tbm-characterization,legall-disk-snake}. Moreover, a Poisson-Voronoi cell is determined by the metric ball centered at its center point whose radius equals twice the distance from the center point to the boundary of the cell, together with the points of the Poisson point process which intersect this ball. 
It is possible that one could apply this property at the cells hit by the walk to solve Problem~\ref{prob-bm}. 

Problem~\ref{prob-bm} might have some relevance to Problem~\ref{prob-rw-conv}. Indeed, if one can show that Brownian motion on the Brownian map is uniquely characterized by the Markov property of Problem~\ref{prob-bm} together with the law of the boundary lengths of the complementary connected components, then potentially this could be used to identify a subsequential scaling limit of random walk on random planar maps (one would also have to establish tightness). A similar strategy is used to prove the convergence of percolation on random planar maps to SLE$_6$ in~\cite{gwynne-miller-perc}. 

Theorem~\ref{thm-rw-conv0} together with the result of Yadin and Yehudayoff~\cite{yy-lerw} allow us to give an intrinsic definition of SLE$_2$ on a Brownian surface as the limit of the loop-erased random walk on Poisson-Voronoi tessellations.

\begin{prob} \label{prob-sle2}
Does the perspective of this paper lead to any insights about SLE$_2$ on a Brownian surface (concerning, e.g., the law of the surface parameterized by its complement or its relationship to random planar maps)?
\end{prob}

Problem~\ref{prob-sle2} would be very interesting to solve since currently very little is known about the behavior of SLE$_\kappa$ curves on a $\gamma$-LQG surface for $\kappa \notin \{\gamma^2,16/\gamma^2\}$. 

Appendix~\ref{sec-voronoi-cells} includes several basic properties of Voronoi cells which are needed in the proofs of our main results. However, there are many questions about such cells which have not been answered, for example the following.

\begin{prob} \label{prob-voronoi-cells}
Is the boundary of a Voronoi cell a.s.\ given by the union of finitely many disjoint simple curves? 
What is the Hausdorff dimension of this boundary (with respect to the Euclidean or $\sqrt{8/3}$-LQG metric)?  Is the collection of Voronoi cell boundaries a.s.\ conformally removable?
\end{prob}

The simulation in Figure~\ref{fig:pv_sim} seems to suggest that the answer to the first part of Problem~\ref{prob-voronoi-cells} is affirmative.
Although we will not explain this in detail here, we expect that the Hausdorff dimension w.r.t.\ the $\sqrt{8/3}$-LQG metric should be $2$. (Roughly speaking, this is because one expects that on a Brownian surface, the set of points equidistant to generic points $z_1$ and $z_2$ should be a curve that has the same local structure as a branch of the dual of the tree of geodesics drawn toward a fixed root.)  We do not have a conjecture for the Euclidean Hausdorff dimension of the cell boundaries. 

\appendix

\section{Basic properties of Voronoi cells}
\label{sec-voronoi-cells}
 
In this section we will prove a number of a.s.\ properties of Voronoi cells which are intuitively obvious from, e.g., the simulations (see Figure~\ref{fig:pv_sim}).
In particular, we will show that such cells are connected (Lemma~\ref{lem-cell-geodesic}), they are compact and locally finite (Lemma~\ref{lem-cell-basic}), and their boundaries have zero LQG measure and zero Lebesgue measure (Lemmas~\ref{lem-zero-lqg-mass} and~\ref{lem-zero-lebesgue}).
Throughout, we will consider the setup described at the beginning of Section~\ref{sec-tutte-conv}, so that for a GFF-type distribution $h $ and $\lambda  > 0$, $\mcl P_h^\lambda$ and $\mcl H_h^\lambda$ denote the associated Poisson point process and collection of Voronoi cells, respectively. The arguments in this section do not use any of the results of Sections~\ref{sec-tutte-conv} and~\ref{sec-ball-estimates}, so can be read independently of those sections.
 
\begin{remark}
\label{remark-const}
We will often consider the Voronoi tessellations $\mcl H_h^\lambda$ for fields $h$ for which there is a choice in the way that the additive constant is fixed.  Different choices scale $\mu_h$, and hence the intensity measure for $\mcl P_h^\lambda$, by a constant factor. 
However, for any fixed choice of additive constant for $h$, the conditional law of $\mcl P_h^\lambda$ given $h$ is well-defined. 
Furthermore, if $\mu_h$ is a.s.\ finite and we condition on $h$ and the event $\{\#\mcl P_h^\lambda = n\}$ for $n\in\BB N$, then the conditional law of $\mcl P_h^\lambda$ is that of a collection of $n$ i.i.d.\ samples from $\mu_h$, so this conditional law does not depend on the choice of additive constant for $h$ (or on $\lambda$). 
If $\mu_h$ is only locally finite, one can apply the preceding sentence with $h$ replaced by its restriction to a compact set.
As a particular consequence of this, the a.s.\ statements for Voronoi cells which we consider in this section do not depend on the choice of additive constant, so we will not specify it. 
\end{remark}  
 
We start with the following elementary deterministic fact.

\begin{lem} \label{lem-cell-geodesic}
Suppose $h$ is a whole-plane GFF.
For each $z\in \mcl P_h^\lambda$ and each $u$ in the corresponding cell $H_z$ (resp.\ each $u$ in the interior of $H_z $), each $D_h$-geodesic from $z$ to $u$ is contained in $H_z $ (resp.\ the interior of $H_z $).  
In particular, $H_z  $ and its interior are both connected.  
The same is true with $h$ replaced by a free-boundary GFF on a Jordan domain or an embedding of a quantum cone, sphere, disk, or wedge.
\end{lem}
\begin{proof}
Suppose $z\in \mcl P_h^\lambda$ and $ u \in H_z$ (resp.\ $u$ is in the interior of $H_z $). Let $\gamma_{u,z}$ be a $D_h$-geodesic from $u$ to $z$. 
If $\gamma_{u,z}$ were not contained in $H_z $ (resp.\ its interior), then there would be a $t\in [0,D_h(u,z)]$ and a $z' \in \mcl P_h^\lambda$, $z'\not=z$, such that $\gamma_{u,z}(t)$ is strictly (resp.\ weakly) $D_h$-closer to $z'$ than to $z$. This implies that $u$ is strictly (resp.\ weakly) $D_h$-closer to $z'$ than to $z$, which contradicts that $u\in H_z $ (resp.\ $u$ is in the interior of $H_z $). 
\end{proof}

\begin{lem} \label{lem-max-cell-diam}
Suppose $h$ is a whole-plane GFF. For each compact set $K\subset\BB C$,  
\eqb \label{eqn-max-cell-diam}
\sup_{H\in\mcl H_h^\lambda , \: H\cap K\not=\emptyset} \sup_{z,w \in H} D_h(z,w) \rta 0 ,\quad \text{in law as $\lambda\rta\infty$} .
\eqe 
The same is true with $h$ replaced by a free-boundary GFF on a Jordan domain or an embedding of a quantum cone, sphere, disk, or wedge.
In the case of a free-boundary GFF or a quantum sphere or wedge, the compact set $K$ is only required to lie in the closure of the domain for $h$.
\end{lem}
\begin{proof}
For each $\ep > 0$, a.s.\ $K$ can be covered by finitely many $D_h$-balls of radius at most $\ep$ and each of these balls has positive $\mu_h$-mass. 
The conditional probability given $h$ that each such ball contains a point of $\mcl P_h^\lambda$ tends to 1 as $\lambda\rta\infty$. 
On this event, each cell which intersects $K$ has $D_h$-diameter at most $2\ep$. 
\end{proof}

\begin{lem} \label{lem-cell-basic}
Suppose $h$ is a whole-plane GFF and $\lambda > 0$.
Almost surely, each cell in $\mcl H_h^\lambda$ has non-empty interior and is compact.
Furthermore, a.s.\ each compact subset of $\BB C$ intersects only finitely many cells of $\mcl H_h^\lambda$.  
The same is true with $h$ replaced by a free-boundary GFF on a Jordan domain or an embedding of a quantum cone, sphere, disk, or wedge.
\end{lem}

The proof may seem harder than the reader expects. This is because we need to rule out cells which have extremely long (perhaps even infinitely long) ``tentacles" which could make the cell unbounded or cause the cell to intersect a compact set very far from its center point (which may cause difficulties with local finiteness).
This requires some basic control on how ``spread out" the points of $\mcl H_h^\lambda$ can be. 
We will prove much more quantitative estimates for cells in Section~\ref{sec-ball-estimates}.

\begin{proof}[Proof of Lemma~\ref{lem-cell-basic}]
Since $\mu_h$ is a.s.\ a locally finite measure, a.s.\ we can find a $D_h$-ball centered at any given point of $\mcl P_h^\lambda$ which does not contain any other points of $\mcl P_h^\lambda$. Since $D_h$ induces the Euclidean topology on the domain of $h$ a.s.\ each $D_h$-ball contains a Euclidean neighborhood of its center point. 
Hence a.s.\ every cell in $\mcl H_h^\lambda$ has non-empty interior. 
By the continuity of $(z,w) \mapsto D_h(z,w)$ with respect to the Euclidean topology (which follows from the fact that $D_h$ induces the Euclidean topology on its domain), we see that the cells of $\mcl H_h^\lambda$ are a.s.\ closed.

We will now argue that a.s.\ all of the cells of $\mcl H_h^\lambda$ are bounded and that a.s.\ each compact subset of $\BB C$ intersects only finitely many cells of $\mcl H_h^\lambda$ in the case when $h$ is a whole-plane GFF normalized so that its circle average over $\bdy\BB D$ is zero.  
We first claim that a.s.\ for each large enough $k\in\BB N$, there is a point of $\mcl P_h^\lambda$ in $B_{2^k}(0) \setminus B_{2^{k-1}}(0)$. To see this, we first observe that by standard estimates for the $\sqrt{8/3}$-LQG measure~\cite[Lemma 4.6]{shef-kpz}, the LQG coordinate change formula, and the fact that $h(2^{ k}\cdot) - h_{2^{ k}}(0) \eqD h$, it holds except an event of probability decaying faster than any power of $2^{-k}$ that 
\eqbn
\mu_h\left( B_{2^k}(0) \setminus B_{2^{k-1}}(0) \right)  \geq  2^{\left(Q\sqrt{8/3} - \frac{1}{100} \right) k} e^{\sqrt{8/3} h_{2^k}(0)} .
\eqen
Since $t\mapsto h_{e^t}(0)$ evolves as a standard linear Brownian motion~\cite[Section 3.1]{shef-kpz}, by the Borel-Cantelli lemma it follows that a.s.\ $\mu_h\left( B_{2^k}(0) \setminus B_{2^{k-1}}(0) \right)$ grows exponentially in $k$ as $k\rta\infty$. Since $\mcl P_h^\lambda$ is a Poisson point process with intensity measure $\lambda\mu_h$, our claim now follows.

For $k \in \BB N$, let $E_k$ be the event that 
\alb
&\left( \max_{z,w\in B_{2^k}(0) \setminus B_{2^{k-1}}(0)} D_h(z,w) \right) \vee \left( \max_{z,w\in B_{2^{k-2}}(0) \setminus B_{2^{k-3}}(0)} D_h(z,w)    \right)  \notag\\
&\qquad \qquad < D_h\left( \bdy B_{2^{k-1}}(0 )  , \bdy B_{2^{ k-2}}(0) \right) \wedge D_h\left( \bdy B_{2^{k-3}}(0 )  , \bdy B_{2^{ k-4}}(0) \right)
\ale
i.e., the $D_h$-diameters of the two dyadic annuli on either side of $ B_{2^{k-1}}(0 ) \setminus   B_{2^{ k-2}}(0)$ are each strictly smaller than the $D_h$-distance across this annulus and the $D_h$-distance across $\bdy B_{2^{k-3}}(0 ) \setminus B_{2^{ k-4}}(0)$.  
We claim that a.s.\ $E_k$ occurs for arbitrarily large values of $k\in\BB N$. 
It is easily seen from the local absolute continuity properties of the GFF (see, e.g.,~\cite[Lemma 4.2]{gwynne-miller-gluing}) that $\BB P[E_1] > 0$.
By the scale invariance of the law of $h$, viewed modulo additive constant, we see that there is a universal constant $p > 0$ such that $\BB P[E_k]  \geq p$ for each $k \in \BB N$. 
The event $E_k$ is determined by $h|_{\BB C\setminus B_{2^{k-2}}(0)}$, so since the tail $\sigma$-algebra $\bigcap_{r>0} \sigma\left( h|_{\BB C\setminus B_{2^{ k-2}}(0)} \right)$ is trivial (see, e.g.,~\cite[Lemma 2.2]{hs-euclidean}), we obtain our claim.

Combining the preceding two paragraphs shows that a.s.\ there exists arbitrarily large values of $k \in\BB N$ for which $E_k$ occurs and $B_{2^{ k}}(0)\setminus B_{2^{ k-1}}(0)$ and $B_{2^{k-2}}(0)\setminus B_{2^{k-3}}(0)$ each contain a point of $\mcl P_h^\lambda$. If $k$ is one of these values, then each point of $\BB C \setminus B_{2^{ k-1}}(0)$ is $D_h$-closer to one of the points of $\mcl P_h^\lambda$ in $B_{2^k}(0)\setminus B_{2^{k-1}}(0)$ than it is to any point of $B_{2^{k-2}}(0)$, so cannot be contained in a cell centered at a point of $B_{2^{k-2}}(0)$. Hence every cell centered at a point of $B_{2^{k-2}}(0)$ is contained in $B_{2^{k-1}}(0)$, so in particular is bounded. 
Furthermore, each point of $B_{2^{k-2}}(0)$ is $D_h$-closer to a point of $\mcl P_h^\lambda$ which is in $B_{2^{k-2}}(0)\setminus B_{2^{k-3}}(0)$ than it is to any point of $\BB C\setminus B_{2^{k-1}}(0)$. Hence each cell which intersects $B_{2^{k-2}}(0)$ must be centered at a point in $B_{2^{k-2}}(0)$. 
Since $\mu_h(B_{2^{k-2}}(0) ) < \infty$, a.s.\ $B_{2^{k-2}}(0)$ contains at most finitely many points of $\mcl P_h^\lambda$, so $B_{2^{k-3}}(0)$ intersects only finitely many cells of $\mcl H_h^\lambda$. 
Since this happens for arbitrarily large values of $k$, we conclude the proof in the case of the whole-plane GFF. 
 
The case when $h$ is a quantum cone or a quantum wedge can be deduced from the case of a whole-plane GFF using local absolute continuity, or alternatively be treated similarly to the case of a whole-plane GFF using the radii $R_b$ of~\eqref{eqn-mass-hit-time} with $b  = 2^k$ and the relation~\eqref{eqn-cone-scale}. 
\end{proof}

From basic properties of the Brownian disk, we get the following. 

\begin{lem} \label{lem-zero-lqg-mass}
Suppose $h$ is a whole-plane GFF (with some choice of additive constant).  
Almost surely, the boundary of each of the cells in $\mcl H_h^\lambda$ has zero $\sqrt{8/3}$-LQG area measure.  
In fact, it is a.s.\ the case that for any $z,w\in\mcl P_h^\lambda$, the set of $u \in\BB C$ for which $D_h(u,z) = D_h(u,w)  $ has zero $\sqrt{8/3}$-LQG measure. 
The same is true with $h$ replaced by a free-boundary GFF on a Jordan domain or an embedding of a quantum cone, sphere, disk, or wedge.

\end{lem}
\begin{proof}
We will prove the lemma in the case of the quantum disk. The general case follows from this and the local absolute continuity of the fields mentioned in the lemma with respect to an embedding of the quantum disk (this local absolute continuity holds away from the boundary of the domain in the case of the whole-plane GFF or the quantum cone or sphere and up the domain boundary in the case of a free-boundary GFF or quantum wedge). 
Since the quantum disk is equivalent to the Brownian disk, general Brownian disk theory shows that if $v$ is sampled uniformly from $\mu_h$, then a.s.\ for each $r>0$ one has $\mu_h(\bdy B_r(v;D_h))    = 0$: indeed, this follows, e.g., from the fact that the set of times for which the head of the Brownian snake takes any particular value $r>0$ has zero Lebesgue measure.
Since $\mcl P_h^\lambda$ is a Poisson point process with intensity measure $\lambda \mu_h$, this shows that the probability that two points of $\mcl P_h^\lambda$ lie on the boundary of the same $D_h$-ball centered at $v$ is zero. That is, the probability that $D_h(u,v) = D_h(u,w)$ is zero. Since $v$ is sampled uniformly from $\mu_h$, the statement of the lemma follows.
\end{proof}

We next check that cell boundaries have zero Lebesgue measure. 
The basic idea of the proof is as follows. If $\phi$ is a smooth bump function and $a \in \BB R$, then the laws of $h + a\phi$ and $h$ are mutually absolutely continuous.
Furthermore, a certain ``good" generic event for $h+a\phi$ occurs for Lebesgue-a.e.\ choice of $a$. 
By taking $a$ to be random according to a distribution with a density with respect to Lebesgue measure (e.g., sampled from the standard Gaussian distribution) this implies that the desired generic behavior holds with probability 1, which will then allow us to conclude that the cells have zero Lebesgue measure.
Similar arguments can be used to obtain that cells of $\mcl H^\lambda$ a.s.\ have ``generic" behavior in certain senses.
Some results to this affect appeared in an earlier arXiv version of this paper. 

\begin{lem} \label{lem-zero-lebesgue}
Suppose $h$ is a whole-plane GFF (with some choice of additive constant). 
Almost surely, the boundary of each of the cells in $\mcl H_h^\lambda$ has zero Lebesgue measure.
In fact, it is a.s.\ the case that for any $z,w\in\mcl P_h^\lambda$, the set of $u \in\BB C$ for which $D_h(u,z) = D_h(u,w)  $ has zero Lebesgue measure.
The same is true with $h$ replaced by a free-boundary GFF on a Jordan domain or an embedding of a quantum cone, sphere, disk, or wedge.
\end{lem}
\begin{proof}
We will prove the lemma in the case when $h$ is a free-boundary GFF on $\BB D$. 
This implies the lemma in general by local absolute continuity.  
Note that the lemma statement does not depend on the choice of additive constant for $h$ by Remark~\ref{remark-const} and since adding a constant to $h$ scales $D_h$ by a constant factor.  
If we condition on $h$ and the total number of points in $\mcl P_h^\lambda$, then the elements of $\mcl P_h^\lambda$ are i.i.d.\ samples from $\mu_h$, normalized to be a probability measure.
It therefore suffices to show that if $z,w\in \BB D$ are independent samples from $\mu_h$, normalized to be a probability measure, then a.s.\ the set of $u\in\BB D$ with $D_h(u,z) = D_h(u,w)$ has zero Lebesgue measure. 
For this purpose, it suffices to show that for any fixed $u\in\BB D$, we have $\BB P[D_h(u,z) = D_h(u,w)] = 0$. 
\medskip

\noindent\textit{Step 1: reducing to an event with deterministic sets.}
Henceforth fix $u\in\BB D$ and for open sets $V\subset V' \subset \BB D$ with $\ol V  \subset V'$ and $\ol V'\not=\ol{\BB D}$, let $E_h = E_h(u,V,V')$ be the event that the following is true.
\begin{enumerate}
\item $z\in V$ and $D_h(u,z) = D_h(u,w) $. 
\item There is a $D_h$-geodesic from $w$ to $u$ which does not enter $\ol V'$. 
\end{enumerate}
Since $D_h$-geodesics have zero $\mu_h$-mass, the probability that every $D_h$-geodesic from $z$ to $u$ passes through $w$ is zero, and the same is true with $z$ and $w$ interchanged. 
Consequently, on the event $\{ D_h(u,z) =D_h(u,w)\}$ there a.s.\ exists deterministic open sets $V\subset V' \subset \BB D$ with $\ol V  \subset V'$ such that $E_h(u,V,V')$ occurs, and we can take $V$ and $V'$ to be finite unions of Euclidean balls with rational centers and radii. 

It therefore suffices to show that for any fixed deterministic choice of $V$ and $V'$ as above, one has $\BB P[E_h] = 0$. 
Henceforth fix such a deterministic choice of $V$ and $V'$. 
Since adding a constant to $h$ does not effect the occurrence of the event $E_h$, we can assume without loss of generality that the additive constant for $h$ is fixed so that $\mu_h(U) = 0$ for some deterministic open set $U\subset\BB C$ which is disjoint from $\ol V'$. 
\medskip

\noindent\textit{Step 2: randomly perturbing the field.}
Consider a smooth bump function $\phi : \BB D\rta [0,1]$ with $\phi|_{V } \equiv 1$ and $\phi|_{\BB D\setminus V'} \equiv 0$.
For $a\in\BB R$, the laws of the fields $h+a\phi$ and $h$ are mutually absolutely continuous (this follows from, e.g.,~\cite[Lemma 3.4]{ig1} and the fact that $\phi|_U\equiv 0$, so adding $a\phi$ does not affect the choice of normalization for the field).
It follows from this that the joint laws of $(h,z,w)$ and $(h+ a \phi,z,w)$ are mutually absolutely continuous.
We emphasize that this still holds even though we are always sampling $z$ and $w$ from $\mu_h$, rather than from $\mu_{h+a\phi}$, since adding a smooth function to $h$ results in an LQG measure which is absolutely continuous with respect to $\mu_h$. 
 
Let $E_{h+ a \phi}$ be defined in the same manner as above but with $(D_{h+a\phi} ,z,w)$ in place of $(D_h,z,w)$.
Also let $A$ be a standard Gaussian random variable, independent from everything else. 
The laws of $(h + A\phi , z,w)$ and $(h,z,w)$ are mutually absolutely continuous, so it suffices to show that $\BB P[E_{h+A \phi}] =0$.
For this purpose, it is enough to show that if $E_{h+a  \phi}$ occurs for some $a \in \BB R$, then $E_{h+b\phi}$ does not occur for any $b\in\BB R\setminus \{a\}$ (since this implies that $\BB P[E_{h+A\phi} | (h,z,w)  ]= 0$). 

Let us therefore suppose that $E_{h + a\phi}$ occurs and $b\not=a$. 
We seek to show that $E_{h+b\phi}$ does not occur, so we can assume that all of the conditions in the definition of $E_{h+b\phi}$ occur except possibly for the condition that $D_{h+b\phi}(u,z) = D_{h+b\phi}(u,w)$. We seek to show that $D_{h+b\phi}(u,z) \neq D_{h+b\phi}(u,w)$. 
Since there is a $D_{h+a\phi}$-geodesic and a $D_{h+b\phi}$-geodesic from $w$ to $u$ which do not enter $\ol V'$ and $\phi$ is supported on $\ol V'$, we see that $D_{h+a \phi}(u,w) = D_{h+b\phi}(u,w)$.
On the other hand, since $z \in V $ Lemma~\ref{lem-metric-f} implies that $D_{h+a\phi}(u,z) < D_{h+b\phi}(u,z)$ if $b > a$, and one has the reverse inequality if $b < a$. 
This shows that $D_{h+b\phi}(u,z) \neq D_{h+b\phi}(u,w)$, as required. 
\end{proof}

\bibliography{cibiblong,cibib}

\bibliographystyle{hmralphaabbrv}

\end{document}